\documentclass[smallextended]{svjour3}       
\smartqed  
\usepackage{graphicx}
\usepackage{geometry}
\usepackage{amsmath}
\usepackage{amssymb}
\usepackage{graphicx}
\usepackage{color}
\usepackage{ifpdf}
\usepackage{url}
\usepackage{hyperref}
\usepackage{algorithm}
\usepackage{algorithmic}
\usepackage[mmddyyyy]{datetime}
\usepackage{framed}
\usepackage[usenames, dvipsnames, table]{xcolor}
\usepackage{xcolor}
\usepackage{booktabs} 
\usepackage{float} 
\usepackage{multicol,multirow} 
\usepackage{caption} 
\usepackage{colortbl}

\usepackage{lipsum} 

\usepackage[T1]{fontenc} 
\usepackage{lettrine} 

\newtheorem{assumption}{Assumption A.\!}

\makeatletter
\renewcommand\@biblabel[1]{#1.}
\makeatother

\newcommand{\R}{\mathbb{R}}
\newcommand{\Rext}{\mathbb{R}\cup\{+\infty\}}

\newcommand{\abs}[1]{\left\vert#1\right\vert}

\newcommand{\set}[1]{\left\{#1\right\}}
\newcommand{\norm}[1]{\left\Vert#1\right\Vert}

\newcommand{\graph}[1]{\mathrm{gr}\left(#1\right)}
\newcommand{\Eproof}{\hfill $\square$}
\newcommand{\prox}{\mathrm{prox}}
\newcommand{\argmin}{\mathrm{arg}\!\min}
\newcommand{\dom}[1]{\mathrm{dom}(#1)}
\newcommand{\cl}[1]{\mathrm{cl}\!\left(#1\right)}
\newcommand{\bdry}[1]{\mathrm{bdry}\!\left(#1\right)}

\newcommand{\fopt}[1]{f^{\star}}

\newcommand{\xb}{\mathbf{x}}

\newcommand{\yb}{\mathbf{y}}
\newcommand{\zb}{\mathbf{z}}
\newcommand{\wb}{\mathbf{w}}
\newcommand{\ub}{\mathbf{u}}
\newcommand{\vb}{\mathbf{v}}
\newcommand{\ab}{\mathbf{a}}
\newcommand{\bb}{\mathbf{b}}
\newcommand{\cb}{\mathbf{c}}
\newcommand{\db}{\mathbf{d}}

\newcommand{\eb}{\mathbf{e}}
\renewcommand{\sb}{\mathbf{s}}

\newcommand{\Ac}{\mathcal{A}}
\newcommand{\Pc}{\mathcal{P}}

\newcommand{\Qb}{\mathbf{Q}}
\newcommand{\Qc}{\mathcal{Q}}
\newcommand{\Bc}{\mathcal{B}}

\newcommand{\Xc}{\mathcal{X}}
\newcommand{\Yc}{\mathcal{Y}}
\newcommand{\Zc}{\mathcal{Z}}
\newcommand{\Sc}{\mathcal{S}}
\newcommand{\Dc}{\mathcal{D}}
\newcommand{\Lc}{\mathcal{L}}

\newcommand{\Nc}{\mathcal{N}}
\newcommand{\Kc}{\mathcal{K}}

\newcommand{\Zb}{\mathbf{Z}}

\newcommand{\iprods}[1]{\langle #1\rangle}

\newcommand{\intx}[1]{\mathrm{int}\left(#1\right)}

\newcommand{\Id}{\mathbb{I}}

\newcommand{\trace}[1]{\mathrm{trace}\left(#1\right)}
\renewcommand{\vec}[1]{\mathrm{vec}\left(#1\right)}
\newcommand{\mat}[1]{\mathrm{mat}\left(#1\right)}

\newcommand{\zero}{\boldsymbol{0}}
\newcommand{\Zopt}{\mathcal{Z}^{\star}}

\newcommand{\ShuR}[1]{{#1}}
\newcommand{\Shu}[1]{{#1}}
\newcommand{\Tianxiao}[1]{{#1}}

\newcommand{\beforesubsec}{\vspace{-4ex}}
\newcommand{\aftersubsec}{\vspace{-2.5ex}}
\newcommand{\beforesec}{\vspace{-4ex}}
\newcommand{\aftersec}{\vspace{-3ex}}
\newcommand{\beforesubsubsec}{\vspace{-2ex}}
\newcommand{\aftersubsubsec}{\vspace{-2.5ex}}

\begin{document}

\title{Self-concordant inclusions: A unified framework for path-following generalized Newton-type algorithms
}

\titlerunning{Self-concordant inclusions: A unified framework for generalized interior-point methods}

\author{Quoc Tran-Dinh \and Tianxiao Sun \and Shu Lu 
\vspace{-2ex}
}

\authorrunning{Q. Tran-Dinh \and T. Sun \and S. Lu} 

\institute{Q. Tran-Dinh, T. Sun and S. Lu \at
		Department of Statistics and Operations Research\\
	        The University of North Carolina at Chapel Hill (UNC), USA\\
        		\email{quoctd@email.unc.edu}
}

\date{Received: date / Accepted: date}

\maketitle

\vspace{-2ex}
\begin{abstract}
We study a class of monotone inclusions  called ``self-concordant inclusion'' which covers three fundamental convex optimization formulations as special cases.
We develop a new generalized Newton-type framework to solve this inclusion.
Our framework subsumes three schemes: full-step, damped-step and path-following methods as specific instances, while allows one to use inexact computation to form generalized Newton directions.
We prove a local quadratic convergence of both the full-step and damped-step algorithms.
Then, we propose a new two-phase inexact  path-following scheme for solving this monotone inclusion
which possesses an $\mathcal{O}(\sqrt{\nu}\log(1/\varepsilon))$-worst-case iteration-complexity to achieve an $\varepsilon$-solution, where $\nu$ is the barrier parameter and $\varepsilon$ is a desired accuracy.
As byproducts,  we customize our scheme to solve three convex problems: convex-concave saddle-point, nonsmooth constrained convex program, and nonsmooth convex program with linear constraints.
We also provide three numerical examples to illustrate our theory and compare with existing methods.


\end{abstract}

\beforesec
\section{Introduction}\label{sec:intro}
\aftersec
\vspace{1ex}
\subsection{\bf Problem statement}
\aftersubsec
This paper is devoted to studying the following monotone inclusion which covers three important convex optimization templates \cite{Bauschke2011,Facchinei2003,Rockafellar1997}:
\begin{equation}\label{eq:mono_inclusion}
\text{Find $\zb^{\star}\in\R^p$ such that:} ~~0 \in \Ac_{\Zc}(\zb^{\star}) := \Ac(\zb^{\star}) + \Nc_{\Zc}(\zb^{\star}),
\end{equation}
where $\Zc$ is a nonempty, closed and convex set in $\R^p$; $\Ac : \R^p \rightrightarrows2^{\R^p}$ is a multivalued and maximally monotone operator (\textit{cf.}~Definition \ref{de:monotone_opers});
 $\Nc_{\Zc}(\zb)$ is the normal cone of $\Zc$ at $\zb$ given by $\set{\wb\in\R^p \mid \iprods{\wb, \zb - \hat{\zb}} \geq 0, ~\forall \hat{\zb}\in\Zc}$ if $\zb\in\Zc$, and $\emptyset$ otherwise; and ``$:=$'' stands for ``is defined as''.
Throughout this paper, we assume that $\Zc$ is endowed with a ``$\nu$ - self-concordant barrier''  $F$ (\textit{cf.}~Definition \ref{de:self_con_barrier}).
We denote by $\Zopt := \set{\zb^{\star} \mid 0 \in\Ac(\zb^{\star}) + \Nc_{\Zc}(\zb^{\star})}$ the solution set of \eqref{eq:mono_inclusion}.

Without the self-concordance of  $\Zc$, \eqref{eq:mono_inclusion} is  a classical monotone inclusion \cite{Bauschke2011,Rockafellar1997}, and can be reformulated into a multivalued variational inequality problem \cite{Facchinei2003}.
In particular, \eqref{eq:mono_inclusion} covers the optimality  (or  KKT) conditions of unconstrained and constrained  convex programs, and convex-concave saddle-point problems as described in  Subsection~\ref{subsec:motivating_exam}.
Hence, it can be used as a unified tool to study and develop numerical methods for these problems \cite{Bauschke2011,Facchinei2003}.
Methods for solving \eqref{eq:mono_inclusion} and its special instances are well-developed under different structure assumptions imposed on $\Ac$ and $\Zc$ \cite{Bauschke2011,Facchinei2003}.
See Section~\ref{sec:discussion} for a more thorough discussion.

We instead focus on a class of \eqref{eq:mono_inclusion}, where $\Zc$ is equipped with a ``\textbf{self-concordant}'' barrier (\textit{cf.}~Definition \ref{de:self_con_barrier}).
The self-concordance notion was introduced by Nesterov and Nemirovskii \cite{Nesterov2004,Nesterov1997} in the 1990s to develop a unified theory and polynomial time algorithms in interior-point methods for structural convex programming, but has not been well exploited \Shu{in} other classes of optimization methods in both the convex and nonconvex cases.

Our approach in this paper can briefly be described as follows.  Let $\Zc$ be equipped with a $\nu$-self-concordant barrier $F$.
Since $\Nc_{\Zc}(\zb) = \set{\zero^p}$ for any $\zb\in\intx{\Zc}$, the interior of $\Zc$, we can define the following barrier problem associated with \eqref{eq:mono_inclusion}:
\vspace{-0.5ex}
\begin{equation}\label{eq:barrier_mono_inc}
\text{Find $\zb^{\star}_t\in\intx{\Zc}$ such that:} ~~0 \in \Ac_t(\zb^{\star}_t) := t\nabla{F}(\zb^{\star}_t) + \Ac(\zb^{\star}_t),
\vspace{-0.5ex}
\end{equation}
where $t > 0$ is a penalty parameter.
For any $t > 0$, $\Ac_t$ remains a maximally monotone operator. Hence, \eqref{eq:barrier_mono_inc} is a parametric monotone inclusion depending on the parameter $t$.
As we will show in Lemma~\ref{le:existence_sol_m} that the solution \ShuR{$\zb^{\star}_t$} of \eqref{eq:barrier_mono_inc} exists and is unique for any $t > 0$ under mild conditions.
By perturbation theory \Shu{\cite{Dontchev2014,Robinson1980}}, one can show that \ShuR{$\zb^{\star}_t$ is continuous w.r.t. $t > 0$.}
The set $\set{\zb^{\star}_t \mid t > 0}$ containing solutions of \eqref{eq:barrier_mono_inc} for each $t$  generates a trajectory \Shu{called} the central path of \eqref{eq:mono_inclusion}.
Each point $\zb^{\star}_t$ on this path is called a central point.
Our objective is to design efficient numerical methods for solving \eqref{eq:mono_inclusion} from the linearization of \eqref{eq:barrier_mono_inc}.

\beforesubsec
\subsection{\bf Three fundamental convex optimization templates}\label{subsec:motivating_exam}
\aftersubsec
We present three basic  problems in convex optimization covered by \eqref{eq:mono_inclusion} to motivate our work.

\beforesubsubsec
\subsubsection{Constrained convex programs}
\aftersubsubsec
Consider a general constrained convex optimization problem as studied in \cite{TranDinh2013e,TranDinh2015f}:
\vspace{-0.5ex}
\begin{equation}\label{eq:constr_cvx}
g^{\star} := \min_{\xb}\set{ g(\xb) \mid \xb \in\Xc},
\vspace{-0.5ex}
\end{equation}
where $g : \R^n\!\to\!\Rext$ is proper, closed and  convex, and $\Xc$ is a nonempty, closed and convex set in $\R^n$ endowed with a $\nu$-self-concordant barrier $f$ ({\emph{cf.}} Definition \ref{de:self_con_barrier}).
Let $\partial{g}$ be the subdifferential of $g$ (\textit{cf.} Section \ref{sec:preliminary}).
The following optimality condition is necessary and sufficient for $\xb^{\star}\in\R^n$ to be an optimal solution of \eqref{eq:constr_cvx} under a given constraint qualification:
\vspace{-0.5ex}
\begin{equation*}
0 \in\partial{g}(\xb^{\star}) + \Nc_{\Xc}(\xb^{\star}).
\vspace{-0.5ex}
\end{equation*}
By letting $\zb := \xb$, $\Ac := \partial{g}$ and $\Zc := \Xc$, this inclusion exactly has the same form as \eqref{eq:mono_inclusion}.
The barrier problem associated with \eqref{eq:constr_cvx} becomes
\vspace{-0.5ex}
\begin{equation*}
\Bc^{\star}(t) := \min_{\xb\in\R^n}\set{ \Bc(\xb; t) := g(\xb) + t f(\xb) \mid \xb \in\intx{\Xc}},
\vspace{-0.5ex}
\end{equation*}
where $t > 0$ is a penalty parameter.
The optimality condition of this barrier problem is $0 \in t\nabla{f}(\xb^{\star}_t) + \partial{g}(\xb^{\star}_t)$ which is exactly \eqref{eq:barrier_mono_inc} with  $F := f$.

\vspace{-0.75ex}
\beforesubsubsec
\subsubsection{Constrained convex  programs with linear constraints}
\aftersubsubsec
\vspace{-0.5ex}
We are interested in the following constrained convex optimization problem:
\begin{equation}\label{eq:constr_cvx2}
\mathcal{G}^{\star} := \max_{\xb\in\R^n, \sb\in\R^m}\set{ \mathcal{G}(\xb, \sb) := \iprods{\cb, \xb} - g(\sb) \mid L\xb - W\sb = \bb, ~~\xb\in \Kc},
\end{equation}
where $\cb\in\R^n$, $\bb\in\R^p$, $L : \R^n\to\R^p$ and $W : \R^{n_s} \to \R^p$ are  linear operators, \Shu{$g : \R^{n_s}\to\Rext$} is a proper, closed, convex and possibly nonsmooth function, and $\Kc$ is a proper, nonempty, closed, pointed and  convex cone  endowed with a $\nu$-self-concordant logarithmically homogeneous barrier $f$ (\textit{cf.}~Definition \ref{de:self_con_barrier}).
In addition, we assume that \Shu{$n \leq p$}.

The corresponding dual problem of \eqref{eq:constr_cvx2} can be written as follows:
\begin{equation}\label{eq:dual_prob1}
\mathcal{H}^{\ast} := \min_{\yb}\set{ \mathcal{H}(\yb) := g^{\ast}(W^{\ast}\yb) + \iprods{\bb, \yb} \mid L^{\ast}\yb - \cb \in\Kc^{\ast}},
\end{equation}
where $\Kc^{\ast} := \set{\ub\in\R^n \mid \iprods{\xb,\ub} \geq 0,~\forall \xb\in\Kc}$ is the dual cone of $\Kc$, $L^{\ast}$ and $W^{\ast}$ are the adjoint operators of $L$ and $W$, respectively, and $g^{\ast}(\ub) := \sup_{\sb}\set{\iprods{\ub,\sb} - g(\sb) }$ is the conjugate of $g$.
Let $\Yc := \set{\yb\in\R^p \mid L^{\ast}\yb - \cb \in\Kc^{\ast}}$.
Then, the optimality condition of \eqref{eq:dual_prob1} becomes
\vspace{-0.5ex}
\begin{equation*}
0 \in\partial{g^{\ast}}(W^{\ast}\yb^{\star}) + \bb + \Nc_{\Yc}(\yb^{\star}),
\vspace{-0.5ex}
\end{equation*}
which fits the  form of \eqref{eq:mono_inclusion}.
The barrier problem associated with the dual problem~\eqref{eq:dual_prob1} is
\vspace{-0.5ex}
\begin{equation}\label{eq:dual_constr_cvx2}
\min_{\yb\in\R^p}\set{ g^{*}(W^{\ast}\yb) + \iprods{\bb, \yb} + tf^{*}\left(\cb - L^{\ast}\yb\right) },
\vspace{-0.5ex}
\end{equation}
where $f^{*}$ is the Fenchel conjugate of $f$.
If we define $\psi(\cdot) = g^{*}(W^{\ast}(\cdot)) + \iprods{\bb,\cdot}$ and $\varphi(\cdot) := f^{*}\left(\cb - L^{\ast}(\cdot)\right)$ the barrier of $\Yc$, then the optimality condition of \eqref{eq:dual_constr_cvx2} becomes
\vspace{-0.5ex}
\begin{equation}\label{eq:constr_cvx2_opt}
0 \in -tL\left(\nabla{f^{*}}(\cb - L^{\ast}\yb^{\star}_t )\right) + \partial{\psi}(\yb^{\star}_t),
\vspace{-0.5ex}
\end{equation}
which falls into the form \eqref{eq:barrier_mono_inc} with $\zb := \yb$, $F(\cdot) := \varphi(\cdot) = f^{*}(\cb - L^{\ast}(\cdot))$ and $\Ac(\cdot) := \partial{\psi}(\cdot)$.

\beforesubsubsec
\subsubsection{Convex-concave saddle-point problems}
\aftersubsubsec
\vspace{-0.5ex}
Consider the following convex-concave saddle-point problem that covers many applications including signal/image processing and duality theory \cite{Chambolle2011,Combettes2005}:
\vspace{-0.5ex}
\begin{equation}\label{eq:saddle_point_prob}
\Phi^{\star} := \min_{\yb\in\Yc}\big\{ \Phi(\yb) := \psi(\yb) + \max_{\xb\in\Xc}\set{\iprods{\yb, L\xb} - g(\xb) } \big\},
\vspace{-0.5ex}
\end{equation}
where $g : \R^n \to \Rext$ is a proper, closed and convex function; $\psi : \R^m \to \Rext$ is also a proper, closed and convex function; $\Xc$ and $\Yc$ are two nonempty, closed and convex sets in $\R^n$ and $\R^m$, respectively; and $L : \R^{n} \to \R^m$ is a given linear operator.
The optimality condition of  \eqref{eq:saddle_point_prob} for a saddle point $(\xb^{\star}, \yb^{\star})$ is
\vspace{-0.5ex}
\begin{equation}\label{eq:saddle_point_opt_cond}
\left\{\begin{array}{ll}
0 &\in \partial{g}(\xb^{\star}) - L^{\ast}\yb^{\star} + \Nc_{\Xc}(\xb^{\star}), \vspace{0.5ex}\\
0 &\in \partial{\psi}(\yb^{\star}) + L\xb^{\star} + \Nc_{\Yc}(\yb^{\star}).
\end{array}\right.
\vspace{-0.5ex}
\end{equation}
where $\Nc_{\Xc}$ and $\Nc_{\Yc}$ are the normal cone of $\Xc$ and $\Yc$, respectively.
If we define $\zb := (\xb, \yb)$,  $\Zc := \Xc\times\Yc$,
\vspace{-0.5ex}
\begin{equation}\label{eq:pd_map}
\Ac(\zb) := \begin{bmatrix}\partial{g}(\xb) - L^{\ast}\yb\\ \partial{\psi}(\yb) + L\xb\end{bmatrix},~~\text{and}~~\Nc_{\Zc}(\zb) := \Nc_{\Xc}(\xb) \times \Nc_{\Yc}(\yb),
\vspace{-0.5ex}
\end{equation}
then  \eqref{eq:saddle_point_opt_cond} can be cast into the form \eqref{eq:mono_inclusion}.

Let $\Xc$ and $\Yc$ be endowed with a self-concordant barrier $f$ and $\varphi$, respectively.
Then, we can write down the barrier problem of \eqref{eq:saddle_point_prob} as
\vspace{-0.5ex}
\begin{equation*}
\Bc^{\star}(t) := \min_{\yb\in\intx{\Yc} }\Big\{ \Bc(\yb; t) := \psi(\yb) + t\varphi(\yb) + \max_{\xb\in\intx{\Xc}}\set{\iprods{\yb, L\xb} - g(\xb)  -tf(\xb) } \Big\},
\vspace{-0.5ex}
\end{equation*}
where $t > 0$ is  a penalty parameter.
Hence, its optimality condition becomes
\vspace{-0.5ex}
\begin{equation}\label{eq:saddle_opt_cond}
\left\{\begin{array}{ll}
0 &\in t\nabla{f}(\xb^{\star}_t) - L^{\ast}\yb^{\star}_t  + \partial{g}(\xb^{\star}_t) \vspace{1ex}\\
0 &\in t\nabla{\varphi}(\yb^{\star}_t) + L\xb^{\star}_t  + \partial{\psi}(\yb^{\star}_t).
\end{array}\right.
\vspace{-0.5ex}
\end{equation}
If we define $F(\zb) := f(\xb) + \varphi(\yb)$, then  \eqref{eq:saddle_opt_cond} can be written into the form~\eqref{eq:barrier_mono_inc}.

\vspace{-0.5ex}
\beforesubsec
\subsection{\bf Our contributions}
\aftersubsec
We unify the proximal-point and the path-following interior-point schemes to design a joint treatment between these methods for solving the monotone inclusion \eqref{eq:mono_inclusion}.
Our approach is fundamentally different from existing methods, where we use the means of self-concordant barriers of the feasible set $\Zc$ in \eqref{eq:mono_inclusion} to develop generalized Newton-type algorithms.

We propose  a  unified framework that covers three fundamental convex problems as previously described.
We develop three different generalized Newton-type methods for solving \eqref{eq:mono_inclusion}.
Our framework covers the previous work in \cite{TranDinh2013e,TranDinh2015f} for the convex problem \eqref{eq:constr_cvx}  as special cases.
Our approach relies on specific structure of $\Zc$ in \eqref{eq:mono_inclusion} where we can treat \eqref{eq:mono_inclusion} via the linearization of its barrier formulation \eqref{eq:barrier_mono_inc}.
By introducing a new scaled resolvent mapping and generalized proximal Newton decrement, we develop a generalized Newton framework for solving \eqref{eq:mono_inclusion}.
Then, we combine it and a homotopy strategy for the penalty parameter $t$ to obtain a path-following scheme for solving \eqref{eq:mono_inclusion}.
Our approach relates to classical proximal-point and interior-point methods in the literature as discussed in Section~\ref{sec:discussion}.

\vspace{-2ex}
\paragraph{\textbf{Contributions}:} To this end, we can summarize the contributions of this paper as follows:
\vspace{-1ex}
\begin{itemize}
\item [\textrm{(a)}] (\textit{Theory})
We study a class of monotone inclusions, which we call ``self-concordant inclusions'', that provides a unified framework using self-concordant barriers to investigate three fundamental classes of convex optimization problems.
We prove the existence and uniqueness of the central path of \eqref{eq:barrier_mono_inc} under mild assumptions.

\item [\textrm{(b)}] (\textit{Algorithms})
We propose a generalized Newton-type framework for solving \eqref{eq:mono_inclusion}.
This framework covers three  methods: full-step generalized Newton, damped-step generalized Newton, and full-step generalized Newton path-following  schemes.
Our schemes allow one to use inexact computation to form generalized Newton-search directions, and adaptively update the contraction factor for the penalty parameter $t$ associated with $F$.

\item[\textrm{(c)}] (\textit{Convergence theory})
We prove a local quadratic convergence of the first two inexact generalized Newton-methods, and estimate the worst-case iteration-complexity of  the third inexact path-following scheme to achieve an $\varepsilon$-solution, where $\varepsilon$ is the desired accuracy.
Surprisingly, this worst-case complexity is $\mathcal{O}(\sqrt{\nu}\log(1/\epsilon))$ which is the same as in standard path-following method for smooth convex programming \cite{Nesterov2004,Nesterov1994}.

\item[\textrm{(d)}] (\textit{Special instances}) We customize our path-following framework to solve three convex problems: \eqref{eq:constr_cvx}, \eqref{eq:constr_cvx2} and \eqref{eq:saddle_point_prob}, and investigate the overall worst-case iteration-complexity for each method.
In addition, we provide an explicit scheme to recover primal solutions from the duals in the linear constrained case \eqref{eq:constr_cvx2} with rigorous convergence guarantee.
\vspace{-1ex}
\end{itemize}

Let us emphasize the following points of our contributions.
First, using barrier function for the constraint set $\Zc$ in \eqref{eq:mono_inclusion} allows us to handle a wide class of problems where projections onto $\Zc$ is no longer efficient, e.g., $\Zc$ is a general polyhedral, or a hyperbolic cone.
Second, these are second order methods which often achieve high accuracy solutions and have a fast local convergence rate.
This is an advantage when the evaluation of the barrier function values and its derivatives is expensive.
In addition, they are known to be robust to inexact computations and noise.
However, as a compensation, the complexity-per-iteration is often higher than first order methods.
Fortunately, inexact computation allows us to apply iterative methods for computing generalized Newton search directions.
Third, when applied to \eqref{eq:constr_cvx}, \eqref{eq:constr_cvx2} and \eqref{eq:saddle_point_prob}, the efficiency of our algorithms depends on the cost of the scaled proximal operator of $g$ and $\psi$ which is a key component in first order, primal-dual, and splitting methods.
Finally, our framework is sufficiently general and can  be customized to specific classes of structural convex problems such as conic and geometric programming.

\beforesubsec
\subsection{\bf Outline of the paper}
\aftersubsec
The rest of this paper is organized as follows.
In Section \ref{sec:preliminary}, we recall some preliminary results including monotone operators and self-concordance notions \cite{Nesterov1994} \Shu{used} in this paper.
Section \ref{sec:GNMs} presents a unified generalized Newton-type framework that covers three different methods and analyzes their local convergence properties as well as their  worst-case iteration-complexity.
Section \ref{sec:pd_path_following} customizes our path-following framework to solve the convex-concave minimax problem \eqref{eq:saddle_point_prob}, the primal constrained convex problem \eqref{eq:constr_cvx}, and the linear constrained convex problem \eqref{eq:constr_cvx2}.
Section \ref{sec:num_exp} deals with \Shu{specific} applications and illustrates numerically the performance of our algorithm.
For clarity of exposition, technical proofs of the results in the main text are deferred  to the appendix.

\beforesec
\section{Preliminaries: monotonicity, convexity and self-concordance}\label{sec:preliminary}
\aftersec
We recall some preliminary results from classical convex analysis including monotonicity, convexity and self-concordance which will be used in the sequel.

\beforesubsec
\subsection{\bf Basic definitions}
\aftersubsec
Let $\iprods{\ub, \vb}$ or $\ub^{\top}\vb$ denote the inner product, and $\Vert \ub\Vert_2$ denote the Euclidean norm for any $\ub,\vb\in\R^p$.
For a proper, closed and convex function $F : \R^p\to\Rext$,  $\dom{F} := \set{ \zb \in \R^p \mid F(\zb) < + \infty}$ denotes its domain, $\mathrm{Dom}(F) := \cl{\dom{F}}$ denotes the closure of $\dom{F}$, $\partial{F}(\zb) := \big\{ \wb \in\R^p \mid F(\ub) \geq F(\zb) + \iprods{\wb, \ub - \zb}, ~\forall \ub \in\dom{F} \big\}$ denotes its subdifferential at $\zb$ \cite{Rockafellar1970}.
We also use $\mathcal{C}^3(\mathcal{Z})$ to denote the class of three-time continuously differentiable functions from $\mathcal{Z}\subseteq\mathbb{R}^p$ to $\mathbb{R}$.
Given a multivalued operator $\Ac : \R^p\rightrightarrows 2^{\R^p}$, $\dom{\Ac} := \set{\zb \in\R^p \mid \Ac(\zb) \neq\emptyset}$ denotes the domain of $\Ac$, and $\graph{\Ac} := \set{(\zb, \wb) \in\R^p\times \R^p \mid \wb\in\Ac(\zb)}$ denotes the graph of $\Ac$.
$\Sc^p_{+}$ stands for the symmetric positive semidefinite cone of dimension $p$, and $\Sc^p_{++}$ is its interior, i.e., $\Sc_{++}^p = \intx{\Sc_{+}^p}$.
For any $\Qb\in\Sc^p_{++}$, we denote $\Vert \zb\Vert_{\Qb} := \iprods{\Qb\zb, \zb}^{1/2}$ the weighted norm of $\zb$, and $\Vert \zb\Vert^{\ast}_{\Qb} := \iprods{\Qb^{-1}\zb, \zb}^{1/2}$ is its dual norm.

For the three-time continuously differentiable and convex function $F : \R^p\to\R$ defined in \eqref{eq:barrier_mono_inc} such that $\nabla^2{F}(\zb) \succ 0 $ at some $\zb\in\dom{F}$ (i.e., $\nabla^2{F}(\zb)$ is symmetric positive definite), we define the local norm, and its dual norm, respectively as
\begin{equation}\label{eq:local_norm}
\norm{\ub}_{\zb} := \iprods{\nabla^2{F}(\zb)\ub, \ub}^{1/2}, ~~~\text{and}~~~ \norm{\vb}_{\zb}^{*} := \iprods{\nabla^2{F}(\zb)^{-1}\vb, \vb}^{1/2},
\end{equation}
for given $\ub, \vb\in\R^p$.
Clearly, with this definition, the well-known Cauchy-Schwarz inequality $\iprods{\ub, \vb} \leq \norm{\ub}_{\zb}\norm{\vb}^{*}_{\zb}$ holds.

\beforesubsec
\subsection{\bf Maximally monotone, resolvent and proximal operators}
\aftersubsec
\begin{definition}\label{de:monotone_opers}
\textit{
Given a multivalued operator $\Ac :\R^p\rightrightarrows 2^{\R^p}$, we say that $\Ac$ is monotone if for any $\zb,\hat{\zb}\in\dom{\Ac}$, $\iprods{\wb - \hat{\wb}, \zb - \hat{\zb}} \geq 0$ for $\wb\in\Ac(\zb)$ and $\hat{\wb}\in\Ac(\hat{\zb})$; and $\Ac$ is maximal if its graph is not properly contained  in  the graph of any other monotone operator.
}
\end{definition}

Given a  maximally monotone operator $\Ac : \R^p\rightrightarrows 2^{\R^p}$, and $\Qb\in\Sc^p_{++}$, we define
\begin{equation}\label{eq:resolvent}
J_{\Qb^{-1}\Ac}(\zb) = (\Id + \Qb^{-1}\Ac)^{-1}(\zb) := \set{ \wb \in\R^p \mid 0 \in  \Qb(\wb - \zb) +  \Ac(\wb) },
\end{equation}
the scaled resolvent operator of $\Ac$ \cite{Bauschke2011,Rockafellar1997}.
It is well-known that  $\textrm{dom}(J_{\Qb^{-1}\Ac}) = \R^p$ and $J_{\Qb^{-1}\Ac}$ is well-defined and single-valued.
If $\Qb = \Id$, the identity matrix, then $J_{\Id^{-1}\Ac} \equiv J_{\Ac}$ is the standard resolvent of $\Ac$.
When $\Ac = \partial{g}$, the subdifferential of a proper, closed and convex function  $g$, $J_{\Qb^{-1}\Ac}$ becomes a scaled proximal operator of $g$, which is defined as follows:
\begin{equation}\label{eq:prox_oper}
\textrm{prox}_{\Qb^{-1}g}(\xb) := \argmin_{\ub}\set{ g(\ub) + (1/2)\Vert \ub - \xb \Vert_{\Qb}^2 \mid \ub \in \dom{g} }.
\end{equation}
Methods for evaluating $\textrm{prox}_{\Qb^{-1}g}$ have been discussed in the literature, see, e.g., \cite{Becker2012a,friedlander2016efficient}.
If $\Qb = \Id$, then $\prox_{\Qb^{-1}g} = \prox_g$ the standard proximal operator of $g$.
Examples of such functions can be found, e.g., in \cite{Bauschke2011,Combettes2011,Parikh2013}.

\beforesubsec
\subsection{\bf Self-concordant functions and self-concordant barriers}\label{subsec:self_concordance}
\aftersubsec
We also use the self-concordance concept introduced by Nesterov and Nemirovskii \cite{Nesterov2004,Nesterov1994}.

\begin{definition}\label{de:concordant}
\textit{
A univariate convex function $\varphi \in \mathcal{C}^3(\dom{\varphi})$ is called \emph{standard self-concordant} if
$\abs{\varphi'''(\tau)} \leq 2\varphi''(\tau)^{3/2}$  for all $\tau\in\dom{\varphi}$, where $\dom{\varphi}$ is an open set in $\R$.
A function $F : \dom{F}\subseteq \R^{p} \to \R$ is standard self-concordant if for any $\zb\in\dom{F}$ and $\vb\in\R^p$, the univariate function $\varphi$ defined by $\tau \mapsto \varphi(\tau) := F(\zb + \tau\vb)$ is standard self-concordant.
}
\end{definition}

\begin{definition}\label{de:self_con_barrier}
\textit{
A standard self-concordant function $F : \Zc\subset\R^p\to\R$ is a \emph{$\nu$-self-concordant barrier} for a convex set $\Zc$ with parameter $\nu > 0$ if $\dom{F} = \intx{\Zc}$ and
\begin{equation*}
\sup_{\ub \in\R^p} \left \{2\iprods{\nabla{F}(\zb), \ub} - \Vert \ub\Vert_{\zb}^2\right \} \leq \nu, ~~\forall \zb \in\dom{F}.
\end{equation*}
In addition, $F(\zb)$ \Shu{shall tend to $+\infty$ as $\zb$ approaches} the boundary of $\Zc$.
A function $F$ is called a $\nu$-self-concordant \textit{logarithmically homogeneous  barrier} function of $\Zc$ if $F(\tau \zb) = F(\zb) - \nu\log(\tau)$ for all $\zb\in\intx{\Zc}$ and $\tau > 0$.
}
\end{definition}
\vspace{-1ex}

Several simple sets are equipped with a self-concordant logarithmically homogeneous  barrier.
For instance, $F_{\R^p_{+}}(\zb) := -\sum_{i=1}^p\log(\zb_i)$ is a $p$-self-concordant barrier of $\R^p_{+}$, $F_{\Sc_{+}^n}(\Zb) := -\log\det(\Zb)$ is an $n$-self-concordant barrier of  $\Sc^n_{+}$, and $F(\zb, t) = -\log(t^2 - \norm{\zb}_2^2)$ is a $2$-self-concordant barrier of the Lorentz cone $\mathcal{L}_{p+1} := \set{(\zb, t) \in\R^p\times\R_{+} \mid \norm{\zb}_2 \leq t}$.

When $\Zc$ is bounded and $F$ is a $\nu$-self-concordant barrier for $\Zc$, the analytical center $\bar{\zb}^{\star}_f$ of $f$ exists and is unique. It is defined by
\begin{equation}\label{eq:analytical_center}
\bar{\zb}^{\star}_F := \argmin\set{ F(\zb) \mid \zb\in\intx{\Zc} },~~~~(\text{and its optimality condition is}~\nabla{F}(\bar{\zb}^{\star}_F) = 0).
\end{equation}
Let us define $\kappa := \nu+2\sqrt{\nu}$ for a general self-concordant barrier, and $\kappa := 1$ for a self-concordant logarithmically homogeneous barrier.
Then, we have  $\Vert \vb \Vert_{\zb}^{*} \leq \kappa\Vert \vb \Vert_{\bar{\zb}_F^{\star}}^{*} $ for any $\zb\in\mathrm{int}(\Zc)$ and $\vb\in\R^p$.

Let $\Kc$ be a proper, closed and pointed convex cone.
If $\Kc$ is endowed with a $\nu$-self-concordant logarithmically homogeneous  barrier function $F$, then its Fenchel conjugate (also called Legendre transformation \cite{Nesterov1994})
\begin{equation*}
F^{*}(\wb) := \sup_{\zb}\set{ \iprods{\wb, \zb} - F(\zb) \mid \zb\in\Kc}.
\end{equation*}
is also a $\nu$-self-concordant logarithmically homogeneous  barrier of the anti-dual cone $-\Kc^{*}$ of $\Kc$.
For instance, if $\Kc = \Sc_{+}^n$, then $\Kc^{*} = S_{+}^n = \Kc$ (self-dual cone).
A barrier function of $\Sc_{+}^n$ is $F(\zb) := -\log\det(\zb)$. Hence, $F^{*}(\wb) = -n - \log\det(-\wb)$ is a barrier function of $-\Kc^{*}$.

\beforesec
\section{Generalized Newton-type methods for self-concordant inclusions}\label{sec:GNMs}
\aftersec
\vspace{0.5ex}
We propose a novel generalized Newton-type scheme for solving \eqref{eq:mono_inclusion}.
Then, we develop three inexact generalized Newton-type schemes: full-step,  damped-step  and path-following algorithms based on the linearization of  \eqref{eq:barrier_mono_inc}.
We provide a unified analysis for convergence.

\beforesubsec
\subsection{\bf Fundamental assumptions and fixed-point characterization}
\aftersubsec
Throughout this paper, we rely on the following fundamental, but standard assumption.
%
\vspace{-1ex}
\begin{assumption}\label{as:A0}
\begin{itemize}
\item[$\mathrm{(a)}$] The feasible set $\Zc$ is nonempty, closed and convex, and is equipped with a $\nu$-self-concordant barrier $F$.
\item[$\mathrm{(b)}$] The operator $\Ac$ is maximally monotone, $\intx{\Zc}\cap\dom{\Ac}\neq\emptyset$, and $\dom{\Ac}$ is either an open set or a closed set.
\item[$\mathrm{(c)}$] The solution set $\Zopt$ of \eqref{eq:mono_inclusion} is nonempty.
\end{itemize}
\vspace{-1ex}
\end{assumption}
We note that since $\dom{\nabla{F}} = \intx{\Zc}$, Assumption~A.\ref{as:A0} is sufficient for $\Ac_t$ defined by \eqref{eq:barrier_mono_inc} to be maximally monotone \cite[Corollary 25.5]{Bauschke2011}.
This assumption can be relaxed to different conditions as discussed in \cite[Section 25.1]{Bauschke2011}, which we omit here.

Our aim is to compute an approximate solution of \eqref{eq:mono_inclusion} up to a given accuracy as follows:

\vspace{-1ex}
\begin{definition}\label{de:MVIP_approx_sol}
\textit{
Given  $\varepsilon \geq 0$, we say that $\tilde{\zb}^{\star}_{\varepsilon}\in\intx{\Zc}$ is an $\varepsilon$-solution to \eqref{eq:mono_inclusion}  if
\begin{equation*}
\mathrm{dist}_{\tilde{\zb}^{\star}_{\varepsilon}}\big(\boldsymbol{0}, \Ac(\tilde{\zb}^{\star}_{\varepsilon}) \big) := \min_{\eb}\set{\Vert\eb\Vert_{\tilde{\zb}^{\star}_{\varepsilon}}^{\ast} \mid \eb \in\Ac(\tilde{\zb}^{\star}_{\varepsilon})} \leq \varepsilon.
\end{equation*}
}
\end{definition}
Here, $\mathrm{dist}_{\zb}(\wb,\Omega)$ defines a weighted distance from $\wb\in\R^p$ to a nonempty, closed and convex set $\Omega$ in $\R^p$.
Since $\tilde{\zb}^{\star}_{\varepsilon}\in\intx{\Zc}$, we have $\Nc_{\Zc}(\tilde{\zb}^{\star}_{\varepsilon}) = \set{\boldsymbol{0}}$. Hence, $\Ac_{\Zc}(\tilde{\zb}^{\star}_{\varepsilon}) \equiv \Ac(\tilde{\zb}^{\star}_{\varepsilon})$.
We can modify Definition~\ref{de:MVIP_approx_sol} as  $\mathrm{dist}_{\zb}\big(\boldsymbol{0}, \Ac_{\Zc}(\tilde{\zb}^{\star}_{\varepsilon}) \big) \leq \varepsilon$, where $\zb\in\intx{\Zc}$ is fixed a priori. Then, all the results in the next sections remain preserved but require a slight justification.
In the sequel, we develop different numerical methods to generate a sequence $\set{\zb^k}$ from the interior of $\Zc$.

\paragraph{\bf The scaled resolvent operator of $\Ac$: }
\aftersubsubsec
Let us fix $\hat{\zb} \in\intx{\Zc}$ and $t > 0$. Then, we have $\nabla^2{F}(\hat{\zb})\in\Sc^p_{++}$.
For simplicity of presentation, using \eqref{eq:resolvent} we denote by
\begin{equation}\label{eq:Pc_oper}
\Pc_{\hat{\zb}}(\cdot; t) := J_{(t\nabla^2{F}(\hat{\zb}))^{-1}\Ac}(\cdot) = \left(\Id + t^{-1}\nabla^2{F}(\hat{\zb})^{-1}\Ac\right)^{-1}(\cdot),
\end{equation}
the scaled resolvent of $\Ac$.
Using $\Pc_{\hat{\zb}}(\cdot; t)$, we can formulate the monotone inclusion \eqref{eq:barrier_mono_inc} as a fixed-point equation
\begin{equation}\label{eq:fixed_point_expression}
\zb_t^{\star} = \Pc_{\hat{\zb}}\left(\zb^{\star}_t - \nabla^2{F}(\hat{\zb})^{-1}\nabla{F}(\zb^{\star}_t); t \right).
\end{equation}
Clearly, if we define $R_{\hat{\zb}}(\cdot) := \Pc_{\hat{\zb}}(\cdot - \nabla^2{F}(\hat{\zb})^{-1}\nabla{F}(\cdot); t)$, then $\zb^{\star}_t$ is a fixed-point of $R_{\hat{\zb}}(\cdot)$.

\paragraph{\bf The existence of the central path: }
\aftersubsubsec
We prove in Appendix~\ref{apdx:le:existence_sol_m} the following existence result for \eqref{eq:barrier_mono_inc}.
Let us recall that the horizon cone of a convex set $C$ consists of vectors $\omega$ such that $\zb+\tau \omega \in \cl{C}$ for any $\zb\in C$ and any $\tau>0$, where $\cl{C}$ stands for the closure of $C$.

\begin{lemma}\label{le:existence_sol_m}
Suppose that for any nonzero $\boldsymbol{\omega}$ in the horizon cone of $\intx{\Zc}\cap\dom{\Ac}$, there exists some $\hat{\zb}\in \intx{\Zc}\cap\dom{\Ac}$ with $\hat{\ab}\in \Ac(\hat{\zb})$ such that $\iprods{\hat{\ab}, \boldsymbol{\omega}} > 0$. Then, for each $t > 0$, problem \eqref{eq:barrier_mono_inc} has a unique solution.
Moreover, we have $\mathrm{dist}_{\zb^{\star}_t}\big(\boldsymbol{0}, \Ac(\zb^{\star}_t) \big) \leq t\sqrt{\nu}$, which shows that $\zb^{\star}_t$ is an $\varepsilon$-solution to \eqref{eq:mono_inclusion} in the sense of Definition~\ref{de:MVIP_approx_sol} if $t \leq \frac{\varepsilon}{\sqrt{\nu}}$.
\end{lemma}

The assumption in Lemma \ref{le:existence_sol_m} is quite general. There are two special cases in which this assumption holds.
First, if \ShuR{$\intx{\Zc}\cap\dom{\Ac}$} is bounded, then the only element in the horizon cone of \ShuR{$\intx{\Zc}\cap\dom{\Ac}$} is $\boldsymbol{0}$ and the assumption trivially holds.
Second, if the solution set $\Zopt$ of \eqref{eq:mono_inclusion} is nonempty and bounded, and \ShuR{the set-valued map $\Ac$ is continuous at points in $\bdry{\Zc \cap \dom{\Ac}}$ relative to $\Zc\cap \dom{\Ac}$, then this assumption also holds as can be shown using \cite[Theorem 12.51]{Rockafellar1997}. Here, $\bdry{\Zc}$ stands for the boundary of $\Zc$, and we refer to \cite[Definition 5.4]{Rockafellar1997} for the definition of the continuity of a set-valued map.}

\paragraph{\bf Generalized gradient mapping:}
\aftersubsubsec
Fix $\hat{\zb}\in\intx{\Zc}$ with $\nabla^2{F}(\hat{\zb}) \succ 0$, we consider the following linear monotone inclusion in $\sb$:
\begin{equation}\label{eq:s_mapping_inc}
0 \in t\nabla{F}(\zb) + t\nabla^2{F}(\hat{\zb})(\sb  - \zb) + \Ac(\sb).
\end{equation}
If we take $\hat{\zb} = \zb$, then it becomes a linearization (with respect to $\nabla{F}$) of \eqref{eq:barrier_mono_inc} at a given point $\zb$.
It is obvious that \eqref{eq:s_mapping_inc} is strongly monotone and maximally, its solution exists and is unique.
We denote this solution by $\sb_{\hat{\zb}}(\zb; t)$, and, by using $\Pc_{\hat{\zb}}(\cdot; t)$, it can be written as
\begin{equation}\label{eq:s_mapping}
\sb_{\hat{\zb}}(\zb; t) := \Pc_{\hat{\zb}}\left(\zb - \nabla^2{F}(\hat{\zb})^{-1}\nabla{F}(\zb); t \right).
\end{equation}
Next, we define the following mapping
\begin{equation}\label{eq:gradient_mapping}
G_{\hat{\zb}}(\zb; t) := \nabla^2{F}(\hat{\zb})\left(\zb - \sb_{\hat{\zb}}(\zb; t)\right) \equiv \nabla^2{F}(\hat{\zb})\left(\zb - \Pc_{\hat{\zb}}\left(\zb - \nabla^2{F}(\hat{\zb})^{-1}\nabla{F}(\zb); t \right)\right),
\end{equation}
When $\Ac = 0$, $G_{\hat{\zb}}(\zb; t) = \nabla{F}(\zb)$, which is exactly the gradient of $F$.
Then, we adopt the name in \cite{Nesterov2004} to call $G_{\hat{\zb}}(\cdot; t)$ a generalized gradient mapping.

Given $G_{\zb}(\zb; t)$ as in \eqref{eq:gradient_mapping} with $\hat{\zb} = \zb$, we define the following generalized Newton decrement \Shu{$\lambda_t(\zb)$} to analyze the convergence of generalized Newton-type methods below:\begin{equation}\label{eq:nt_decrement}
\lambda_t(\zb) := \Vert G_{\zb}(\zb; t) \Vert_{\zb}^{\ast} =  \Vert \zb - \Pc_{\zb}\left(\zb - \nabla^2{F}(\zb)^{-1}\nabla{F}(\zb); t\right)\Vert_{\zb}.
\end{equation}
If $\Ac(\zb) = \cb$, a constant operator, then $\lambda_t(\zb) = \Vert t^{-1}\cb + \nabla{F}(\zb)\Vert_{\zb}^{\ast}$, which is exactly the Newton decrement defined in \cite[Formula 4.2.16]{Nesterov2004}.

To conclude, we summarize the result of this subsection in the following lemma.
This result is a direct consequence of the definition of $G_{\zb}(\cdot; t)$ and $\lambda_t(\cdot)$. We omit the proof.
\begin{lemma}\label{le:gradient_mapping}
The solution $\sb_{\hat{\zb}}(\cdot; t)$ of \eqref{eq:s_mapping_inc} exists and is unique for any $\zb\in\dom{F}$.
Consequently, $G_{\hat{\zb}}(\cdot; t)$ given by \eqref{eq:gradient_mapping} is well-defined on $\dom{F}$.

Let $\zb^{\star}_t\in \Shu{\intx{\Zc}}$ be a given point and $\lambda_t(\cdot)$ be defined by \eqref{eq:nt_decrement}.
Then, $\lambda_t(\zb^{\star}_t) = 0$ if and only if $\zb^{\star}_t$ is a solution to \eqref{eq:barrier_mono_inc}.
\end{lemma}
In the sequel, we only work with the solution $\sb_{\hat{\zb}}(\zb; t)$ of \eqref{eq:s_mapping_inc} which exists and is unique.
However, we assume throughout this paper that the assumptions of Lemma~\ref{le:existence_sol_m} hold so that the solution $\zb^{\star}_t$ of \eqref{eq:barrier_mono_inc} exists and is unique for each $t > 0$.
We do not use $\zb^{\star}_t$ of \eqref{eq:barrier_mono_inc} at any step of our algorithms.
Since $\zb^{\star}_t$ is on the central path of \eqref{eq:mono_inclusion} at $t > 0$, $\zb^{\star}_t\in\Zc$.
If $t >0 $ is sufficiently small, e.g., $t := \varepsilon/\sqrt{\nu}$, then we can say that $\zb^{\star}_t$ is also an $\varepsilon$-solution of \eqref{eq:mono_inclusion} as stated in Lemma~\ref{le:existence_sol_m} in the sense of Definition~\ref{de:MVIP_approx_sol}.

\beforesubsec
\subsection{\bf Inexact generalized Newton-type schemes}
\aftersubsec
The main step  of the generalized Newton method is presented as follows:
For a fixed value $t > 0$, and a given iterate $\zb\in\intx{\Zc}$, we approximate $F$ by its Taylor's expansion and define
\begin{equation}\label{eq:quad_surrogate}
\widehat{\Ac}_t(\wb; \zb) := t \left[ \nabla{F}(\zb) + \nabla^2{F}(\zb)(\wb - \zb)\right] + \Ac(\wb).
\end{equation}
Since $\nabla^2{F}(\zb) \succ 0$, we can compute the unique solution of the linearized inclusion:
\begin{equation}\label{eq:zero_k_set}
\sb_{\zb}(\zb;t) := \big\{\wb\in \intx{\Zc} \mid 0 \in \widehat{\Ac}_t(\wb; \zb) \big\} \equiv (\widehat{\Ac}_t(\cdot; \zb))^{-1}(\boldsymbol{0}).
\end{equation}
Computing $\sb_{\zb}(\zb; t)$ exactly is often impractical, \Shu{so we allow} one to approximate it as follows. 

\begin{definition}\label{de:approx_sol}
\textit{
Given an accuracy $\delta\in [0, 1)$, we say that $\zb_{+}$ is a $\delta$-approximation to the true solution $\bar{\zb}_{+} := \sb_{\zb}(\zb; t)$ defined in \eqref{eq:zero_k_set}  $($and is denoted by $\zb_{+} \approx \bar{\zb}_{+}$$)$ if
\vspace{-0.75ex}
\begin{equation}\label{eq:approx_sol}
\mathrm{dist}_{\zb}\big(\boldsymbol{0}, \Shu{\widehat{\Ac}_t(\zb_{+}; \zb)} \big) = \min_{\eb}\set{ \Vert\eb\Vert_{\zb}^{\ast} \mid \eb \in\widehat{\Ac}_t(\zb_{+};\zb) } \leq t\delta.
\vspace{-0.75ex}
\end{equation}
}
\end{definition}
First, we show that, under \eqref{eq:approx_sol}, we have $\Vert \zb_{+} - \bar{\zb}_{+} \Vert_{\zb} \leq \delta$.
Next, since we are working with the linearization \eqref{eq:zero_k_set} of \eqref{eq:barrier_mono_inc},
the following lemma, whose proof is in Appendix~\ref{apdx:le:sol_measure}, shows that an approximate solution of  \eqref{eq:zero_k_set} is also an approximate solution of problem \eqref{eq:mono_inclusion}.

\begin{lemma}\label{le:sol_measure}
\Shu{Let} $\zb_{+}$ be a $\delta$-approximation solution \Shu{to $\bar{\zb}_{+}$} of \eqref{eq:zero_k_set} in the sense of Definition~\ref{de:approx_sol}.
Then, we have $\Vert \zb_{+} - \bar{\zb}_{+} \Vert_{\zb} \leq \delta$.
Furthermore, if $\lambda_t(\zb) + \delta < 1$, then
\vspace{-0.75ex}
\begin{equation}\label{eq:sol_measure}
\mathrm{dist}_{\zb_{+}}(\boldsymbol{0}, \Ac_{\Zc}(\zb_{+})) \leq (1- \lambda_t(\zb) - \delta)^{-1}\left(\sqrt{\nu} + \lambda_t(\zb) + 2\delta \right)t.
\vspace{-0.75ex}
\end{equation}
If we choose $t > 0$ such that $t  \leq (1-\lambda_t(\zb) - \delta)\left(\sqrt{\nu} + \lambda_t(\zb) + 2\delta \right)^{-1}\varepsilon$, then $\zb_{+}$ is an $\varepsilon$-solution to \eqref{eq:mono_inclusion} in the sense of Definition~\ref{de:MVIP_approx_sol}.
\end{lemma}
We now investigate the convergence of the inexact full-step, damped-step, and path-following generalized Newton methods.

\beforesubsubsec
\subsubsection{\Shu{A key} estimate}
\aftersubsubsec
The following theorem provides a key estimate to analyze the convergence of the generalized Newton-type schemes above, whose proof can be found in Appendix \ref{apdx:th:main_estimate1}.

\begin{theorem}\label{th:main_estimate1}
For a given $\zb\in\Shu{\intx{\Zc}}$, let $\zb_{+}$ be the point generated by the inexact generalized Newton scheme $($in the sense of Definition \ref{de:approx_sol}$)$:
\begin{equation}\label{eq:full_step_Newton}
\zb_{+} \approx \bar{\zb}_{+} := \Pc_{\zb}\left(\zb  - \nabla^2{F}(\zb)^{-1}\nabla{F}(\zb); t_{+}\right).
\end{equation}
Then, if $\lambda_{t_{+}}(\zb) + \delta(\zb) < 1$, where $\lambda_{t_{+}}(\zb)$ is defined by \eqref{eq:nt_decrement} and $\delta(\zb) := \Vert \zb_{+} - \bar{\zb}_{+}\Vert_{\zb}$, then
$\zb_{+} \in \intx{\Zc}$ and the following estimate holds:
\begin{equation}\label{eq:key_estimate1}
\lambda_{t_{+}}(\zb_{+}) \leq  \left(\frac{\lambda_{t_{+}}(\zb) + \delta(\zb)}{1 - \lambda_{t_{+}}(\zb) - \delta(\zb)}\right)^2 + \frac{\delta(\zb)}{(1-\lambda_{t_{+}}(\zb) - \delta(\zb))^3}.
\end{equation}
Moreover, the right-hand side of \eqref{eq:key_estimate1} is monotonically increasing w.r.t. $\lambda_{t_{+}}(\zb)$ and $\delta(\zb)$.
\end{theorem}
Clearly, if $\zb_{+} = \bar{\zb}_{+} $ (i.e., the subproblem \eqref{eq:full_step_Newton} is solved exactly), then \eqref{eq:key_estimate1} reduces to
\begin{equation}\label{eq:key_estimate1b}
\lambda_{t_{+}}(\zb_{+}) \leq  \left(\frac{\lambda_{t_{+}}(\zb)}{1 - \lambda_{t_{+}}(\zb)}\right)^2,
\end{equation}
which is in the form of \cite[Theorem 4.1.14]{Nesterov2004}, but for the exact variant of \eqref{eq:full_step_Newton}.

\beforesubsubsec
\subsubsection{Neighborhood of the central path and quadratic convergence region}
\aftersubsubsec
Given the generalized Newton decrement $\lambda_t(\cdot)$ defined by \eqref{eq:nt_decrement}, we  consider the following set
\begin{equation}\label{eq:quad_converg_region}
\Omega_{t}(\beta) := \set{\zb\in\intx{\Zc} \mid \lambda_t(\zb) \leq \beta},
\end{equation}
where $\beta \in (0, 1)$.
We call $\Omega_{t}(\beta)$ a neighborhood  of the central path of \eqref{eq:barrier_mono_inc} with the radius $\beta$.

If we can choose $\beta \in (0, 1)$ such that:
\begin{itemize}
\vspace{-0.5ex}
\item[(i)] the sequence $\set{\zb^k}$ generated by a generalized Newton scheme starting from $\zb^0\in\Omega_t(\beta)$ belongs to $\Omega_{t}(\beta)$, and
\item[(ii)] the corresponding sequence of the generalized Newton decrements $\set{\lambda_t(\zb^k)}$ converges quadratically to zero,
\vspace{-0.5ex}
\end{itemize}
then we call $\Omega_t(\beta)$ a quadratic convergence region of this  method, and denote it by $\Qc_t(\beta)$.

Next, we propose two inexact generalized Newton schemes: full-step and damped step, to generate sequence $\set{\zb^k}$ starting from $\zb^0\in\Qc_{t}(\beta)$ for some predefined $\beta\in (0, 1)$, and show that $\set{\lambda_t(\zb^k)}$ converges quadratically to zero.
In these schemes, the penalty parameter $t$ is fixed at a sufficiently small value \textit{a priori}, which may cause some difficulty for computing $\zb^{k+1}$ from $\zb^k$ due to the ill-condition of $\nabla^2{F}(\zb^k)$.
To avoid this situation, we then suggest to use a path-following scheme to gradually decrease $t$ starting from a larger value $t = t_0 > 0$.

\beforesubsubsec
\subsubsection{Inexact full-step generalized Newton method \eqref{eq:iFsGNM}: Local convergence}
\aftersubsubsec
We investigate the convergence of the \ref{eq:iFsGNM} and maximize the radius of its quadratic convergence region $\Qc_{t}(\beta)$.
The following theorem shows a quadratic convergence of the inexact generalized Newton scheme, whose proof is deferred to Appendix \ref{apdx:th:generalized_newton_convergence}.

\begin{theorem}\label{th:generalized_newton_convergence}
Given a fixed parameter $t > 0$, let $\set{\zb^k}$ be a sequence generated by the following inexact full-step generalized Newton scheme \eqref{eq:iFsGNM}:
\begin{equation}\label{eq:iFsGNM}
\zb^{k+1} \approx \bar{\zb}^{k+1} := \Pc_{\zb^k}\left(\zb^k  - \nabla^2{F}(\zb^k)^{-1}\nabla{F}(\zb^k); t\right). \tag{$\mathrm{FGN}$}
\end{equation}
where the approximation $\approx$ is in the sense of Definition \ref{de:approx_sol}.
Then, we have three statements:
\begin{itemize}
\item[$\mathrm{(a)}$]
Let $0 < \beta < \frac{1}{2}(3-\sqrt{5})$ be a given radius, and $\Omega_{t}(\beta)$ be defined by \eqref{eq:quad_converg_region}.
If we choose $\zb^0\in \Omega_{t}(\beta)$ and the tolerance $\delta_k$ in Definition \ref{de:approx_sol} such that
\begin{equation*}
\Vert\zb^{k+1} - \bar{\zb}^{k+1}\Vert_{\zb^{k}} \leq \delta_k \leq \bar{\delta}_k(\beta) := \frac{\beta(1 - 3\beta + \beta^2)(1-\beta)^4}{2\beta^3 - 5\beta^2 + 3\beta + 1},
\end{equation*}
then $\set{\zb^k}$ generated by \ref{eq:iFsGNM} belongs to $\Omega_{t}(\beta)$.

\item[$\mathrm{(b)}$]
If we choose $\delta_k \leq \frac{\lambda_{t}(\zb^k)^2}{1-\lambda_{t}(\zb^k)}$, then, for $k\geq 0$ and $\lambda_t(\zb^0) < 1$, we have
\begin{equation}\label{eq:iFsGNM_est}
\lambda_{t}(\zb^{k\!+\!1})  \leq \left(\frac{2-4\lambda_t(\zb^k) + \lambda_t(\zb^k)}{(1-2\lambda_t(\zb^k))^3}\right)\lambda_t(\zb^k)^2 < 1.
\end{equation}
For any $\beta \in (0, 0.18858]$, if we choose $\zb^0\in\Qc_t(\beta)$, the quadratic convergence region of \ref{eq:iFsGNM}, then the sequence $\set{\zb^k} \subset\Qc_t(\beta)$, and $\{\lambda_t(\zb^k)\}$ \textbf{quadratically} converges to zero.

\item[$\mathrm{(c)}$]
Let $c := \frac{2 - 4\beta + \beta^2}{(1-2\beta)^3} \in (0, 1)$ and $\varepsilon > 0$ be a given tolerance for $\beta \in (0, 0.18858]$.
If we choose $t := (1-\epsilon)(\sqrt{\nu} + \epsilon + 2\epsilon^2/(1-\epsilon))^{-1}\varepsilon$ for a sufficiently small $\epsilon \in (0, \beta)$, and update  $\delta_k :=  \frac{2\beta_k^2}{1-\beta_k}$ with $\beta_k := c^{2^k-1}\beta^{2^k}$, then after at most $k := \mathcal{O}(\ln(\ln(1/\epsilon)))$ iterations,  $\zb^k$ is  an $\varepsilon$-solution  of \eqref{eq:mono_inclusion} in the sense of Definition~\ref{de:MVIP_approx_sol}.
\end{itemize}
\end{theorem}

By a numerical experiment, we can show that $\bar{\delta}_k$ defined in Theorem \ref{th:generalized_newton_convergence}(a) is maximized at $\beta_{*} = 0.0997 \in (0, 0.18858]$ with $\bar{\delta}_k^{\ast} = 0.0372$.
Therefore, if we choose these values, we can maximize the tolerance $\delta_k$.
We note that $\bar{\delta}_k$ is decreasing when $\beta$ is increasing in $\left(0.0997, \frac{1}{2}(3-\sqrt{5})\right)$ and vice versa.
Hence, we can trade-off between the radius $\beta$ of $\Omega_{t}(\beta)$ and the tolerance $\delta_k$ of the subproblem in \eqref{eq:iFsGNM}.

\beforesubsubsec
\subsubsection{Inexact damped-step GN method \eqref{eq:iDsGNM}: Local convergence}
\aftersubsubsec
We now consider a damped-step generalized Newton scheme.
The following theorem summarizes the result whose proof is moved to Appendix \ref{apdx:th:generalized_ds_newton_convergence}.

\vspace{-1ex}
\begin{theorem}\label{th:generalized_ds_newton_convergence}
Given a fixed parameter $t > 0$, let $\set{\zb^k}$ be the sequence generated by the following inexact damped-step generalized Newton scheme \eqref{eq:iDsGNM}:
\begin{equation}\label{eq:iDsGNM}
\vspace{-0.75ex}
\left\{\begin{array}{ll}
\tilde{\zb}^{k+1} &\approx \bar{\zb}^{k+1} := \Pc_{\zb^k}\left(\zb^k  - \nabla^2{F}(\zb^k)^{-1}\nabla{F}(\zb^k); t\right),\vspace{0.75ex}\\
\alpha_k &:= \frac{1}{(1 +\tilde{\lambda}_t(\zb^k))}~~\textrm{with}~~\tilde{\lambda}_t(\zb^k) := \Vert \tilde{\zb}^{k+1} - \zb^k\Vert_{\zb^k}, \vspace{0.75ex}\\
\zb^{k+1} &:= (1-\alpha_k)\zb^k + \alpha_k\tilde{\zb}^{k+1}.
\end{array}\right.\tag{$\mathrm{DGN}$}
\vspace{-0.5ex}
\end{equation}
Then, we have three statements:
\begin{itemize}
\vspace{-1ex}
\item[$\mathrm{(a)}$]
If we choose $\delta_k$ such that $\delta_k \leq \frac{\tilde{\lambda}_{t}(\zb^k)^2}{1 + \tilde{\lambda}_{t}(\zb^k)}$, then
\vspace{-0.75ex}
\begin{equation}\label{eq:key_estimate_2b}
\tilde{\lambda}_{t}(\zb^{k+1}) \leq \left(\frac{2\tilde{\lambda}_t(\zb^k)^2 + 4\tilde{\lambda}_t(\zb^k) + 3}{1 - \tilde{\lambda}_t(\zb^k)^2\left(2\tilde{\lambda}_t(\zb^k)^2 + 4\tilde{\lambda}_t(\zb^k) + 3\right)}\right)\tilde{\lambda}_{t}(\zb^k)^2.
\vspace{-0.75ex}
\end{equation}
For any $\beta \in (0, 0.21027]$, the sequence $\set{\zb^k}$ generated by \ref{eq:iDsGNM} starting from any $\zb^0\in\Omega_{t}(\beta)$ belongs to $\Omega_t(\beta)$, i.e., $\set{\zb^k}\subset\Omega_t(\beta)$.

\item[$\mathrm{(b)}$]
If we choose $\beta \in (0, 0.21027]$, then the sequence $\big\{\tilde{\lambda}(\zb^k)\big\}$ generated by \ref{eq:iDsGNM} starting from any $\zb^0\in\Qc_{t}(\beta)$ also converges \textbf{quadratically} to zero.

\item[$\mathrm{(c)}$]
Let $\bar{c} := \frac{2\beta^2 + 4\beta + 3}{1 - \beta^2\left(2\beta^2 + 4\beta + 3\right)} \in (0, 1)$ and $\varepsilon > 0$ be a given tolerance for  $\beta \in (0, 0.21027]$.
If we choose $t := (1 - 2\epsilon^2)\left(\sqrt{\nu}(1+\epsilon) + \epsilon + 3\epsilon^2\right)^{-1}\varepsilon$ for a sufficiently small $\epsilon \in (0, \beta)$, and update  $\delta_k := (1+ \beta_k)^{-1}\beta_k^2$ with $\beta_k := \bar{c}^{2^k-1}\beta^{2^k}$, then after at most $k := \mathcal{O}\left(\ln\left(\ln(1/\epsilon)\right)\right)$ iterations, $\zb^k$ is an $\varepsilon$-solution of \eqref{eq:mono_inclusion} in the sense of Definition~\ref{de:MVIP_approx_sol}.
\end{itemize}
\end{theorem}

We note that the quadratic convergence stated in Theorem \ref{th:generalized_ds_newton_convergence} is given through $\big\{\tilde{\lambda}_t(\zb^k)\big\}$, which is computable as opposed to $\{\lambda_t(\zb^k)\}$ in Theorem \ref{th:generalized_newton_convergence}.
Due to the fact that $\lambda_t(\zb^k) \leq \tilde{\lambda}_t(\zb^k) + \delta(\zb^k) \leq \tilde{\lambda}_t(\zb^k) + \frac{\tilde{\lambda}_{t}(\zb^k)^2}{1 + \tilde{\lambda}_{t}(\zb^k)} \to 0^{+}$ as $k\to\infty$, we conclude that $\{\lambda_t(\zb^k)\}$ also converges to zero at a quadratic rate in the \ref{eq:iDsGNM} scheme.

\beforesubsubsec
\subsubsection{Inexact path-following GN method \eqref{eq:inexact_pf_scheme}: The worst-case iteration-complexity}
\aftersubsubsec
We consider the following inexact path-following generalized Newton scheme \eqref{eq:inexact_pf_scheme} for solving \eqref{eq:mono_inclusion} directly by simultaneously updating both $\zb$ and $t$ at each iteration:
\vspace{-0.5ex}
\begin{equation}\label{eq:inexact_pf_scheme}
\left\{\begin{array}{ll}
t_{k+1} &:= (1-\sigma_{\beta})t_k\vspace{1ex}\\
\zb^{k+1} &\approx \bar{\zb}^{k+1} := \Pc_{\zb^k}\left(\zb^k - \nabla^2{F}(\zb^k)^{-1}\nabla{F}(\zb^k); t_{k+1}\right),
\end{array}\right.\tag{$\mathrm{PFGN}$}
\vspace{-0.5ex}
\end{equation}
where $\sigma_{\beta}\in (0, 1)$ is a given factor.
As before, the approximation $\zb^{k+1}\approx\bar{\zb}^{k+1}$ is in the sense of Definition \ref{de:approx_sol} with a tolerance $\delta_k\geq 0$.

We emphasize that our \ref{eq:inexact_pf_scheme} scheme updates $t$ by decreasing it at each iteration, while the standard path-following scheme in \cite[4.2.23]{Nesterov2004} increases the penalty parameter at each iteration.
When $\Ac(\zb) = \cb$ is constant, we can define $s := \frac{1}{t}$ to obtain the scheme  \cite[4.2.23]{Nesterov2004}, and it allows us to start from $s = 0$.
This is not the case in our scheme when $\Ac(\zb) \neq \cb$.

Given $\beta \in (0, \frac{1}{2}(3-\sqrt{5}))$, we first find $\sigma_{\beta} \in (0, 1)$ such that if $\zb^k\in\Qc_{t_k}(\beta)$, then the new point $\zb^{k+1}$ at a new parameter $t_{k+1}$ still satisfies $\zb^{k+1}\in\Qc_{t_{k+1}}(\beta)$.
The following lemma proves  this key property, whose proof is deferred to Appendix \ref{apdx:le:update_penalty_param}.

\begin{lemma}\label{le:update_penalty_param}
Let $\set{(\zb^k, t_k)}$ be the sequence generated by the inexact path-following generalized Newton scheme  \eqref{eq:inexact_pf_scheme}.
Then, for $\zb^k$ with $ \lambda_{t_k}(\zb^k) < 1$, we have
\begin{align}\label{eq:key_estimate2}
\lambda_{t_{k+1}}(\zb^k) &\leq  \lambda_{t_k}(\zb^k) + \left(\frac{\sigma_{\beta}}{1 - \sigma_{\beta}}\right)\Big[ \Vert \nabla{F}(\zb^k)\Vert_{\zb^k}^{\ast} +   \lambda_{t_k}(\zb^k) \Big] \nonumber\\
& \leq \lambda_{t_k}(\zb^k) + \left(\frac{\sigma_{\beta}}{1 - \sigma_{\beta}}\right)\Big[ \sqrt{\nu} +   \lambda_{t_k}(\zb^k) \Big] .
\end{align}
Let us fix $c\in (0, 1]$. Then, for any $0 < \beta <0.5(1+2c^2 - \sqrt{1+4c^2})$, if the factor $\sigma_{\beta}$ and the tolerance $\delta_k$ are respectively chosen such that
\begin{equation}\label{eq:choice_of_sigma}
\begin{array}{lllll}
&0 < &\sigma_{\beta} \leq &\bar{\sigma}_{\beta} &:= \frac{c\sqrt{\beta} - \beta(1 + c\sqrt{\beta})}{(1+c\sqrt{\beta})\sqrt{\nu} +c \sqrt{\beta}},
~~\text{and}~~\vspace{1.2ex}\\
&0 \leq &\delta_k \leq &\bar{\delta}_t(\beta) &:= \frac{(1-c^2)\beta}{(1+c\sqrt{\beta})^3\left[3c\sqrt{\beta} + c^2\beta + (1+c\sqrt{\beta})^3\right]},
\end{array}
\end{equation}
then $\lambda_{t_k}(\zb^k) \leq \beta$ implies $\lambda_{t_{k+1}}(\zb^{k+1}) \leq \beta$.
In addition, $\lambda_{t_{k+1}}(\zb^k) \leq \frac{c\sqrt{\beta}}{1+c\sqrt{\beta}}$.
\end{lemma}

As an example, if we choose $c := 0.95$, then the possible interval for $\beta$ is $(0, 0.32895)$.
Now, if we choose $\beta := \frac{1}{9c^2} \in (0, 0.32895)$ (i.e., $\beta \approx 0.12311$), then $\bar{\sigma}_{\beta} = \frac{5}{36\sqrt{\nu} + 9}$, which is the same as in the standard path-following method in \cite{Nesterov2004}.
In this case, the tolerance $\delta_k$ for the subproblem at the second line of \ref{eq:inexact_pf_scheme} must be chosen such that $0 \leq \delta_k \leq 7.45933\times 10^{-4}$.
Figure \ref{fig:delta_sigma} plots the values of $\bar{\delta}_t(\beta)$ and $\bar{\sigma}_{\beta}$ in \eqref{eq:choice_of_sigma} as a function of $\beta$, respectively for given $c = 0.95$ and $\nu = 1000$.
\begin{figure}[ht!]
\vspace{-3ex}
\begin{center}
\includegraphics[width = 1\textwidth]{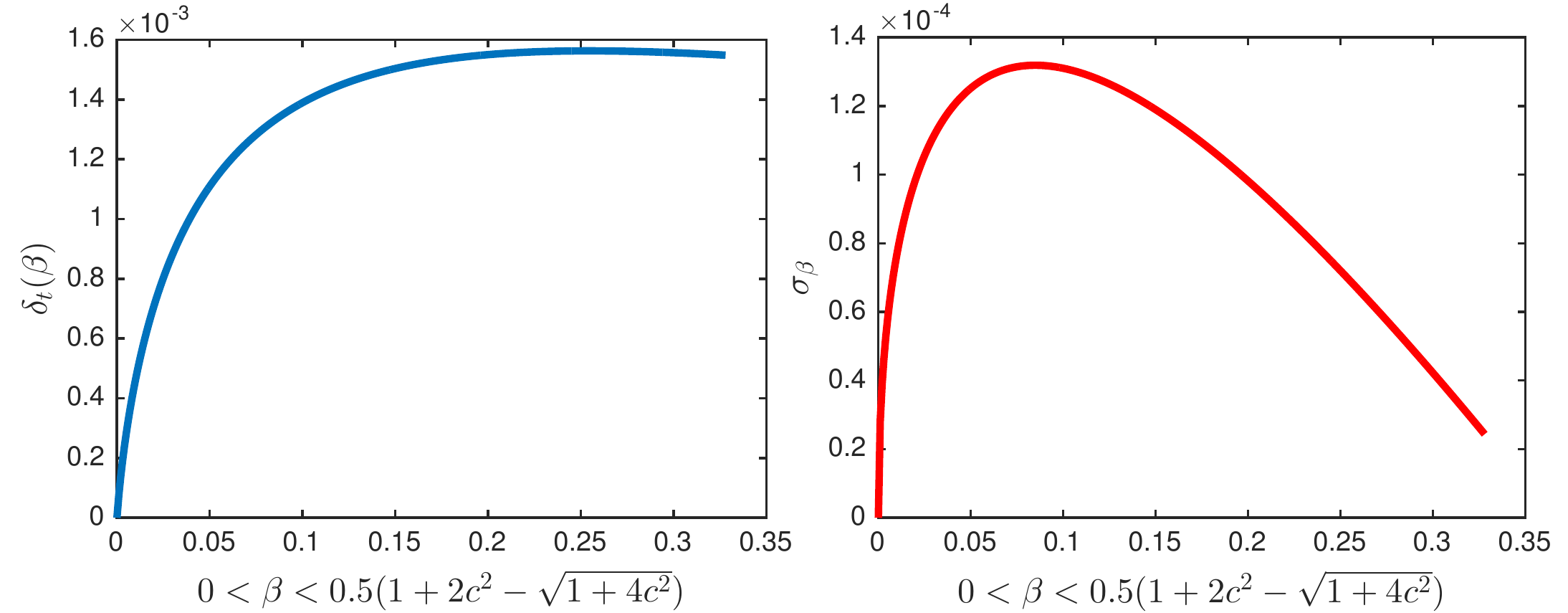}
\caption{The graph of the two functions $\bar{\delta}_t$ and $\bar{\sigma}_{\beta}$ with respect to $\beta$.}\label{fig:delta_sigma}
\end{center}
\vspace{-5ex}
\end{figure}
This figure shows that $\bar{\delta}_t$ is an increasing function of $\beta$, while $\bar{\sigma}_{\beta}$ has the maximum point at $\beta = 0.0870$.
Hence, a good choice of $\beta$ is  $\beta = 0.0870$.

The following theorem investigates the worst-case iteration-complexity of \ref{eq:inexact_pf_scheme} using the update rule \eqref{eq:choice_of_sigma} for $\sigma_{\beta}$.
The proof of this theorem can be found in Appendix \ref{apdx:th:ppf_convergence}.

\begin{theorem}\label{th:ppf_convergence}
Let $\set{(\zb^k, t_k)}$ be  generated by \ref{eq:inexact_pf_scheme} under the following configuration:
\begin{itemize}
\item[$\mathrm{(i)}$]  $c\in (0, 1]$ is given, and $\beta$ is chosen such that $0 < \beta <0.5(1+2c^2 - \sqrt{1+4c^2})$.
\item[$\mathrm{(ii)}$] The initial \Shu{points} $\zb^0$ and $t_0 > 0$ are chosen such that $\zb^0\in\intx{\Zc}$ and $\lambda_{t_0}(\zb^0) \leq \beta$.
\end{itemize}
Then, the following conclusions hold:
\begin{itemize}
\item[$\mathrm{(a)}$] $\lambda_{t_k}(\zb^k) \leq \beta$ for all $k\geq 0$.
\item[$\mathrm{(b)}$] The number of iterations $k$ to achieve an $\varepsilon$-solution $\zb^k$ of \eqref{eq:mono_inclusion} in the sense of Definition \ref{de:MVIP_approx_sol} does not exceed
\begin{equation*}
k_{\max} := \left\lfloor \left(\frac{(1+ c\sqrt{\beta})\sqrt{\nu} + c\sqrt{\beta}}{c\sqrt{\beta} - \beta(1 + c\sqrt{\beta})}\right) \ln\left(\frac{M_0 t_0}{\varepsilon}\right) \right\rfloor + 1,
\end{equation*}
where  $M_0 := \left(1 - \frac{c\sqrt{\beta}}{1 + c\sqrt{\beta}} - \bar{\delta}_t(\beta)\right)^{-1}\left(\sqrt{\nu} + \frac{c\sqrt{\beta}}{1 + c\sqrt{\beta}} + 2\bar{\delta}_t(\beta)\right) = \mathcal{O}(\sqrt{\nu})$.
\item[$\mathrm{(c)}$] Consequently, the worst-case iteration-complexity of  \ref{eq:inexact_pf_scheme} is $\mathcal{O}\left( \sqrt{\nu}\ln\Big(\frac{\sqrt{\nu} t_0}{\varepsilon}\Big)\right)$.
\end{itemize}
\end{theorem}

Theorem \ref{th:ppf_convergence} requires a starting point $\zb^0 \in\Omega_{t_0}(\beta)$ at a given penalty parameter $t_0 > 0$. In order to find $\zb^0$, we perform an initial phase (called Phase 1) as described below.

\beforesubsubsec
\subsubsection{Finding an initial point with the path-following iterations using auxiliary problem}
\aftersubsubsec
When $\Ac(\cdot) = \partial{H}(\cdot)$ the subgradient of a proper, closed and convex function $H$, we  can find $\zb^0 \in\Omega_{t_0}(\beta)$ for \ref{eq:inexact_pf_scheme} by applying the [inexact] damped-step generalized Newton scheme \eqref{eq:iFsGNM} to solve \eqref{eq:barrier_mono_inc} for fixed penalty parameter $t = t_0$.
This scheme has a sublinear convergence rate \cite{TranDinh2013e}.
However, it is still unclear how to generalize this method to \eqref{eq:mono_inclusion}.

We instead propose a new auxiliary problem for \eqref{eq:barrier_mono_inc}, and then apply \ref{eq:inexact_pf_scheme} to solve this auxiliary problem in order to obtain $\zb^0$.
Then, we estimate the maximum number of the path-following iterations needed in this phase.

Let us fix some $\hat{\zb}^0\in\intx{\Zc}$. We first compute a vector $\hat{\xi}^0\in\Ac(\hat{\zb}^0)$ and evaluate $\nabla{F}(\hat{\zb}^0)$.
Then, we define $\hat{\zeta}^0 := t_0\nabla{F}(\hat{\zb}^0) + \hat{\xi}^0$, and consider the following auxiliary problem of \eqref{eq:barrier_mono_inc}:
\begin{equation}\label{eq:auxi_prob1}
\textrm{Find~$\hat{\zb}^{\star}_{\tau}\in\intx{\Zc}$ such that:}~0 \in t_0\nabla{F}(\hat{\zb}^{\star}_{\tau}) - \tau\hat{\zeta}^0 + \Ac(\hat{\zb}^{\star}_{\tau}),
\end{equation}
where $\tau \in [0, 1]$ is a new homotopy parameter.
Clearly, \eqref{eq:auxi_prob1} has a similar form as \eqref{eq:barrier_mono_inc}.
\begin{itemize}
\item[$\mathrm{(a)}$] When $\tau = 1$, we have $0 \in t_0\nabla{F}(\hat{\zb}^0) - \hat{\zeta}^0 + \Ac(\hat{\zb}^0)$ due to the choice of $\hat{\zeta}^0$.
Hence, $\hat{\zb}^0$ is an exact solution of \eqref{eq:auxi_prob1} at $\tau = 1$.

\item[$\mathrm{(b)}$] When $\tau =0$, we have $0 \in t_0\nabla{F}(\hat{\zb}^{\star}_{\tau}) + \Ac(\hat{\zb}^{\star}_{\tau})$. Hence, any solution $\hat{\zb}^{\star}_{\tau}$ of \eqref{eq:auxi_prob1} is also \Tianxiao{a} solution of \eqref{eq:barrier_mono_inc} at $t = t_0$.
\end{itemize}
Now, we can apply \ref{eq:inexact_pf_scheme} to solve \eqref{eq:auxi_prob1} starting from $\tau_0 = 1$ but using a different update rule for $\tau$.
More precisely, this scheme can be written as follows:
\begin{equation}\label{eq:inexact_pf_scheme_phase1}
\left\{\begin{array}{ll}
\tau_{j+1} &:= \tau_j - \Delta_j,\vspace{1ex}\\
\hat{\zb}^{j+1} &\approx \bar{\hat{\zb}}^{j+1} := \Pc_{\hat{\zb}^j}\left(\hat{\zb}^j - \nabla^2{F}(\hat{\zb}^j)^{-1}\left(\nabla{F}(\hat{\zb}^j)- \tau_{j+1}t_0^{-1}\hat{\zeta}^0\right); t_0\right),
\end{array}\right.
\end{equation}
where $\hat{\sigma}_{\eta}\in (0, 1)$ is a given factor.
Here, the approximation $\hat{\zb}^{j+1}\approx\bar{\hat{\zb}}^{j+1}$ is in the sense of Definition \ref{de:approx_sol} with a given tolerance $\hat{\delta}_j\geq 0$.
We also use  the index $j$ to distinguish with the index $k$ in \ref{eq:inexact_pf_scheme}, and using the notation ``hat'' for the iterate vectors.

Similar to \eqref{eq:nt_decrement}, we define the following generalized Newton decrement for \eqref{eq:auxi_prob1}:
\begin{equation}\label{eq:nt_decrement0}
\hat{\lambda}_{\tau}(\hat{\zb}) := \big\Vert \hat{\zb} - \Pc_{\hat{\zb}}\big(\hat{\zb} - \nabla^2{F}(\hat{\zb})^{-1}\big(\nabla{F}(\hat{\zb})- \tau t_0^{-1}\hat{\zeta}^0\big); t_0\big)\big\Vert_{\hat{\zb}}.
\end{equation}
The theorem below provides the number of iterations $j$ needed to find an initial point $\zb^0 \in\Omega_{t_0}(\beta)$ for \ref{eq:inexact_pf_scheme} using \eqref{eq:inexact_pf_scheme_phase1}, whose proof can be found in Appendix \ref{apdx:th:complexity_of_phase_1}.

\begin{theorem}\label{th:complexity_of_phase_1}
Let $c \in (0, 1]$ and $\beta$ be chosen as in Theorem \ref{th:ppf_convergence}, and  $\eta$ be chosen such that $0 < \eta < \beta$.
Let $\set{(\hat{\zb}^j, \tau_j)}$ be  generated by \eqref{eq:inexact_pf_scheme_phase1}.
If $\Delta_j$ and $\hat{\delta}_j$ are chosen such that
\begin{equation}\label{eq:choice_of_sigma0}
\begin{array}{lllll}
&0 \leq &\Delta_j \leq &\frac{\bar{\mu}_{\eta}}{\Vert\hat{\zeta}_0\Vert_{\hat{\zb}^j}^{\ast}}&~\text{with}~\bar{\mu}_{\eta} := \frac{t_0}{\Vert\hat{\zeta}_0\Vert_{\hat{\zb}^j}^{\ast}}\left(\frac{c\sqrt{\eta}}{1+c\sqrt{\eta}} - \eta\right),~~~\text{and}\vspace{1.25ex}\\
&0 \leq &\hat{\delta}_j \leq &\bar{\delta}_{\tau}(\eta) := & \frac{(1-c^2)\eta}{(1+c\sqrt{\eta})^3\left[3c\sqrt{\eta} + c^2\eta + (1+c\sqrt{\eta})^3\right]},
\end{array}
\end{equation}
then $\hat{\lambda}_{\tau_j}(\hat{\zb}^j)$ defined \Shu{in} \eqref{eq:nt_decrement0} satisfies $\hat{\lambda}_{\tau_j}(\hat{\zb}^j) \leq \eta$ for all $j\geq 0$.

Let $\zb^0 := \hat{\zb}^{j_{\max}}$ be obtained after $j_{\max}$ iterations.
Then, $\hat{\lambda}_{\tau_0}(\hat{\zb}^0) = 0$ and we have
\begin{equation}\label{eq:diff_nt_decrement}
\lambda_{t_0}(\zb^0) \leq \hat{\lambda}_{\tau_j}(\hat{\zb}^j) + t_0^{-1}\tau_j\Vert\hat{\zeta}^0\Vert_{\hat{\zb}^j}^{\ast} \leq \eta + \frac{\kappa\Vert\hat{\zeta}^0\Vert_{\bar{\zb}^{\star}_F}^{\ast}}{t_0} - j\left(\frac{c\sqrt{\eta}}{1+c\sqrt{\eta}} - \eta\right),~\forall j\geq j_{\max},
\end{equation}
where $\lambda_t(\zb)$ is defined by \eqref{eq:nt_decrement}, and $\bar{\zb}^{\star}_F$ and $\kappa$ are defined by \eqref{eq:analytical_center} and  below \eqref{eq:analytical_center}, respectively.

The number of iterations $j$ to achieve $\zb^0 := \hat{\zb}^j$ such that $\lambda_{t_0}(\zb^0) \leq \beta$ does not exceed
\begin{equation*}
j_{\max} := \left\lfloor \frac{ \kappa(1 + c\sqrt{\eta})\Vert\hat{\zeta}^0\Vert_{\bar{\zb}^{\star}_F}^{\ast}}{t_0\left(c\sqrt{\eta} - \eta(1 + c\sqrt{\eta})\right)}  -  \frac{(\beta - \eta)(1 + c\sqrt{\eta})}{c\sqrt{\eta} - \eta(1 + c\sqrt{\eta}} \right\rfloor  + 1.
\end{equation*}
The worst-case iteration-complexity of \eqref{eq:inexact_pf_scheme_phase1} to obtain $\zb^0$ such that $\lambda_{t_0}(\zb^0) \leq \beta$ is $\mathcal{O}\left( \frac{\kappa\Vert\hat{\zeta}^0\Vert_{\bar{\zb}^{\star}_F}^{\ast}}{t_0} \right)$.
\end{theorem}
Theorem~\ref{th:complexity_of_phase_1} suggests us to choose $t_0 := \kappa$.
In this case, the maximum number of iterations in Phase 1 does not exceed $ \frac{ (1 + c\sqrt{\eta})\Vert\hat{\zeta}^0\Vert_{\bar{\zb}^{\star}_F}^{\ast}}{c\sqrt{\eta} - \eta(1 + c\sqrt{\eta})}$, which is a constant.

\begin{remark}\label{re:analytical_center}
From \eqref{eq:diff_nt_decrement}, we can compute $\Vert\hat{\zeta}^0\Vert_{\hat{\zb}^j}^{\ast}$ directly in order to terminate \eqref{eq:inexact_pf_scheme_phase1} by checking $\tau_j\Vert\hat{\zeta}^0\Vert_{\hat{\zb}^j}^{\ast} \leq t_0(\beta-\eta)$.
Hence, \eqref{eq:inexact_pf_scheme_phase1} does not require to evaluate the analytical center $\bar{\zb}^{\star}_F$ of $F$.
If $F$ is a self-concordant logarithmically homogeneous barrier, then we simply choose $t_0 := 1$.
Otherwise, we can choose $t_0 := \nu + 2\sqrt{\nu}$.
\end{remark}

\beforesubsubsec
\subsubsection{Two-phase inexact path-following generalized Newton algorithm}
\aftersubsubsec
Putting two schemes \eqref{eq:inexact_pf_scheme_phase1} and \ref{eq:inexact_pf_scheme} together, we obtain a two-phase inexact path-following generalized Newton  algorithm for solving \eqref{eq:mono_inclusion} as described in Algorithm \ref{alg:A1}.

\vspace{-2ex}
\begin{algorithm}[ht!]\caption{(\textit{Two-phase inexact path-following generalized Newton  algorithm})}\label{alg:A1}
\begin{normalsize}
\begin{algorithmic}[1]
    \STATE {\bfseries Initialization:}
     \STATE  \hspace{0.16cm} Choose an arbitrary initial point $\hat{\zb}^0 \in\intx{\Zc}$ and a desired accuracy $\varepsilon > 0$
     \STATE  \hspace{0.16cm} Fix $t_0 > 0$ and $c\in (0, 1]$ (e.g., $t_0 := \kappa$, and $c := 0.95$).
      \STATE  \hspace{0.16cm} Compute $\hat{\xi}^0\in\Ac(\hat{\zb}^0)$ and evaluate $\nabla{F}(\hat{\zb}^0)$.
       Set $\hat{\zeta}^0 := t_0\nabla{F}(\hat{\zb}^0) + \hat{\xi}^0$ and $\tau_0 := 1$.
       \STATE  \hspace{0.16cm} Fix $\beta$ as in Theorem \ref{th:ppf_convergence} (e.g., $\beta := \frac{1}{9c^2}$) and choose  $\eta < \beta$ (e.g., $\eta := 0.5\beta$).
       \STATE\label{step:6}\hspace{0.16cm} Compute $\bar{\delta}_{\tau}$, $\bar{\mu}_{\eta}$, $\bar{\delta}_t$ and $\bar{\sigma}_{\beta}$ by \eqref{eq:choice_of_sigma0} and \eqref{eq:choice_of_sigma}, respectively, and $M_0$ from Theorem~\ref{th:ppf_convergence}.\\
   \vspace{-1ex}
    \rule{0.88\textwidth}{0.1mm}
    \STATE\begin{center}\textbf{Phase 1: Computing an initial point by path-following iterations}\vspace{-1ex} \end{center}
    \rule{0.88\textwidth}{0.1mm}
   \STATE {\bfseries For} $j = 0, \cdots, j_{\max}$, \textbf{perform:}
   \STATE\label{step:init_z0}\hspace{0.16cm} If $\tau_j\Vert\hat{\zeta}^0\Vert_{\hat{\zb}^j}^{\ast} \leq t_0(\beta-\eta)$, then set $\zb^0 := \hat{\zb}^j$ and TERMINATE.
   \STATE \hspace{0.16cm} Update $(\hat{\zb}^{j+1}, \tau_{j+1})$ by \eqref{eq:inexact_pf_scheme_phase1} with $\Delta_j := \tfrac{\bar{\mu}_{\eta}}{\Vert\hat{\zeta}_0\Vert_{\hat{\zb}^j}^{\ast}}$ up to an accuracy $\hat{\delta}_j \leq \bar{\delta}_{\tau}$.
   \STATE {\bfseries End for}\\
   \vspace{-1ex}
   \rule{0.88\textwidth}{0.1mm}
    \STATE\begin{center}{\textbf{Phase 2: Inexact path-following generalized Newton iterations}}\vspace{-1ex}\end{center}
    \rule{0.88\textwidth}{0.1mm}
   \STATE {\bfseries For} $k = 0, \cdots, k_{\max}$, \textbf{perform:}
   \STATE \hspace{0.16cm} If $M_0 t_k \leq \varepsilon$, then return $\zb^k$ as an $\varepsilon$-solution of \eqref{eq:mono_inclusion}, and TERMINATE.
   \STATE\label{step:pf_z} \hspace{0.16cm} Update $(\zb^{k+1}, t_{k+1})$ by \eqref{eq:inexact_pf_scheme} up to an accuracy $\delta_k \leq \bar{\delta}_t$.
   \STATE {\bfseries End for}
\end{algorithmic}
\end{normalsize}
\end{algorithm}
\vspace{-2ex}

The main computational cost of Algorithm \ref{alg:A1} is the solution of the two linear monotone inclusions in \eqref{eq:inexact_pf_scheme_phase1} and \ref{eq:inexact_pf_scheme}, respectively.
When $\Ac = \partial{H}$ the subdifferential of a convex function $H$, various methods  including fast gradient, primal-dual methods, and splitting techniques can be used to solve these problems \cite{Beck2009,Becker2012a,Boyd2011,Esser2010a,friedlander2016efficient,Nesterov2007}.

The overall worst-case iteration-complexity of Algorithm \ref{alg:A1} is given in the following theorem which is a direct consequence of Lemma~\ref{le:sol_measure}, Theorems \ref{th:ppf_convergence} and \ref{th:complexity_of_phase_1}.

\begin{theorem}\label{th:overall_complexity}
Let us choose $t_0 := \kappa$ as defined below \eqref{eq:analytical_center}.
Then, the overall worst-case iteration-complexity of Algorithm \ref{alg:A1} to achieve an $\varepsilon$-solution $\zb^k$ of \eqref{eq:mono_inclusion} as in Definition \ref{de:MVIP_approx_sol} is
\begin{equation*}
\mathcal{O}\left( \frac{\kappa\Vert\hat{\zeta}^0\Vert_{\bar{\zb}^{\star}_F}^{\ast}}{t_0} + \sqrt{\nu}\ln\left(\frac{M_0t_0}{\varepsilon}\right) \right) ~~~~~\left(\text{or simpler}~~\mathcal{O}\left(\sqrt{\nu}\ln\left(\frac{\kappa\sqrt{\nu}}{\varepsilon}\right)\right) \right),
\end{equation*}
where $t_0 > 0$ is an initial penalty parameter and $M_0 = \mathcal{O}(\sqrt{\nu})$ is defined in Theorem~\ref{th:ppf_convergence}.
\end{theorem}

\begin{proof}
The total number of iterations requires in Phase 1 and Phase 2 of Algorithm~\ref{alg:A1} is
 \begin{equation*}
 K_{\max} \geq   \frac{ \kappa(1 + c\sqrt{\eta})\Vert\hat{\zeta}^0\Vert_{\bar{\zb}^{\star}_F}^{\ast}}{t_0\left(c\sqrt{\eta} - \eta(1 + c\sqrt{\eta})\right)}  + C_2\ln\left(\frac{M_0 t_0}{\varepsilon}\right) - C_1.
 \end{equation*}
 where $C_1 :=  \frac{(\beta - \eta)(1 + c\sqrt{\eta})}{c\sqrt{\eta} - \eta(1 + c\sqrt{\eta}} $, and $C_2 := \left(\frac{(1+ c\sqrt{\beta})\sqrt{\nu} + c\sqrt{\beta}}{c\sqrt{\beta} - \beta(1 + c\sqrt{\beta})}\right)$.
 We note that $C_1$ is a constant, while $C_2 = \mathcal{O}(\sqrt{\nu})$.
 Hence, $K_{\max} \geq C_3\frac{\kappa\Vert\hat{\zeta}^0\Vert_{\bar{\zb}^{\star}_F}^{\ast}}{t_0} + \mathcal{O}(\sqrt{\nu})\ln\left(\frac{M_0t_0}{\varepsilon}\right) - C_1$, where $C_3 := \frac{1+c\sqrt{\eta}}{c\sqrt{\eta} - \eta(1 + c\sqrt{\eta})}$.
We can write this as  $K_{\max} \geq \mathcal{O}\left( \frac{\kappa\Vert\hat{\zeta}^0\Vert_{\bar{\zb}^{\star}_F}^{\ast}}{t_0} + \sqrt{\nu}\ln\left(\frac{M_0t_0}{\varepsilon}\right) \right)$.
We finally note that $M_0 = \mathcal{O}(\sqrt{\nu})$ and $t_0 = \kappa$, and the first term is a constant and independent of  $\varepsilon$, which is dominated by the second term. Hence, we obtain the simpler second estimate of Theorem~\ref{th:overall_complexity}.
\Eproof
\end{proof}

The complexity bound in Theorem \ref{th:overall_complexity} also \Shu{depends} on the choice of $\beta$, $\eta$ and $t_0$.
Adjusting these parameters allows us to trade-off between Phase 1 and Phase 2 in Algorithm \ref{alg:A1}.
Clearly, if $t_0$ is large, the number of iterations required in Phase 1 is small, but the number of iterations in Phase 2 is large, and vice versa.

\begin{remark}\label{re:exact_case}
We note that we can recover the convergence guarantee of the exact generalized Newton-type schemes as consequences of Theorems \ref{th:generalized_newton_convergence}, \ref{th:generalized_ds_newton_convergence} and \ref{th:overall_complexity}, respectively.
For instance, in the exact variant of \eqref{eq:inexact_pf_scheme}, if we can choose $\beta\in (0, 0.5(3-\sqrt{5}))$, then the upper bound $\bar{\sigma}_{\beta}$ in  \eqref{eq:choice_of_sigma} reduces to $\bar{\sigma}_{\beta} := \frac{\sqrt{\beta} - \beta(1 + \sqrt{\beta})}{(1+\sqrt{\beta})\sqrt{\nu} + \sqrt{\beta}}$.
Hence, we can show that the worst-case iteration-complexity estimate of this exact scheme \Shu{coincides} with the standard path-following scheme for smooth structural convex programming given in \cite[Theorem 4.2.9]{Nesterov2004}.
\end{remark}

\beforesec
\section{Inexact  path-following proximal  Newton algorithms}\label{sec:pd_path_following}
\aftersec
\vspace{1ex}
We now specify our framework, Algorithm \ref{alg:A1}, to solve three  problems:  \eqref{eq:constr_cvx},  \eqref{eq:constr_cvx2} and \eqref{eq:saddle_point_prob}.

\beforesubsec
\subsection{\bf Inexact primal-dual path-following algorithm for saddle-point problems}
\aftersubsec
We recall the convex-concave saddle-point problem \eqref{eq:saddle_point_prob}.
Our primal-dual path-following proximal  Newton method relies on the following assumption.

\begin{assumption}\label{as:A1}
\begin{itemize}
\item[$\mathrm{(a)}$] The feasible set $\Xc$ $($resp., $\Yc$$)$ is a nonempty, closed, and convex cone with nonempty interior, and is endowed with a $\nu_f$-self-concordant barrier $f$ $($respectively, a $\nu_{\varphi}$-self-concordant barrier $\varphi$$)$ such that $\mathrm{Dom}(f) = \Xc$ $($respectively, $\mathrm{Dom}(\varphi)  = \Yc$$)$.
\item[$\mathrm{(b)}$] Both $g$ and \Shu{$\psi$}  in \eqref{eq:saddle_point_prob} are proper, closed and convex such that $\intx{\Xc}\cap\dom{g}\neq\emptyset$ and $\intx{\Yc}\cap\dom{\psi}\neq\emptyset$.
\item[$\mathrm{(c)}$] The solution set $\Zc^{\star}$ of \eqref{eq:saddle_point_prob} is nonempty.
\end{itemize}
\end{assumption}
For any $\zb = (\xb, \yb)$, $\hat{\zb} = (\hat{\xb}, \hat{\yb})$, $(\xi_g, \xi_{\Shu{\psi}}) \in\partial{g}(\xb)\times\partial{\Shu{\psi}}(\yb)$, and $(\hat{\xi}_g,\hat{\xi}_{\Shu{\psi}}) \in\partial{g}(\hat{\xb}) \times \partial{\Shu{\psi}}(\hat{\yb})$:
\begin{align*}
\begin{bmatrix}\xi_g - L^{\ast}\yb - \hat{\xi}_g + L^{\ast}\hat{\yb}\\ \xi_{\Shu{\psi}} + L\xb - \hat{\xi}_{\Shu{\psi}} - L\hat{\xb}\end{bmatrix}^{\top}\begin{bmatrix} \xb - \hat{\xb}\\ \yb - \hat{\yb}\end{bmatrix} \geq 0.
\end{align*}
This shows that $\Ac$ defined by \eqref{eq:pd_map} is maximally monotone.
In addition, $F$ is a self-concordant barrier of $\Zc = \Xc\times \Yc$ with the barrier parameter $\nu := \nu_f+\nu_{\varphi}$.

\beforesubsubsec
\subsubsection{Inexact primal-dual path-following  proximal  Newton method}
\aftersubsubsec
We specify \ref{eq:inexact_pf_scheme}  to solve \eqref{eq:saddle_point_prob}.
Let $\zb^k := (\xb^k, \yb^k)$ be a given point at $t_k > 0$.
We update  $\zb^{k + 1} := (\xb^{k + 1}, \yb^{k + 1})$ and $t_{k + 1}$ using \ref{eq:inexact_pf_scheme}, which reduces to the following form:
\begin{equation}\label{eq:cvx_subprob_saddle_point}
\left\{\begin{array}{ll}
0 &\in t_{k+1}\left[\nabla{f}(\xb^k) +   \nabla^2{f}(\xb^k)(\xb - \xb^k)\right] - L^{\ast}\yb + \partial{g}(\xb),\vspace{1ex}\\
0 &\in t_{k+ 1}\left[\nabla{\varphi}(\yb^k)  + \nabla^2{\varphi}(\yb^k)(\yb - \yb^k)\right] + L\xb + \partial{\psi}(\yb).
\end{array}\right.
\end{equation}
Here, we solve \eqref{eq:cvx_subprob_saddle_point} approximately as done in \ref{eq:inexact_pf_scheme}.
Hence,  \ref{eq:inexact_pf_scheme} can be rewritten as
\begin{equation}\label{eq:cvx_subprob2b}
{\!\!\!\!}\left\{ \begin{array}{lll}
t_{k  + 1} &:= (1-\sigma_{\beta})t_k, & \vspace{1ex}\\
\zb^{k +1} &\approx \bar{\zb}^{k+ 1} \! :=\! \displaystyle\mathrm{arg}\!\min_{\yb}\!\max_{\xb}\Big\{  t_{k + 1}Q_{\varphi}(\yb;\yb^k) \!+\! \psi(\yb)  \!+\!  \iprods{L\xb,\yb} \!-\! t_{k+1}Q_f(\xb;\xb^k) \!-\! g(\xb)  \Big\},
\end{array}\right.{\!\!\!\!}
\end{equation}
where $Q_f(\cdot; \xb^k)$ and $Q_{\varphi}(\cdot; \yb^k)$ are the quadratic surrogates of $f$ and $\varphi$, respectively, i.e.:
\vspace{-0.5ex}
\begin{equation}\label{eq:cvx_subprob2b1}
\left\{ \begin{array}{lll}
&Q_f(\xb;\xb^k)  &:= \iprods{\nabla{f}(\xb^k), \xb-\xb^k} + \tfrac{1}{2}\iprods{\nabla^2{f}(\xb^k)(\xb - \xb^k), \xb -\xb^k},\vspace{1ex}\\
&Q_{\varphi}(\yb;\yb^k) &:=  \iprods{\nabla{\varphi}(\yb^k), \yb - \yb^k } + \tfrac{1}{2}\iprods{\nabla^2{\varphi}(\yb^k)(\yb - \yb^k), \yb - \yb^k}.
\end{array}\right.
\vspace{-0.5ex}
\end{equation}
The second line of \eqref{eq:cvx_subprob2b} is again a linear convex-concave saddle-point problem with strongly convex objectives.
Methods for solving this problem can be found, e.g., in \cite{Bauschke2011,Chambolle2011,Esser2010a}.

\beforesubsubsec
\subsubsection{Finding initial point}
\aftersubsubsec
We specify \eqref{eq:inexact_pf_scheme_phase1} for finding an initial point as follows.
\begin{enumerate}
\vspace{-0.5ex}
\item Provide a value $t_0 > 0$ (e.g., $t_0 := \kappa$), and an initial point $\hat{\zb}^0 := (\hat{\xb}^0, \hat{\yb}^0) \in \intx{\Zc}$.

\vspace{0.25ex}
\item Compute a subgradient $\hat{\xi}^0_g\in\partial{g}(\hat{\xb}^0)$ and $\hat{\xi}_{\psi}^0\in\partial{\psi}(\hat{\yb}^0)$, and  evaluate $\nabla{f}(\hat{\xb}^0)$ and $\nabla{\varphi}(\hat{\yb}^0)$.

\vspace{0.25ex}
\item Define $\hat{\zeta}^0_g := t_0\nabla{f}(\hat{\xb}^0) - L^{\ast}\hat{\yb}^0 + \hat{\xi}^0_g$ and $\hat{\zeta}^0_{\psi} := t_0\nabla{\varphi}(\hat{\yb}^0) + L\hat{\xb}^0 + \hat{\xi}^0_{\psi}$.

\vspace{0.25ex}
\item Perform  Phase 1 of Algorithm \ref{alg:A1} applied to solve \eqref{eq:saddle_point_prob} as follows:
\vspace{-0.5ex}
\begin{equation}\label{eq:cvx_subprob2b_x0}
{\!\!\!\!\!}\left\{\begin{array}{ll}
\tau_{j + 1} &{\!\!} := \tau_j - \bar{\Delta}_{j}, \vspace{1ex}\\
\hat{\zb}^{j + 1} &{\!\!}\approx \bar{\hat{\zb}}^{j + 1} :=  \displaystyle\mathrm{arg}\min_{\yb}\max_{\xb}\Big\{ t_0Q_{\varphi}(\yb;\hat{\yb}^j) - \tau_{j  + 1}\iprods{\hat{\zeta}^0_{\psi},\yb} + \psi(\yb) +  \iprods{L\xb,\yb}  \vspace{1ex}\\
&{~~~~~~~~~~~~~~~~~~~~~~~~~~~~~~} - t_0Q_f(\xb;\hat{\xb}^j) + \tau_{j  +1}\iprods{\hat{\zeta}^0_g,\xb} - g(\xb) \Big\}.\\
\end{array}\right.{\!\!\!\!\!}
\vspace{-0.5ex}
\end{equation}
Here, $\tau \in (0, 1]$ is referred to as a new homotopy parameter starting from $\tau_0 := 1$.
\vspace{-0.5ex}
\end{enumerate}
Now, we substitute this scheme into Phase 1 and \eqref{eq:cvx_subprob2b} into Step~\ref{step:pf_z} of Algorithm \ref{alg:A1}, respectively, \Shu{to }obtain a new variant to solve \eqref{eq:saddle_point_prob}, which we call Algorithm~\ref{alg:A1}(a).

The worst-case iteration-complexity of Algorithm~\ref{alg:A1}(a) to achieve an $\varepsilon$-primal-dual solution $\zb^k := (\xb^k, \yb^k)$ in the sense of Definition \ref{de:MVIP_approx_sol} for the optimality condition \eqref{eq:saddle_point_opt_cond} instead of \eqref{eq:mono_inclusion} remains guaranteed by Theorem \ref{th:overall_complexity}.
We omit the \Tianxiao{detailed} proof in this paper.

\beforesubsec
\subsection{\bf Inexact path-following primal proximal  Newton algorithm}\label{subsec:primal_path_following}
\aftersubsec
We present an inexact  primal path-following proximal Newton method obtained from Algorithm \ref{alg:A1} to solve \eqref{eq:constr_cvx}.
This algorithm has  several new  features compared to \cite{TranDinh2013e,TranDinh2015f}.

First, associated with the barrier function $f$ of $\Xc$ in \eqref{eq:constr_cvx}, we define the local norm $\Vert\ub\Vert_{\xb} := \iprods{\nabla^2f(\xb)\ub,\ub}^{1/2}$ and its corresponding dual norm $\Vert\vb\Vert_{\xb}^{\ast} := \iprods{\nabla^2f(\xb)^{-1}\vb,\vb}^{1/2}$ for a given $\xb\in\dom{f}$.
Next, let $Q_f$ be the quadratic surrogate of $f$ around $\xb^k$ as defined in \eqref{eq:cvx_subprob2b1}.
Then, the main step of Algorithm \ref{alg:A1} applied to \eqref{eq:constr_cvx} performs the following inexact path-following proximal Newton scheme:
\vspace{-0.5ex}
\begin{equation}\label{eq:primal_cvx_subprob}
\left\{\begin{array}{ll}
t_{k+1} &:= (1-\sigma_{\beta})t_k,\vspace{1.0ex}\\
\xb^{k+1} &\approx \bar{\xb}^{k+1} := \displaystyle\argmin_{\xb}\set{h_k(\xb) := Q_f(\xb; \xb^k) + t_{k+1}^{-1}g(\xb)},
\end{array}\right.
\vspace{-0.5ex}
\end{equation}
Here, the approximation $\approx$ \Shu{is in the sense of Definition \ref{de:approx_sol} and implies  $\Vert\xb^{k+1} - \bar{\xb}^{k+1}\Vert_{\xb^k} \leq \delta_k$} for a given tolerance $\delta_k \geq 0$. As shown in  \cite{TranDinh2013e}, this condition is satisfied if
\vspace{-0.5ex}
\begin{equation*}
h_k(\xb^{k+1}) - h_k(\bar{\xb}^{k+1}) \leq 0.5\delta_k^2,
\vspace{-0.5ex}
\end{equation*}
where $h_k(\cdot)$ is the objective function of \eqref{eq:primal_cvx_subprob}.
This condition is different from Definition~\ref{de:approx_sol}, where we can check it by evaluating the objective values.

We redefine the following proximal Newton decrement using \eqref{eq:prox_oper} as follows:
\vspace{-0.5ex}
\begin{equation}\label{eq:PN_decrement2}
\lambda_t(\xb) := \big\Vert \xb - \prox_{(t\nabla^2f(\xb))^{-1}g}\big(\xb - \nabla^2{f}(\xb)\nabla{f}(\xb)\big)\big\Vert_{\xb}.
\vspace{-0.5ex}
\end{equation}
Although the scheme \eqref{eq:primal_cvx_subprob} has been studied in \cite{TranDinh2013e,TranDinh2015f}, the following features are new.
\begin{enumerate}
\item \textit{Phase 1: Finding initial point $\xb^0$}:
We solve the following auxiliary problem by applying  \eqref{eq:inexact_pf_scheme_phase1} to obtain an initial point $\xb^0\in\intx{\Xc}$ such that $\lambda_{t_0}(\xb^0) \leq \beta$:
\vspace{-0.5ex}
\begin{equation}\label{eq:auxi_prob_constr_cvx}
\min_{\xb}\set{ f(\xb) - \tau\iprods{\nabla{f}(\hat{\xb}^0) + t_0^{-1}\hat{\xi}^0, \xb} + t_0^{-1}g(\xb)},
\vspace{-0.5ex}
\end{equation}
where $\hat{\xb}^0\in\intx{\Xc}$ is an arbitrary initial point, and $\hat{\xi}^0 \in\partial{g}(\hat{\xb}^0)$.
The inexact proximal path-following scheme for solving \eqref{eq:auxi_prob_constr_cvx} rendering from \eqref{eq:inexact_pf_scheme_phase1} becomes
\vspace{-0.5ex}
\begin{equation}\label{eq:primal_cvx_subprob0}
\left\{\begin{array}{ll}
\tau_{j+1} &:= \tau_j - \bar{\Delta}_{j},\vspace{0.75ex}\\
\hat{\xb}^{j+1} &\approx \bar{\hat{\xb}}^{j+1} := \displaystyle\argmin_{\xb}\set{Q_f(\xb;\hat{\xb}^j) - \tau_{j+1}\iprods{\nabla{f}(\hat{\xb}^0) + t_0^{-1}\hat{\xi}^0, \xb} + t_0^{-1}g(\xb)}.
\end{array}\right.
\vspace{-0.5ex}
\end{equation}
\item \textit{New neighborhood of the central path}:
We choose $\beta \in (0, 0.329)$ which is approximately twice larger than $\beta\in (0, 0.15]$ as given in \cite{TranDinh2013e}.

\item \textit{Adaptive rule for $t$}: We can update $t_k$ in \eqref{eq:primal_cvx_subprob} adaptively using the value $\Vert\nabla{f}(\xb^k)\Vert_{\xb^k}^{\ast}$ as
\vspace{-0.5ex}
\begin{equation*}
t_{k+1} := (1-\sigma_k)t_k, ~~\text{where}~\sigma_k := \frac{c\sqrt{\beta} - \beta(1+c\sqrt{\beta})}{(1+c\sqrt{\beta})\Vert\nabla{f}(\xb^k)\Vert_{\xb^k}^{\ast} +c\sqrt{\beta}}\geq\bar{\sigma}_{\beta}.
\vspace{-0.5ex}
\end{equation*}

\item \textit{Implementable stopping condition:}
We can terminate Phase 1 using $\tau_j\Vert \nabla{f}(\hat{\xb}^0) + t_0^{-1}\hat{\xi}^0\Vert_{\hat{\xb}^j}^{\ast} \leq (\beta-\eta)$, which is implementable without incurring significantly computational cost.
\end{enumerate}
Let us denote this algorithmic variant by Algorithm~\ref{alg:A1}(b).
Instead of terminating this algorithmic variant with  $t_k \leq \frac{\varepsilon}{M_0}$, we  use $\Delta(\beta,\nu)t_k \leq \varepsilon$ to terminate Algorithm~\ref{alg:A1}(b), where $\Delta(\beta,\nu)$ is a function defined as in \cite[Lemma 5.1.]{TranDinh2013e}.
The following corollary provides the worst-case iteration-complexity of Algorithm~\ref{alg:A1}(b) as a direct consequence of Theorem \ref{th:overall_complexity}.

\begin{corollary}\label{co:pf_convergence2}
Let us choose $t_0 :=  \kappa$ defined  below the formula \eqref{eq:analytical_center}.
Then, the worst-case iteration-complexity of Algorithm~$\mathrm{\ref{alg:A1}(b)}$ to achieve an $\varepsilon$-solution $\xb^k$ of \eqref{eq:constr_cvx} such that $\xb^k\in\Xc$ and $g(\xb^k) - g^{\star} \leq \varepsilon$ is
\vspace{-0.5ex}
\begin{equation*}
\mathcal{O}\left( \frac{\kappa\Vert \nabla{f}(\hat{\xb}^0) + t_0^{-1}\hat{\xi}^0 \Vert_{\bar{\xb}^{\star}_f}^{\ast}}{t_0} + \sqrt{\nu}\ln\left(\frac{\Delta(\beta,\nu)t_0}{\varepsilon}\right) \right) ~~~~~\left(\text{or simpler}~~\mathcal{O}\left(\sqrt{\nu}\ln\left(\frac{\kappa \nu}{\varepsilon}\right)\right) \right).
\vspace{-0.5ex}
\end{equation*}
\end{corollary}
\vspace{-1ex}
We note that the worst-case iteration-complexity bound in Corollary \ref{co:pf_convergence2} is the overall iteration-complexity.
It is similar to the one given in \cite{TranDinh2015f} but the method is different.

\beforesubsec
\subsection{\bf Inexact dual path-following proximal  Newton algorithm}\label{subsec:dual_path_following}
\aftersubsec
We develop an inexact dual path-following scheme to solve \eqref{eq:constr_cvx2}, which works in the dual space.
For simplicity of presentation, we assume that $W = \Id$. Otherwise, we can use $\hat{g}(\cdot) := g(W(\cdot))$.
We first write the barrier formulation of  the dual problem  \eqref{eq:dual_prob1} as follows:
\vspace{-0.5ex}
\begin{equation*}
\min_{\yb \in\R^p}\set{ tf^{\ast}(\cb - L^{\ast}\yb) + g^{\ast}(\yb) + \iprods{\bb, \yb} },
\vspace{-0.5ex}
\end{equation*}
where $t > 0$ is a penalty parameter.
This problem can shortly read  as
\vspace{-0.5ex}
\begin{equation}\label{eq:smooth_dual_prob}
\phi^{\star}_t := \min_{\yb\in\R^p}\Big\{ \phi_t(\yb) := \varphi(\yb) + t^{-1}\psi(\yb) \Big\},
\vspace{-0.5ex}
\end{equation}
where $\varphi$ and $\psi$ are two convex functions defined by
\vspace{-0.5ex}
\begin{equation}\label{eq:smooth_dual_term}
\varphi(\yb) := f^{*}\left(\cb - L^{*}\yb\right),~~~\text{and}~~~\psi(\yb) := g^{*}(\yb) + \iprods{\bb, \yb}.
\vspace{-0.5ex}
\end{equation}
In order to characterize the relation between the primal problem \eqref{eq:constr_cvx2} and its dual form \eqref{eq:dual_prob1}, we formally impose the following  assumption.

\begin{assumption}\label{as:A2}
The objective function $g$ in \eqref{eq:constr_cvx2} is proper, closed and convex.
The linear operator $L : \R^n\to\R^p$ is full-row rank with $n \leq p$.
The following Slater condition holds:
\begin{equation*}
\left(\intx{\Kc}\times\mathrm{ri}(\dom{g}) \right)\cap\set{(\xb, \sb) \mid L\xb - \sb = \bb} \neq\emptyset.
\end{equation*}
In addition, $\Kc$ is a nonempty, closed, and pointed convex cone such that $\intx{\Kc}\neq\emptyset$, and $\Kc$ is  endowed with a $\nu$-self-concordant logarithmically homogeneous barrier $f$ with $\mathrm{Dom}(f) = \Kc$.
The solution set $\Sc^{\star}$ of \eqref{eq:constr_cvx2} is nonempty.
\end{assumption}
The following lemma shows that  \Shu{$\varphi(\cdot) := f^{\ast}( \cb-L^{*}(\cdot))$} remains a $\nu$-self-concordant barrier associated with the dual feasible set, while the scaled proximal operator of $\psi$ can be computed from the one of $g$.
The proof of this lemma is classical and is omitted, see \cite{Bauschke2011,Nesterov1994}.

\begin{lemma}\label{le:pros_of_varphi}
Under Assumption A.\ref{as:A2},  $\varphi(\cdot)$ defined by \eqref{eq:smooth_dual_term} is a $\nu$-self-concordant barrier of the dual feasible set $\mathcal{D}_{\Yc} := \set{\yb \in\R^p \mid  L^{*}\yb - \cb \in\Kc^{*}}$.
The proximal operator of $\psi$ defined in \eqref{eq:smooth_dual_prob}  is computed as $\prox_{\Qb \psi}(\yb) = \yb - \Qb\bb - \Qb \prox_{\Qb^{-1}g}\left(\Qb^{-1}\yb - \bb\right)$ for any $\Qb \in\Sc^p_{++}$.
\end{lemma}

Together with the primal local norm $\Vert\cdot\Vert_{\xb}$ given in Subsection~\ref{subsec:primal_path_following}, we also define a local norm with respect to $\varphi(\cdot)$ as $\Vert\ub\Vert_{\yb} := \iprods{\nabla^2{\varphi}(\yb)\ub,\ub}^{1/2}$ and its dual norm $\Vert\vb\Vert_{\yb}^{\ast} := \iprods{\nabla^2{\varphi}(\yb)^{-1}\vb,\vb}^{1/2}$ for a given $\yb\in\dom{\varphi}$.
Under Assumption~A.\ref{as:A2}, any primal-dual solution $(\xb^{\star}, \sb^{\star}) \in \Sc^{\star}$ and $\yb^{\star}\in\R^p$ of \eqref{eq:constr_cvx2} is also the KKT point of \eqref{eq:constr_cvx2} and vice versa, i.e.:
\begin{equation}\label{eq:kkt_cond}
\Shu{0\in L^{*}\yb^{\star} - \cb + \Nc_{\Kc}(\xb^{\star})}, ~~~  \yb^{\star} \in \partial{g}(\sb^{\star}), ~~~\text{and}~~ L\xb^{\star} - \sb^{\star} = \bb.
\end{equation}
In practice, we cannot solve \eqref{eq:constr_cvx2}-\eqref{eq:dual_prob1} (or equivalently, \eqref{eq:kkt_cond}) exactly to obtain an optimal solution $(\xb^{\star}, \sb^{\star}) \in\Sc^{\star}$ and $\yb^{\star} \in\R^p$ as indicated by the KKT condition \eqref{eq:kkt_cond}.
We can only find an $\varepsilon$-approximate solution $(\xb^{\star}_{\varepsilon}, \sb^{\star}_{\varepsilon})$ and $\yb^{\star}_{\varepsilon}$ as defined in Definition \ref{de:approx_sol2}.

\begin{definition}\label{de:approx_sol2}
\textit{
Given a tolerance $\varepsilon > 0$, we say that $(\xb^{\star}_{\varepsilon}, \sb^{\star}_{\varepsilon})$ is an $\varepsilon$-solution for \eqref{eq:constr_cvx2} associated with a dual solution $\yb^{\star}_{\varepsilon} \in \R^p$ of \eqref{eq:dual_prob1} if $\xb^{\star}_{\varepsilon} \in\intx{\Kc}$ and
\begin{equation*}
\left\{\begin{array}{ll}
\yb^{\star}_{\varepsilon} &  \in  \partial{g}(\sb^{\star}_{\varepsilon} ), \vspace{0.5ex}\\
\Vert L^{\ast}\yb^{\star}_{\varepsilon} - \cb\Vert_{\xb^{\star}_{\varepsilon}}^{\ast} &\leq \varepsilon,\vspace{0.5ex}\\
\Vert L\xb^{\star}_{\varepsilon} - \sb^{\star}_{\varepsilon}  - \bb\Vert_{\yb_{\varepsilon}}^{\ast} & \leq \varepsilon.
\end{array}\right.
\end{equation*}
}
\end{definition}
We note that  our path-following method always generates $\xb^{\star}_{\varepsilon} \in \intx{\Kc}$ which implies  $\xb^{\star}_{\varepsilon} \in \Kc$.

Next, we specify Algorithm \ref{alg:A1} to solve the dual problem \eqref{eq:dual_prob1} and provide a recovery strategy to obtain an $\varepsilon$-solution of \eqref{eq:constr_cvx2}.

\beforesubsubsec
\subsubsection{The inexact  path-following proximal Newton scheme for the dual}
\aftersubsubsec
By Lemma \ref{le:pros_of_varphi}, the function $\varphi$ defined by \eqref{eq:smooth_dual_term} is also self-concordant, and its gradient and Hessian-vector product  are given explicitly as
\begin{equation}\label{eq:deriv_of_varphi}
\nabla{\varphi}(\yb) = -L\nabla{f^{*}}(\cb - L^{*}\yb)~~~\text{and}~~~\nabla^2\varphi(\yb)\db = L\nabla^2{f^{*}}(\cb - L^{*}\yb)L^{*}\db.
\end{equation}
Let us denote by $Q_{\varphi}$ the quadratic surrogate of $\varphi$ at $\yb^k$ defined by \eqref{eq:cvx_subprob2b1}.
Under Assumption A.\ref{as:A2}, $\nabla^2{\varphi}$ is positive definite and hence, $Q_{\varphi}(\cdot; \yb^k)$ is strongly convex.
The main step of our inexact dual path-following proximal  Newton method can be presented as follows:
\begin{equation}\label{eq:dual_ipf_scheme}
\left\{\begin{array}{ll}
t_{k+1} &:= (1-\sigma_{\beta})t_k,\vspace{1ex}\\
\yb^{k+1} &\approx \bar{\yb}^{k\!+\!1} := \mathrm{arg}\!\!\displaystyle\min_{\yb\in\R^p}\set{ \phi(\yb; \yb^k) := Q_{\varphi}(\yb; \yb^k)  + t_{k\!+\!1}^{-1}\psi(\yb)},
\end{array}\right.
\end{equation}
where the approximation $\approx$ is defined as in Definition \ref{de:approx_sol} with a given tolerance $\delta_k$, and $\sigma_{\beta}\in(0, 1)$ is \Tianxiao{a }given factor.

The second line of \eqref{eq:dual_ipf_scheme} is  a composite convex quadratic minimization problem, which
has the same form as \eqref{eq:primal_cvx_subprob}.
To analyze \eqref{eq:dual_ipf_scheme}, we define
\begin{equation}\label{eq:PN_decrements}
\lambda_t(\yb)  :=  \Vert \yb - \Pc_{\yb}(\yb - \nabla^2{\varphi}(\yb)^{-1}\nabla{\varphi}(\yb); t) \Vert_{\yb},
\end{equation}
where $\Pc_{\yb}(\cdot; t) = \prox_{t^{-1}\nabla^2{\varphi}(\yb)^{-1}\psi}(\cdot)$ defined by \eqref{eq:Pc_oper}.

\beforesubsubsec
\subsubsection{Finding a starting point via an auxiliary problem}
\aftersubsubsec
Let us fix $t_0 > 0$ (e.g., $t_0 := \kappa$), and choose $\beta\in (0,1)$ such that $\Omega_{t_0}(\beta)$ defined in \eqref{eq:quad_converg_region} is a central path neighborhood of \eqref{eq:PN_decrements}.
The aim is to find a starting point $\yb^0 \in \Omega_{t_0}(\beta)$.
We again apply \eqref{eq:dual_ipf_scheme} to solve an auxiliary problem of \eqref{eq:smooth_dual_prob} for finding $\yb^0\in\Omega_{t_0}(\beta)$.

Given an arbitrary $\hat{\yb}^0\in \intx{\Dc_{\Yc}}$, let $\hat{\xi}_0\in\partial{\psi}(\hat{\yb}^0)$ be an arbitrary subgradient of $\psi$ at $\hat{\yb}^0$, and $\hat{\zeta}^0 := \nabla{\varphi}(\hat{\yb}_0) + t_0^{-1}\hat{\xi}_0$.
We consider the following auxiliary convex problem:
\begin{equation}\label{eq:auxi_prob}
\min_{\yb\in\R^p}\Big\{ \widehat{\phi}_{\tau}(\yb) := \varphi(\yb) - \tau \iprods{\hat{\zeta}^0, \yb} ~+~ t_0^{-1}\psi(\yb) \Big\},
\end{equation}
where $\tau \in [0, 1]$ is a given continuation parameter.

As seen before, when $\tau = 0$, \eqref{eq:auxi_prob} becomes \eqref{eq:smooth_dual_prob} at $t := t_0$,
\Shu{while with $\tau = 1$} we have $\nabla{\varphi}(\hat{\yb}^0) - \hat{\zeta}_0 = \nabla{\varphi}(\hat{\yb}^0) - \nabla{\varphi}(\hat{\yb}^0) - t_0^{-1}\hat{\xi}_0 =  -t_0^{-1}\hat{\xi}_0 \in - t_0^{-1}\partial{\psi}(\hat{\yb}^0)$.
Hence, $0 \in \nabla{\varphi}(\hat{\yb}^0) -  \hat{\zeta}_0  +  t_0^{-1}\partial{\psi}(\hat{\yb}^0)$, which implies that $\hat{\yb}^0$ is a solution of \eqref{eq:auxi_prob} at $\tau = 1$.

We customize \eqref{eq:dual_ipf_scheme} to solve \eqref{eq:auxi_prob} by tracking the path $\set{\tau_j}$ starting from $\tau_0 := 1$ such that $\set{\tau_j}$ converges to zero.
We use the index $j$ instead of $k$ to distinguish with Phase 2.

Given $\hat{\yb}^j\in\intx{\Dc_{\Yc}}$ and $\tau_j > 0$, similar to \eqref{eq:dual_ipf_scheme}, we update
\begin{equation}\label{eq:proximal_FP_scheme0}
{\!\!\!\!\!\!\!}\left\{\begin{array}{ll}
\tau_{j+1} &:= \tau_j - \Delta_j \vspace{1ex}\\
\hat{\yb}^{j+1} &\approx \bar{\hat{\yb}}^{j+1} := \mathrm{arg}\!\!\displaystyle\min_{\yb\in\R^p}\set{ \hat{\phi}(\yb; \hat{\yb}^j) := Q_{\varphi}(\yb; \hat{\yb}^j) - \tau_{j+1}\iprods{\hat{\zeta}^0, \yb} + t_0^{-1}\psi(\yb)},
\end{array}\right.{\!\!\!}
\end{equation}
where $\hat{\sigma}_{\eta}\in (0, 1)$, and $\approx$ is in the sense of Definition \ref{de:approx_sol} with a given tolerance $\delta_j$.

Since \eqref{eq:proximal_FP_scheme0} has the same form as \eqref{eq:dual_ipf_scheme} \Shu{applied} to \eqref{eq:auxi_prob}, we define
\begin{equation}\label{eq:PN_decrementj}
\hat{\lambda}_j  :=  \big\Vert\hat{\yb}_j -  \prox_{t_0^{-1}\nabla^2{\varphi}(\hat{\yb}^j)^{-1}\psi}\left(\hat{\yb}_j - \nabla^2{\varphi}(\hat{\yb}_j)^{-1}\nabla{\widehat{\varphi}}(\hat{\yb}_j)\right) \big\Vert_{\hat{\yb}_j},
\end{equation}
as  the dual proximal Newton decrement for \eqref{eq:proximal_FP_scheme0}.

\beforesubsubsec
\subsubsection{Primal solution recovery and the worst-case complexity}
\aftersubsubsec
Our next step is  to recover an approximate primal solution $(\xb^{\star}_{\varepsilon}, \sb_{\varepsilon}^{\star})$ of the primal problem \eqref{eq:constr_cvx} from the dual one $\yb^{\star}_{\varepsilon}$ of \eqref{eq:dual_prob1}.
The following theorem provides such a strategy whose proof can be found in Appendix \ref{apdx:th:primal_recovery}.
The notation $ \pi_{\partial{g^{\ast}}(\yb^{k})}(L\xb^{k} - \bb)$ stands for the projection of $L\xb^{k} - \bb$ onto $\partial{g^{\ast}}(\yb^{k})$ which is a nonempty, closed and convex set.

\begin{theorem}\label{th:primal_recovery}
Let $\set{(\yb^{k}, t_k)}$ be the sequence generated by \eqref{eq:dual_ipf_scheme}-\eqref{eq:proximal_FP_scheme0} to approximate a solution of the dual problem \eqref{eq:dual_prob1}.
Then, the point $(\xb^{k}, \sb^{k})$ computed by
\begin{equation}\label{eq:primal_sol}
\xb^{k} := \nabla{f^{*}}\big( t_{k}^{-1}(\cb - L^{*}\yb^{k} \big) \in \intx{\Kc}~~\text{and}~~\sb^{k} \Shu{= \pi_{\partial{g^{\ast}}(\yb^{k})}(L(\xb^k) - \bb)},
\end{equation}
together with $\yb^{k}$ satisfy the following estimate
\begin{equation}\label{eq:recovery}
\left\{\begin{array}{ll}
\yb^k & \in \partial{g}(\sb^k), \vspace{0.5ex}\\
\Vert L^{*}\yb^{k} - \cb\Vert_{\Shu{\xb^{k}}}^{*} &\leq \Shu{\sqrt{\nu}} t_{k}, \vspace{0.5ex}\\
\Vert L\xb^{k} - \sb^{k} - \bb\Vert_{\yb^k}^{\ast} &\leq \theta(c,\beta)t_{k},
\end{array}\right.
\end{equation}
where
\begin{equation}\label{eq:theta_c_beta}
\begin{array}{ll}
\theta(c,\beta) &:= \frac{(1-c^2)\beta}{(1+c\sqrt{\beta})^2\left[3c\sqrt{\beta} + c^2\beta + (1+c\sqrt{\beta})^3\right]-(1-c^2)\beta}
\vspace{1ex}\\
&+ \left(\frac{(1-c^2)\beta + c\sqrt{\beta}(1+c\sqrt{\beta})^2\left[3c\sqrt{\beta} + c^2\beta + (1+c\sqrt{\beta})^3\right]}{(1+c\sqrt{\beta})^2\left[3c\sqrt{\beta} + c^2\beta + (1+c\sqrt{\beta})^3\right]-(1-c^2)\beta}\right)^2 \leq 1,
\end{array}
\end{equation}
is a constant for fixed $c$ and $\beta$ chosen as in Lemma \ref{le:update_penalty_param}.

Consequently, if $\max\set{\sqrt{\nu}, \theta(c,\beta)} t_k = \sqrt{\nu} t_k \leq \varepsilon$, then $(\xb^{k}, \sb^{k})$ is an $\varepsilon$-solution to \eqref{eq:constr_cvx} in the sense of Definition \ref{de:approx_sol2} associated with the $\varepsilon$-dual solution $\yb^{k}$ of \eqref{eq:dual_prob1}.
\end{theorem}

\beforesubsubsec
\subsubsection{Two-phase inexact dual path-following proximal Newton  algorithm}
\aftersubsubsec
Now, we specify Algorithm~\ref{alg:A1} to solve \eqref{eq:constr_cvx2} using \eqref{eq:proximal_FP_scheme0} and \eqref{eq:dual_ipf_scheme} as in Algorithm \ref{alg:A1c}.

\begin{algorithm}[ht!]\caption{(\textit{Two-phase inexact dual path-following proximal Newton algorithm})}\label{alg:A1c}
\begin{normalsize}
\begin{algorithmic}[1]
   \STATE {\bfseries Initialization:}
   \STATE  \hspace{0.16cm} Choose $\hat{\yb}^0\in\R^p$ such that $\Lc^{\ast}\hat{\yb}^0 - \cb \in\intx{\Kc^{\ast}}$. Fix $t_0 :=\kappa$, and an accuracy $\varepsilon > 0$.
   \STATE  \hspace{0.16cm} Compute a vector $\hat{\xi}_0\in\partial{\psi}(\hat{\yb}^0)$ and evaluate $\nabla{\varphi}(\hat{\yb}^0)$.
   \STATE  \hspace{0.16cm} Set $\hat{\zeta}^0 := \nabla{\varphi}(\hat{\yb}^0) + t_0^{-1}\hat{\xi}^0$ and $\tau_0 := 1$.
   \STATE  \hspace{0.16cm} Choose $\beta$, $\eta$, then compute $\bar{\delta}_{\tau}$, $\bar{\mu}_{\eta}$, $\bar{\delta}_t$, $\bar{\sigma}_{\beta}$ as in Algorithm \ref{alg:A1}.
   Compute $\theta(c,\beta)$ by \eqref{eq:theta_c_beta}.
   \vspace{-0.5ex}
   \rule{0.88\textwidth}{0.1mm}
    \STATE \begin{center}\textbf{Phase 1: Computing an initial point}\end{center}
   \vspace{-1ex}
   \rule{0.88\textwidth}{0.1mm}
   \STATE {\bfseries For} $j = 0, \cdots, j_{\max}$, \textbf{perform:}
   \STATE\label{step:10}\hspace{0.16cm} If $\tau_j\Vert\hat{\zeta}^0\Vert_{\hat{\yb}^j}^{\ast} \leq (\beta-\eta)$, then TERMINATE.
   \STATE\label{step:main1d} \hspace{0.16cm}  Perform \eqref{eq:proximal_FP_scheme0} with $\Delta_j := \frac{\bar{\mu}_{\eta}}{\Vert\hat{\zeta}^0\Vert_{\hat{\yb}^j}^{\ast}}$ up to an accuracy $\hat{\delta}_j \leq \bar{\delta}_{\tau}$.
   \STATE {\bfseries End for}
    \\\vspace{-0.5ex}
    \rule{0.88\textwidth}{0.1mm}
    \STATE \begin{center}\textbf{Phase 2: Inexact dual  path-following proximal Newton iterations}\end{center}
   \vspace{-1ex}
   \rule{0.88\textwidth}{0.1mm}
   \STATE {\bfseries For} $k = 0, \cdots, k_{\max}$, \textbf{perform:}
   \STATE \hspace{0.16cm} If $\sqrt{\nu} t_k \leq \varepsilon$, then TERMINATE.
   \STATE\label{step:main2d} \hspace{0.16cm} Perform \eqref{eq:dual_ipf_scheme} up to an accuracy $\delta_k\leq\bar{\delta}_t$.
   \STATE {\bfseries End for}
\STATE\label{step:primal_recovery}\textbf{Primal recovery:}  Recover $(\xb^{k}, \sb^{k})$ from $\yb^{k}$ as in \eqref{eq:primal_sol}. Then, return $(\xb^k, \sb^k, \yb^k)$.
\end{algorithmic}
\end{normalsize}
\end{algorithm}

The main complexity-per-iteration of Algorithm \ref{alg:A1c} lies at Steps \ref{step:main1d} and \ref{step:main2d}, where we need to solve two composite and strongly convex quadratic programs in \eqref{eq:proximal_FP_scheme0} and \eqref{eq:dual_ipf_scheme}, respectively.
The primal solution recovery at Step~\ref{step:primal_recovery} does not significantly increase the computational cost of Algorithm \ref{alg:A1c}.
The worst-case iteration-complexity of Algorithm \ref{alg:A1c} remains the same as in Theorem \ref{th:overall_complexity} with $M_0 := \sqrt{\nu}$ and we do not restate it here.

\beforesec
\section{Preliminarily numerical experiments}\label{sec:num_exp}
\aftersec
We present three numerical examples to illustrate three algorithmic variants described in the previous sections, respectively.
We compare our methods with three common interior-point solvers: SDPT3 \cite{Toh2010}, SeDuMi \cite{Sturm1999}, and Mosek (a commercial software package).
We also compare our methods in the first two examples with the SDPNAL+5.0 solver in \cite{yang2015sdpnalp} which implemented a majorized semi-smooth Newton-CG augmented Lagrangian method.
Our numerical experiments are carried out in Matlab R2014b environment, running on a MacBook Pro Laptop (Retina, 2.7GHz Intel Core i5, 16GM Memory).

\beforesubsec
\subsection{\bf Example 1: Minimizing the maximum eigenvalue with constraints}\label{subsec:exam1}
\aftersubsec
We illustrate Algorithm \ref{alg:A1}(a) via the well-known maximum eigenvalue problem \cite{Nesterov2007d}:
\begin{equation}\label{eq:max_eigen}
\lambda^{\star}_{\max} := \min_{\yb\in\Yc}\set{ \lambda_{\max}\left(C + L\yb\right)},
\end{equation}
where $\lambda_{\max}(U)$ is the maximum eigenvalue of a symmetric matrix $U\in\Sc^n$, $C\in\Sc^n$ is a given matrix,  $L$ is a linear operator from $\R^p\to\Sc^n$, and $\Yc$ is a nonempty, closed and convex set in $\R^p$ endowed with a self-concordant barrier $\varphi(\cdot)$.

As a consequence of Von Neumann's trace inequality, we can show that $\lambda_{\max}(U) = \max_{\xb}\set{ \xb^{\top}U\xb \mid \norm{\xb}_2 = 1} = \max\set{\trace{UX} \mid \trace{X} = 1, X\in\Sc^n_{+}}$.
Hence, if we define $\Xc := \Sc^n_{+}$ and $g(X) := \delta_{\set{ X \mid \trace{X} = 1}}(X) - \trace{CX}$ and using $\iprods{L\yb,X} = \trace{(L\yb)X}$, then we can rewrite \eqref{eq:max_eigen} as \eqref{eq:saddle_point_prob} which is of the form:
\begin{equation}\label{eq:max_eigen_saddle_point}
\tilde{\lambda}^{\star}_{\max} := \min_{\yb\in\Yc}\big\{ \max_{X\in\Sc^n_{+}}\set{ \iprods{L\yb, X} - g(X) } \big\}.
\end{equation}
The corresponding barrier for $\Sc_{+}^n$ is $f(X) := -\log\det(X)$.

Now, we can apply Algorithm \ref{alg:A1}(a) to solve \eqref{eq:max_eigen_saddle_point}.
The main computation of this algorithm is the solution of \eqref{eq:cvx_subprob2b} and \eqref{eq:cvx_subprob2b_x0}, which can be written explicitly as follows for \eqref{eq:max_eigen_saddle_point}:
\vspace{-0.5ex}
\begin{equation}\label{eq:exam3_subprob}
\min_{\yb\in\R^p} \Big\{ \max_{X\in\Sc^n}\set{\iprods{L(y), X} - tQ_f(X;X^k) - g(X)} + tQ_{\varphi}(y; y^k) \Big\},
\vspace{-0.5ex}
\end{equation}
where $Q_f(X;X^k) :=  \iprods{\nabla{f}(X^k), X - X^k}  + \frac{1}{2}\iprods{\nabla^2f(X^k)(X-X^k), X-X^k}$ and $Q_{\varphi}(\yb; \yb^k) := \iprods{\nabla{\varphi}(\yb^k), \yb - \yb^k} + \frac{1}{2}\iprods{\nabla^2\varphi(\yb^k)(\yb - \yb^k), \yb - \yb^k}$.
We can solve \eqref{eq:exam3_subprob} in a closed form as follows:
\vspace{-0.5ex}
\begin{equation}\label{eq:exam3_subprob_sol}
\left\{\begin{array}{ll}
X^{\ast}_k &:= \mat{\left(\frac{\mathrm{trace}(\mat{H_k^{-1}h_k}) + 1}{\mathrm{trace}(\mat{H_k^{-1}\vec{\mathbb{I}})}}\right)H_k^{-1}\vec{\mathbb{I}} - H_k^{-1}h_k},\vspace{1ex}\\
\yb^{\ast}_k &:= \yb^k - \nabla^2{\varphi}(\yb^k)^{-1}\left( \nabla{\varphi}(\yb^k) + t^{-1}L^{\ast}X^{*}_k\right),
\end{array}\right.
\vspace{-0.5ex}
\end{equation}
where $H_k := \nabla^2{f}(X^k) + t^{-2}L\nabla^2{\varphi}(y^k)^{-1}L^{\ast} \succ 0$ and $h_k := \left[\nabla{f}(X^k) - \nabla^2{f}(X^k)\vec{X^k}\right] - t^{-1}L\left(\yb^k - \nabla^2\varphi(\yb^k)^{-1}\nabla{\varphi}(\yb^k)\right)  - t^{-1}\vec{C}$.

We consider a simple case, where $\Yc := \set{\yb\in\R^p \mid \norm{\yb}_{\infty}  \leq 1}$. Then, the barrier function of $\Yc$ is simply $\varphi(\yb) := -\sum_{i=1}^p\log(1 - \yb_i^2)$.
In this case, we can compute both $\nabla{\varphi}(\cdot)$ and $\nabla^2{\varphi}(\cdot)^{-1}$ in a closed form.
The barrier parameter for $F := f + \varphi$ is $\nu := 2p + n$.

We test $5$ solvers: Algorithm~\ref{alg:A1}(a), SDPT3, SeDuMi, Mosek and SDPNAL+5.0 on 10 medium-size problems, where $n$ varies from $5$ to $50$ and $p = 10n^2$ (varies from $250$ to $25,000$).
The linear operator $L$ and matrix $C$ are generated randomly using the standard Gaussian distribution \texttt{randn} in Matlab, which are completely dense.
In Phase 1 of Algorithm~\ref{alg:A1}(a), instead of performing a path-following scheme on the auxiliary problem, we simply perform a damped step variant on the original problem.
We set the initial penalty parameter $t_0 := 0.1$.
We terminate our algorithm if $t_k \leq 10^{-6}$ and $\tilde{\lambda}_k \leq 10^{-8}$.
When $t_k \leq 10^{-6}$, if $\tilde{\lambda}_k$ does not reach the $10^{-8}$ accuracy, we fix $t_k = 10^{-6}$ and perform at most $15$ addional iterations to decrease $\tilde{\lambda}_k$.
We terminate SDPT3, SeDiMi and Mosek with the same accuracy $\sqrt{\varepsilon} = 1.49\times10^{-8}$, where $\varepsilon$ is Matlab's machine precision.
We terminate SDPNAL+ using its default setting, but set the maximum number of iterations at $1000$.

The result and performance of these solvers are reported in Table~\ref{tbl:minmax_eigen}, where $n\times p$ is the size of $L$, \texttt{iter} is the number of iterations in Phase 1 and Phase 2, and $\lambda^{\star}_{-}$ is the reported objective value of each solver.
The most intensive computation of Algorithm~\ref{alg:A1}(a) is $L\mathrm{diag}(\nabla^2{\varphi}(\yb^k))^{-1}L^{\top}$, which costs from $40\%$ to $80\%$ the overall computational time.

\begin{table}[!ht]
\newcommand{\cell}[1]{{\!\!\!}#1{\!\!}}
\newcommand{\cellb}[1]{{\!\!\!}\color{blue}#1{\!\!}}
\rowcolors{2}{white}{black!15!white}
\begin{center}
\vspace{-3ex}
\caption{Summary of the result and performance of $5$ solvers for solving problem \eqref{eq:max_eigen}. }\label{tbl:minmax_eigen}
\vspace{-1ex}
\begin{scriptsize}
\begin{tabular}{ rr  | rrr | rr | rr |  rr | rr}
\toprule
\multicolumn{2}{c|}{Problem} & \multicolumn{3}{c|}{ Algorithm \ref{alg:A1}(a) } & \multicolumn{2}{c|}{SDPT3} &  \multicolumn{2}{c|}{SeDuMi} &  \multicolumn{2}{c|}{Mosek}  &  \multicolumn{2}{c}{SDPNAL+}  \\  \midrule
\cell{$n$} & \cell{$p$} & \cell{iter} & \cell{time[s]} & $\lambda^{\star}_{\mathrm{ours}}$ &  \cell{time[s]} & $\lambda^{\star}_{\mathrm{sdpt3}}$ & \cell{time[s]} & \cell{$\lambda^{\star}_{\mathrm{sedumi}}}$ & \cell{time[s]} & \cell{$\lambda^{\star}_{\mathrm{mosek}}}$  & \cell{time[s]} & \cell{$\lambda^{\star}_{\mathrm{sdpnal}}}$ \\ \midrule
\cell{5} & \cell{250} & \cell{19/40} & \cellb{0.37} & \cell{-80.34} & \cell{2.49} & \cell{-80.34} & \cell{1.32} & \cell{-80.34} & \cell{3.12} & \cell{-80.34} & \cell{27.53} & \cell{-80.34} \\
\cell{10} & \cell{1000} & \cell{32/60} & \cellb{0.30} & \cell{-255.92} & \cell{3.66} & \cell{-255.92} & \cell{1.56} & \cell{-255.92} & \cell{2.06} & \cell{-255.92} & \cell{12.25} & \cell{-255.92} \\
\cell{15} & \cell{2250} & \cell{41/72} & \cellb{1.27} & \cell{-453.51} & \cell{13.35} & \cell{-453.52} & \cell{4.95} & \cell{-453.52} & \cell{2.53} & \cell{-453.52} & \cell{36.99} & \cell{-453.52} \\
\cell{20} & \cell{4000} & \cell{49/81} & \cell{6.88} & \cell{-684.18} & \cell{55.58} & \cell{-684.18} & \cell{21.67} & \cell{-684.18} & \cellb{4.14} & \cell{-684.18} & \cell{38.38} & \cell{-684.19} \\
\cell{25} & \cell{6250} & \cell{57/81} & \cell{23.06} & \cell{-952.12} & \cell{309.29} & \cell{-952.12} & \cell{135.91} & \cell{-952.12} & \cellb{7.91} & \cell{-952.12} & \cell{315.46} & \cell{-952.13} \\
\cell{30} & \cell{9000} & \cell{65/86} & \cell{155.37} & \cell{-1265.71} & \cell{518.82} & \cell{-1265.71} & \cell{209.27} & \cell{-1265.71} & \cellb{14.53} & \cell{-1265.71} & \cell{1202.07} & \cell{-1257.17} \\
\cell{35} & \cell{12250} & \cell{71/104} & \cell{181.21} & \cell{-1582.48} & \cell{1262.64} & \cell{-1582.49} & \cell{494.84} & \cell{-1582.49} & \cellb{30.40} & \cell{-1582.49} & \cell{2912.92} & \cell{-1582.21} \\
\cell{35} & \cell{12250} & \cell{71/104} & \cell{181.21} & \cell{-1582.48} & \cell{1262.64} & \cell{-1582.49} & \cell{494.84} & \cell{-1582.49} & \cellb{30.40} & \cell{-1582.49} & \cell{2912.92} & \cell{-1582.21} \\
\cell{40} & \cell{16000} & \cell{78/110} & \cell{400.89} & \cell{-1931.65} & \cell{2795.90} & \cell{-1931.66} & \cell{1064.91} & \cell{-1931.66} & \cellb{68.78} & \cell{-1931.66} & \cell{2487.30} & \cell{-1925.06} \\
\cell{45} & \cell{20250} & \cell{84/117} & \cell{831.43} & \cell{-2322.09} & \cell{4777.96} & \cell{-2322.12} & \cell{2052.40} & \cell{-2322.12} & \cellb{77.91} & \cell{-2322.11} & \cell{1840.10} & \cell{-2322.91} \\
\cell{50} & \cell{25000} & \cell{89/125} & \cell{1367.36} & \cell{-2694.27} & \cell{9474.14} & \cell{-2694.29} & \cell{4184.44} & \cell{-2694.29} & \cellb{130.61} & \cell{-2694.29} & \cell{13948.03} & \cell{-2696.67} \\
\bottomrule
\end{tabular}
\end{scriptsize}
\vspace{-7ex}
\end{center}
\end{table}

In this test, Mosek is the fastest when the size is increased while SDPNAL+~is the slowest. SDPT3 is slow but is slightly better than SDPNAL+ in this test.
Our algorithm produces nearly optimal objective value while requires reasonable computational time compared to the other solvers.
It slightly gives a lower objective value in some problems, but also violates the bound constraint $\Vert\yb\Vert_{\infty} \leq 1$.
We emphasize that our algorithm is naively implemented in Matlab without optimizing the code or using mex files as other solvers.
Mosek is a well-known commercial software implemented in C++ using several advanced heuristic strategies.

\beforesubsec
\subsection{\bf Example 2: Sparse and low-rank matrix approximation}\label{subsec:exam2}
\aftersubsec
The problem of approximating a given $(n\times n)$-symmetric matrix $M$ as the sum of  a low-rank positive semidefinite matrix $X$ with bounded magnitudes, and a sparse matrix $M-X$ can be formulated into the following convex optimization problem (see \cite{TranDinh2013e}):
\begin{equation}\label{eq:rpca2}
\left\{\begin{array}{ll}
\displaystyle\min_{X} &  \rho\Vert\vec{X - M}\Vert_1 + (1-\rho)\trace{X} \vspace{1ex}\\
\mathrm{s.t.} & X\succeq 0, ~~L_{ij} \leq X_{ij}\leq U_{ij}, ~1\leq i < j \leq n.
\end{array}\right.
\end{equation}
Here, $\rho \in (0, 1)$ is a regularization parameter, and $L$ and $U$ are the lower and upper bounds.

Let us define $\Xc := \Sc^n_{++}$ and $g(X) = \rho\Vert\vec{X - M}\Vert_1 + (1-\rho)\trace{X} + \delta_{[L, U]}(X)$, where $\delta_{[L, U]}$ is the indicator of $[L, U] := \set{X\in\Sc^n\mid L_{ij} \leq X_{ij}\leq U_{ij}, ~1\leq i < j \leq n}$.
Then, we can formulate \eqref{eq:rpca2} into the form \eqref{eq:constr_cvx} with $f(X) := -\log\det(X)$.

We implement Algorithm~\ref{alg:A1}(b) for solving \eqref{eq:constr_cvx} and compare it with Mosek and SDPNAL+0.5 \cite{yang2015sdpnalp}.
The initial parameter is set to $t_0 := 0.1$.
We use a restarting accelerated proximal-gradient algorithm proposed in \cite{Su2014} to solve the subproblems \eqref{eq:primal_cvx_subprob} and \eqref{eq:primal_cvx_subprob0} with at most $150$ iterations.
We apply the same strategy as in Subsection~\ref{subsec:exam1} to terminate this algorithm.
For Mosek and SDPNAL+, we use their default configuration.

We test  three algorithms on $12$ problems with the size reported in Table~\ref{tbl:lowrank_app}. We limit our test to $n= 240$ since Mosek can only solve up to this size in our computer.
The data is generated as follows.
We generate a symmetric matrix $M_0$  using standard Gaussian distribution with the rank of $\lfloor 0.25n \rfloor$ and the sparsity of $25\%$.
Then, we add a sparse Gaussian noise $E$ with the sparsity of $10\%$ and the variance of $10^{-4}$ as $M := M_0 + E$ to obtain $M$.
We generate the lower bound $L_{ij} := 0.9\min\set{ M_{ij} \mid 1\leq i < j\leq n }$ and the upper bound $U_{ij} := 1.1\max\set{ M_{ij} \mid 1\leq i < j\leq n}$.
We choose $\rho = 0.2$ for all the problems.

The performance and result of three algorithms are reported in Table~\ref{tbl:lowrank_app}.
Here, \texttt{iter} is the number of iterations for Phase 1 and Phase 2 of Algorithm~\ref{alg:A1}(b);
\texttt{time} is the computational time in second;
$g(X^k)$ (resp., $g^{\star}_{\mathrm{sdpn}}$ and $g^{\star}_{\mathrm{mosek}}$) is the objective value of  Algorithm~\ref{alg:A1}(b)) (resp., SDPNAL+ and Mosek);
\texttt{spr/rank} is the sparsity level of $X^k - M$ (e.g., $\texttt{spr} := \texttt{nnz}(X^k-M)/n^2$) and the rank of $X^k$ (rounding up to $10^{-6}$ accuracy);
and $\texttt{error} := \tfrac{\Vert X^k - X^k_{\mathrm{mosek}}\Vert_F}{\max\set{\Vert X^k_{\mathrm{mosek}}\Vert_F, \Vert X^k\Vert_F}}$ is the relative error between the solution of  Algorithm~\ref{alg:A1}(b) and  SDPNAL+ to the solution of Mosek. Here, $\norm{\cdot}_F$ is the Frobenius norm.
\begin{table}[!ht]
\newcommand{\cell}[1]{{\!\!\!}#1{\!\!\!}}
\newcommand{\cells}[1]{{\!\!\!\!}#1{\!\!\!}}
\newcommand{\cella}[1]{{\!\!\!\!\!}#1{\!\!\!\!}}
\newcommand{\cellb}[1]{{\!\!}{\color{blue}#1}{\!\!\!}}
\rowcolors{2}{white}{black!15!white}
\begin{center}
\vspace{-4ex}
\caption{Summary of the result and performance of $3$ solvers on $12$ problem instances}\label{tbl:lowrank_app}
\vspace{-2ex}
\begin{scriptsize}
\begin{tabular}{ r  | rrrrr | rrrrr | rrr }
\toprule
\multicolumn{1}{c}{} & \multicolumn{5}{c|}{ Algorithm \ref{alg:A1}(b) } & \multicolumn{5}{c|}{SDPNAL+} &  \multicolumn{3}{c}{Mosek}  \\  \midrule
\cell{$n$} & \cell{iter} & \cell{time} & $g(X^k)$ & \cell{spr/rank} & \cell{error} & \cell{iter} & \cell{time} & \cell{\!\!$g^{\star}_{\mathrm{sdpnal+}}$} & \cell{spr/rank} & \cell{error}  & \cell{time} & \cell{$g^{\star}_{\mathrm{mosek}}}$ & \cell{spr/rank} \\ \midrule
\cell{20} & \cell{11/74} & \cellb{0.39} & \cell{11.37} & \cellb{0.39/13} & \cell{4.14e-01} & \cell{294} & \cell{0.91} & \cell{11.37} & \cell{0.73/18} & \cell{3.28e-01} & \cell{0.13} & \cell{11.37} & \cell{0.46/4} \\
\cell{40} & \cell{16/101} & \cellb{0.93} & \cell{57.06} & \cellb{0.36/32} & \cell{1.62e-03} & \cell{500} & \cell{4.54} & \cell{57.06} & \cell{0.82/39} & \cell{1.42e-04} & \cell{0.65} & \cell{57.06} & \cell{0.60/19} \\
\cell{60} & \cell{23/124} & \cell{2.16} & \cell{162.75} & \cellb{0.27/53} & \cell{2.30e-03} & \cell{350} & \cellb{1.93} & \cell{162.74} & \cell{0.96/59} & \cell{2.81e-03} & \cell{3.07} & \cell{162.75} & \cell{0.89/44} \\
\cell{80} & \cell{25/147} & \cell{2.36} & \cell{245.39} & \cellb{0.21/73} & \cell{1.19e-05} & \cell{210} & \cellb{1.62} & \cell{245.38} & \cell{0.88/79} & \cell{1.70e-05} & \cell{8.99} & \cell{245.39} & \cell{0.85/53} \\
\cell{100} & \cell{29/166} & \cell{3.24} & \cell{427.29} & \cellb{0.18/92} & \cell{9.91e-06} & \cell{216} & \cellb{1.71} & \cell{427.27} & \cell{0.90/100} & \cell{1.21e-05} & \cell{23.32} & \cell{427.29} & \cell{0.79/66} \\
\cell{120} & \cell{33/184} & \cell{3.71} & \cell{662.17} & \cellb{0.14/113} & \cell{1.08e-05} & \cell{211} & \cellb{2.66} & \cell{662.14} & \cell{0.89/119} & \cell{7.06e-06} & \cell{57.58} & \cell{662.17} & \cell{0.89/87} \\
\cell{140} & \cell{35/200} & \cell{4.45} & \cell{893.47} & \cellb{0.12/133} & \cell{9.89e-06} & \cell{191} & \cellb{2.74} & \cell{893.44} & \cell{0.90/139} & \cell{6.06e-06} & \cell{141.89} & \cell{893.47} & \cell{0.82/96} \\
\cell{160} & \cell{38/216} & \cell{5.98} & \cell{1185.71} & \cellb{0.10/154} & \cell{1.01e-05} & \cell{193} & \cellb{3.82} & \cell{1185.67} & \cell{0.90/159} & \cell{5.28e-06} & \cell{266.47} & \cell{1185.71} & \cell{0.80/127} \\
\cell{180} & \cell{41/244} & \cell{8.01} & \cell{1493.43} & \cellb{0.10/173} & \cell{1.09e-05} & \cell{191} & \cellb{3.89} & \cell{1493.38} & \cell{0.90/179} & \cell{4.12e-06} & \cell{542.29} & \cell{1493.43} & \cell{0.89/169} \\
\cell{200} & \cell{43/272} & \cell{10.97} & \cell{1741.32} & \cellb{0.08/196} & \cell{1.32e-05} & \cell{194} & \cellb{4.30} & \cell{1741.26} & \cell{0.90/199} & \cell{4.23e-06} & \cell{1049.39} & \cell{1741.33} & \cell{0.92/194} \\
\cell{220} & \cell{46/297} & \cell{11.52} & \cell{2082.55} & \cellb{0.06/214} & \cell{1.22e-05} & \cell{191} & \cellb{7.13} & \cell{2082.47} & \cell{0.90/218} & \cell{4.70e-06} & \cell{1374.15} & \cell{2082.55} & \cell{0.86/218} \\
\cell{240} & \cell{49/329} & \cell{15.31} & \cell{2577.35} & \cellb{0.04/236} & \cell{1.41e-05} & \cell{194} & \cellb{6.21} & \cell{2577.25} & \cell{0.90/237} & \cell{3.32e-06} & \cell{2644.77} & \cell{2577.36} & \cell{0.91/219} \\
\bottomrule
\end{tabular}
\end{scriptsize}
\vspace{-6ex}
\end{center}
\end{table}

As we can see from Table~\ref{tbl:lowrank_app}, three solvers give similar results in terms of the objective value and the approximate solution $X^k$ (shown by \texttt{error}).
But both Algorithm~\ref{alg:A1}(b) and Mosek perfectly satisfy the positive definiteness constraint with $\lambda_{\min}(X^k) > 0$, while SDPNAL+ still slightly violates this constraint with $\lambda_{\min}(X^k) \leq -\mathcal{O}(10^{-4})$.
 Algorithm~\ref{alg:A1}(b) is slightly slower than SDPNAL+ in terms of time, but it is much faster than Mosek.
Algorithm~\ref{alg:A1} gives best results in terms of the sparsity of $X^k - M$, while achieves similar rank as Mosek and SDPNAL+ in the majority of the test.

\beforesubsec
\subsection{\bf Example 3: Cluster recovery}
\aftersubsec
Finally, we test Algorithm \ref{alg:A1c} for solving the following well studied clustering recovery problem via an SDP relaxation as studied in \cite{hajek2015achieving}:
\begin{equation}\label{eq:cluster_recovery}
\left\{\begin{array}{ll}
\displaystyle\max_{X\in\Sc^n_{+}} & \trace{A^{\top}X}\\
\text{s.t.} & X_{ii}\leq 1,~~ X_{ij} \geq 0, ~i,j=1,\cdots, n,\vspace{0.75ex}\\
& \trace{X} = s_1,~\trace{E_nX} = s_2,
\end{array}\right.
\end{equation}
where $A$ is the adjacency matrix of a given graph, $E_n$ is the all-one matrix in $\R^{n\times n}$, $s_1 = \sum_{i=1}^rK_i$, and $s_2 := \sum_{i=1}^rK_i^2$ with $K_1, K_2,\cdots, K_r$ being the size of $r$ clusters.

Let us define $\Kc := \Sc^n_{+}$, $LX := [\trace{X}, \trace{E_nX}, X_{ii}, X_{ij}] : \R^{n\times n} \to \R^{n(n+1)+2}$, $g(\sb) := \delta_{\set{0}^2}(\sb_{1:2}) + \delta_{-\R_{+}^n}(\sb_{3:n+2}) + \delta_{\R^{n^2}_{+}}(\sb_{n+3:n(n+1)+2})$, and  $\bb := (s_1, s_2, \eb_n, \boldsymbol{0}^{n^2})^{\top}\in \R^{n(n+1)+2}$, where $\eb_n$ is the all-one vector, $\delta_{\Xc}$ is the indicator function of $\Xc$, and $\sb_{k_1:k_2}$ is the subvector of $\sb$ concatenating  from the $k_1$-th entry to the $k_2$-entry.
Using these notations, we can reformulate \eqref{eq:cluster_recovery} into the constrained convex problem \eqref{eq:constr_cvx2}.

Since $\Kc := \Sc^n_{+}$ is a self-dual cone, i.e., $\Kc^{\ast} = \Sc^n_{+}$, and the corresponding self-concordant logarithmically homogeneous barrier of $-\Kc^{\ast}$ is  $f^{\ast}(S) := -\log\det(-S) - n$.

We implement Algorithm~\ref{alg:A1c} to solve \eqref{eq:cluster_recovery}.
We use a restarting proximal-gradient algorithm \cite{Beck2009} to solve the corresponding subproblems in \eqref{eq:dual_ipf_scheme} and \eqref{eq:proximal_FP_scheme0}.
Since we can compute the dual objective values, we use a damped step Newton scheme in Phase 1 to compute an initial point $\yb^0$.
Since our algorithm uses the first order method for solving the subproblems, we set the precision of those solvers above to be low such that the relative error is guaranteed to be less than $\mathcal{O}(10^{-4})$ when terminates.
We choose the number of clusters such that the average number of points of each cluster is between $10$ and $20$.
The initial value of $t_0$ is set to $t_0 := 0.5$ if $n \leq 120$, and to $t_0 := 1.0$, otherwise.
All optimal values of our algorithm have the relative error smaller than $10^{-4}$ when the algorithm is terminated which matches the precision of the above solvers.
The results and performance of these solvers are reported in Table~\ref{tbl:cluster} for small-scale problems.
In this table, we summarize the result and performance of $12$ problems of the size from $60$ to $250$, where $K$ is the number of clusters; \texttt{iter} is the number of iterations in Phase 1 and Phase 2 of Algorithm~\ref{alg:A1c}; and \texttt{spu} is the speed up ratio (i.e., $\texttt{spu} = \frac{\text{time}_{-}}{\text{time}_{\text{Algorithm~\ref{alg:A1c}}}}$) in terms of time compared to other solvers.

\begin{table}[ht!]
\vspace{-3ex}
\newcommand{\cell}[1]{{\!\!\!}#1{\!\!}}
\newcommand{\cellf}[1]{{\!\!\!}#1{\!\!\!}}
\newcommand{\cells}[1]{{\!\!\!\!}#1{\!\!\!}}
\newcommand{\cella}[1]{{\!\!\!\!\!}#1{\!\!\!\!}}
\newcommand{\cellb}[1]{{\!\!}{\color{blue}#1}{\!\!\!}}
\rowcolors{2}{white}{black!15!white}
\begin{center}
\caption{Summary of performance of $4$ solvers on $12$ problem instances}\label{tbl:cluster}
\vspace{-1ex}
\begin{scriptsize}
\begin{tabular}{ rrr| rrr | rrr | rrr |  rrr}
\toprule
\multicolumn{3}{c|}{Problem} & \multicolumn{3}{c|}{ Algorithm \ref{alg:A1c} } & \multicolumn{3}{c|}{SDPT3} &  \multicolumn{3}{c|}{SeDuMi} &  \multicolumn{3}{c}{Mosek}  \\  \midrule
\cell{$n$} & \cell{$K$} & \cell{$\rho[\%]$} & \cell{iter} & \cell{\!time[s]\!} & \cell{$\mathcal{G}_k$} &  \cell{time[s]} & \cell{$\mathcal{G}^{\star}_{\mathrm{sdpt3}}$} & \cell{spu} & \cell{time[s]} & \cell{$\mathcal{G}^{\star}_{\mathrm{sedumi}}}$ & \cell{spu} & \cell{time[s]} & \cell{$\mathcal{G}^{\star}_{\mathrm{mosek}}}$ & \cell{spu}\\
\cellf{60} & \cell{5} & \cell{18.6} & \cell{52/215} & \cellb{3.63} & \cell{660.07} & \cell{6.28} & \cell{660.00} & \cell{1.7} & \cell{11.57} & \cell{660.00} & \cell{3.2} & \cell{4.41} & \cell{660.00} & \cell{1.2} \\
\cellf{70} & \cell{7} & \cell{13.0} & \cell{42/232} & \cellb{3.37} & \cell{630.02} & \cell{9.27} & \cell{630.00} & \cell{2.8} & \cell{26.56} & \cell{630.04} & \cell{7.9} & \cell{6.62} & \cell{630.00} & \cell{2.0} \\
\cellf{80} & \cell{8} & \cell{11.3} & \cell{45/247} & \cellb{4.30} & \cell{720.02} & \cell{14.46} & \cell{720.00} & \cell{3.4} & \cell{71.58} & \cell{720.00} & \cell{16.7} & \cell{8.39} & \cell{720.00} & \cell{2.0} \\
\cellf{90} & \cell{9} & \cell{10.1} & \cell{48/262} & \cellb{6.61} & \cell{810.03} & \cell{27.80} & \cell{810.00} & \cell{4.2} & \cell{127.53} & \cell{810.01} & \cell{19.3} & \cell{14.40} & \cell{810.00} & \cell{2.2} \\
\cellf{100} & \cell{10} & \cell{9.1} & \cell{50/276} & \cellb{7.54} & \cell{900.03} & \cell{44.24} & \cell{900.00} & \cell{5.9} & \cell{237.90} & \cell{900.01} & \cell{31.5} & \cell{23.45} & \cell{900.00} & \cell{3.1} \\
\cellf{110} & \cell{10} & \cell{9.2} & \cell{57/289} & \cellb{8.91} & \cell{1100.05} & \cell{57.26} & \cell{1100.00} & \cell{6.4} & \cell{378.96} & \cell{1100.02} & \cell{42.5} & \cell{33.01} & \cell{1100.00} & \cell{3.7} \\
\cellf{120} & \cell{10} & \cell{9.2} & \cell{78/302} & \cellb{11.13} & \cell{1320.13} & \cell{88.61} & \cell{1320.00} & \cell{8.0} & \cell{839.62} & \cell{1320.02} & \cell{75.4} & \cell{50.80} & \cell{1320.00} & \cell{4.6} \\
\cellf{140} & \cell{10} & \cell{9.4} & \cell{57/359} & \cellb{20.18} & \cell{1820.13} & \cell{203.74} & \cell{1820.00} & \cell{10.1} & \cell{1347.69} & \cell{1820.35} & \cell{66.8} & \cell{110.82} & \cell{1820.00} & \cell{5.5} \\
\cellf{160} & \cell{10} & \cell{9.4} & \cell{76/383} & \cellb{26.57} & \cell{2400.23} & \cell{340.58} & \cell{2400.00} & \cell{12.8} & \cell{2899.46} & \cell{2400.03} & \cell{109.1} & \cell{214.13} & \cell{2400.00} & \cell{8.1} \\
\cellf{180} & \cell{15} & \cell{6.2} & \cell{63/406} & \cellb{36.73} & \cell{1980.16} & \cell{808.40} & \cell{1980.00} & \cell{22.0} & \cell{6859.91} & \cell{1980.07} & \cell{186.8} & \cell{423.28} & \cell{1980.00} & \cell{11.5} \\
\cellf{200} & \cell{20} & \cell{4.5} & \cell{64/428} & \cellb{66.34} & \cell{1800.14} & \cell{1945.14} & \cell{1798.69} & \cell{29.3} & \cell{14023.69} & \cell{1800.04} & \cell{211.4} & \cell{787.78} & \cell{1800.00} & \cell{11.9} \\
\cell{250} & \cell{25} & \cell{3.6} & \cell{72/478} & \cellb{184.18} & \cell{2250.21} & \cell{11378.33} & \cell{2250.00} & \cell{61.8} & \cell{53417.68} & \cell{2250.06} & \cell{290.0} & \cell{3415.16} & \cell{2250.00} & \cell{18.5} \\
\bottomrule
\end{tabular}
\end{scriptsize}
\end{center}
\vspace{-4ex}
\end{table}

Table~\ref{tbl:cluster} shows that Algorithm~\ref{alg:A1c} can achieve the same order of accuracy as other three solvers while outperforms them in terms of computational time.
We can speed up Algorithm~\ref{alg:A1c} up to $18$ times compared to Mosek, $62$ times over SDPT3 and $290$ times over SeDuMi.
This is due to the low cost computation of the projections when we work directly on the dual of the original problem.
The other solvers require to convert it to an appropriate SDP format which increases substantially the problem size as seen from Subsection~\ref{subsec:exam2}.

\beforesec
\section{Discussion}\label{sec:discussion}
\aftersec
We have studied a class of self-concordant inclusions of the form \eqref{eq:mono_inclusion}, and designed an inexact generalized Newton-type framework for solving it.
Problem \eqref{eq:mono_inclusion} is sufficiently general to cope with three fundamental convex optimization formulations  indicated in Subsection \ref{subsec:motivating_exam}.
Moreover, since this problem can \Shu{be} reformulated  into a multivalued variational inequality problem \cite{Rockafellar1997}, theory and methods from this area can be used to \Shu{deal with  \eqref{eq:mono_inclusion}, see \cite{Bauschke2011,Facchinei2003,Rockafellar1997}.

Most existing methods for solving  \eqref{eq:mono_inclusion} exploit specific structures of $\Ac$ and $\Zc$.
When $\Ac$ is single-valued, \eqref{eq:mono_inclusion} is a standard single-valued variational inequality, and it becomes a complementarity problem if additionally $\Zc$ is a box.
The most commonly used methods to solve complementarity problems are based on generalized Newton methods developed for nonsmooth equations, including the path following methods \cite{Dirkse1995,Ralph1994} and semi-smooth Newton methods \cite{Sun2002,Deluca1996,Kummer1988,Qi1993,yang2015sdpnalp}.
The basic idea is to reformulate the complementarity problem as an equation defined by a nonsmooth function, and at each iteration, we approximately solve the equation obtained by some first order approximation or a generalized Jacobian matrix of the nonsmooth function. Another important class of methods to solve \eqref{eq:mono_inclusion} are based on projection and splitting   \cite{Tseng1991a,Tseng1997,xiu2003some}.
These methods can be considered as special cases of the forward-backward splitting scheme when the second operator is simply the normal cone of the convex set $\Zc$.
When $\Ac$ is maximally monotone, its resolvent is well-defined and single-valued. Splitting methods using  proximal-point and projected schemes such as Douglas-Rachford's methods can be applied to solve \eqref{eq:mono_inclusion}.
Other approaches such as augmented Lagrangian \cite{yang2015sdpnalp}, extragradient, mirror descent, hybrid-gradient, gap functions, smoothing techniques and interior-point proximal methods are also widely studied in the literature for different classes of \eqref{eq:mono_inclusion}, see, e.g., \cite{auslender1999logarithmic,Bonnans1994a,Esser2010a,Facchinei2003,Fukushima1992a,Goldstein2013,Korpelevic1976,Monteiro2012c,Nemirovskii2004,Nesterov2007a,solodov1999hybrid,Wright1996} and the references quoted therein.}

From a theoretical viewpoint, the setting \eqref{eq:mono_inclusion} can be used as a unified tool to handle a wide range of convex problems.
Three specific instances \eqref{eq:constr_cvx}, \eqref{eq:constr_cvx2} and \eqref{eq:saddle_point_prob} of \eqref{eq:mono_inclusion} are well studied and have a great impact in different fields including operations research, statistics, machine learning, signal and image processing and controls \cite{Bauschke2011,BenTal2001,Boyd2004,Nesterov2004}.
Methods for solving these instances include sequential quadratic programming \cite{Nocedal2006}, interior-points \cite{Nesterov1994}, augmented Lagrangian-type methods \cite{Wen2010a,yang2015sdpnalp} (e.g., implemented in SDPAD, SDPNAL/SDPNAL), first order/second order  primal-dual and splitting methods \cite{Bauschke2011,Boyd2011,Chambolle2011,Combettes2005,Eckstein1992,Shefi2014,TranDinh2012c,Tseng1991a}, Frank-Wolfe-type algorithms \cite{Frank1956,Jaggi2013}, and stochastic gradient descents \cite{Nemirovski2009,johnson2013accelerating}, just to name a few.

Perhaps, the interior-point method \cite{BenTal2001,Nesterov2004,Nesterov1997,Yamashita2012} is one of the most developed methods for solving standard conic programs covered by \eqref{eq:constr_cvx} and \eqref{eq:constr_cvx2}.
Interior-point methods together with disciplined programming approach \cite{Grant2006} allow us to solve a large class of convex optimization problems arising in different fields.
These techniques have been systematically implemented in several off-the-shelf software packages such as CVX \cite{Grant2006}, YALMIP \cite{Loefberg2004}, CPLEX and Gurobi for both commercial and academic use. While interior-point methods provide a powerful framework to solve a large class of constrained convex problems with high accuracy and numerically robust performance, their high complexity-per-iteration disadvantage prevents them from solving large-scale applications in modern applications.

Although the interior-point method and the proximal-type method have been separately well developed for several classes of convex problems, their joint treatment was first proposed in \cite{TranDinh2013e,TranDinh2015f} to the best of our knowledge.
In these works, the authors proposed a novel path-following proximal Newton framework for the instance \eqref{eq:constr_cvx} of \eqref{eq:mono_inclusion}.
They characterized the $\mathcal{O}(\sqrt{\nu}\log(1/\varepsilon)$-worst-case iteration-complexity as in standard path-following method \cite{Nesterov2004} to achieve an $\varepsilon$-solution of \eqref{eq:constr_cvx}, where $\nu$ is the barrier parameter of the barrier function of the feasible set $\Xc$, and $\varepsilon$ is the desired accuracy.
However, \cite{TranDinh2013e,TranDinh2015f} obtained a smaller neighborhood for the central path compared to standard path-following methods \cite{Nesterov2004}.
In addition, these algorithms used the points on the central path to measure the proximal Newton decrement which leads to a unimplementable stopping criterion.
In contrast,  this paper focuses on developing a unified theory using self-concordant inclusion \eqref{eq:mono_inclusion} as a generic framework.
The main component of our methods is the generalized Newton method studied, e.g., in \Shu{\cite{Bonnans1994a,Pang1991,Robinson1980,Robinson1994}}, but we have extended it to self-concordant settings.
Moreover, we  use a generalized gradient mapping to measure a neighborhood of the central path as well as a quadratic convergence region of the generalized Newton iterations, and this neighborhood has the same size as in standard path-following methods \cite{Nesterov2004}. When this framework is specified to solve \eqref{eq:constr_cvx} and \eqref{eq:constr_cvx2}, such a generalized gradient mapping  allows us to obtain an implementable stopping criterion, an adaptive rule for the penalty parameter, and the overall polynomial time worst-case iteration-complexity bounds.

\vspace{-2ex}
\begin{acknowledgements}
This work was supported in part by the NSF grant, no. DMS-16-2044 (USA).
\end{acknowledgements}

\beforesec
\section{Appendix: The proofs of technical results}\label{sec:appendix}
\aftersec
\vspace{1ex}
This appendix provides the full proofs of all lemmas and theorems in the main text.

\beforesubsec
\subsection{\bf The proof of Lemma \ref{le:existence_sol_m}: The existence and uniqueness of the solution of \eqref{eq:barrier_mono_inc}.}\label{apdx:le:existence_sol_m}
\aftersubsec
\ShuR{Under Assumption A.\ref{as:A0},  the operator $t\nabla{F}(\cdot) + \Ac(\cdot)$ is maximally monotone for any $t > 0$.
We use \cite[Theorem 12.51]{Rockafellar1997} to prove the solution existence of \eqref{eq:barrier_mono_inc}.

To this end, let $\boldsymbol{\omega}\ne \boldsymbol{0}$ be chosen from the horizon cone of $\intx \Zc\cap\dom{\Ac}$.
We need to find $\zb\in \intx \Zc \cap\dom{\Ac}$ with $\vb\in t \nabla{F}(\zb) + \Ac(\zb)$ such that $\iprods{\vb, \boldsymbol{\omega}} > 0$.
By assumption,  there exists $\hat{\zb}\in \intx{\Zc}\cap\dom{\Ac}$ with $\hat{\ab}\in \Ac(\hat{\zb})  $ such that $\iprods{\hat{\ab}, \boldsymbol{\omega}}>0$.

First, we show that $\zb_\tau = \hat{\zb}+ \tau \boldsymbol{\omega}$ belongs to $\intx{\Zc}\cap\dom{\Ac}$ for any $\tau>0$. To see this, note that the assumption $\intx{\Zc}\cap\dom{\Ac}\ne \emptyset$ implies that $\intx{\Zc}\cap \mathrm{ri} \ \dom{\Ac}\ne \emptyset$, which implies that the closure of $\intx{\Zc}\cap\dom{\Ac}$ is exactly $\Zc\cap\cl{\dom{\Ac}}$. Choose $\tau'>\tau$; by definition of the horizon cone, $\zb_{\tau'}$ belongs to the closure of $\intx{\Zc}\cap\dom{\Ac}$, so  $\zb_{\tau'}\in \Zc$ and $\zb_{\tau'}\in \cl{\dom{\Ac}}$. Since $\zb_{\tau}$ is a convex combination of $\hat{\zb}$ and  $\zb_{\tau'}$, it belongs to $\intx{\Zc} \cap \dom{\Ac}$, where we used the assumption that $\dom{\Ac}$ is either closed or open.

Next, for any $\ab_\tau \in \Ac(\zb_\tau)$, we have
\begin{equation*}
\iprods{\ab_\tau, \boldsymbol{\omega}} = \iprods{\ab_\tau - \hat{\ab}, \boldsymbol{\omega}} + \iprods{\hat{\ab}, \boldsymbol{\omega}}=\iprods{\ab_\tau - \hat{\ab}, \tau^{-1}(\zb_\tau-\hat{\zb})} + \iprods{\hat{\ab},\boldsymbol{\omega}} \ge \iprods{\hat{\ab},\boldsymbol{\omega}} > 0.
\end{equation*}
On the other hand, $\iprods{t \nabla{F}(\zb_\tau), \boldsymbol{\omega}} = \iprods{t \nabla{F}(\zb_\tau), \tau^{-1}(\zb_\tau-\hat{\zb})} \ge -\tau^{-1} t\nu$ by \cite[Theorem 4.2.4]{Nesterov2004}.
Combining the above two inequalities, we can see that
\begin{equation*}
\iprods{t \nabla{F}(\zb_\tau) + \ab_\tau, \boldsymbol{\omega}} \ge -\tau^{-1} t\nu + \iprods{\hat{\ab}, \boldsymbol{\omega}} >0
\end{equation*}
as long as $\tau^{-1}t\nu < \iprods{\hat{\ab},\boldsymbol{\omega}}$.
We have thereby verified the condition in \cite[Theorem 12.51]{Rockafellar1997} needed to guarantee \eqref{eq:barrier_mono_inc} to have a nonempty (and bounded) solution set.
Since $\nabla F$ is strictly monotone, the solution of \eqref{eq:barrier_mono_inc} is unique.

We note that $\zb^{\star}_t$ is the solution of \eqref{eq:barrier_mono_inc} and $\zb^{\star}_t \in\intx{\Zc}$, we have $-t\nabla{F}(\zb^{\star}_t) \in \Ac(\zb^{\star}_t) = \Ac_{\Zc}(\zb^{\star}_t)$. Hence, $\mathrm{dist}_{\zb^{\star}_t}(\boldsymbol{0}, \Ac_{\Zc}(\zb^{\star}_t)) \leq t\norm{\nabla{F}(\zb^{\star}_t)}_{\zb^{\star}_t}^{\ast} \leq t\sqrt{\nu}$ due to the property of $F$ \cite{Nesterov2004}.
Using Definition~\ref{de:MVIP_approx_sol}, we have the last conclusion.
\Eproof
}

\beforesubsec
\subsection{\bf The proof of  Lemma~\ref{le:sol_measure}: Approximate solution}\label{apdx:le:sol_measure}
\aftersubsec
First, since $\bar{\zb}_{+}$ is a zero point of \Shu{$\widehat{\Ac}_t(\cdot;z)$}, i.e., \Shu{$0 \in\widehat{\Ac}_t(\bar{\zb}_{+},z)$}, we have  $-t\nabla{F}(\zb) - t\nabla^2{F}(\zb)(\bar{\zb}_{+} - \zb) \in \Ac(\bar{\zb}_{+})$.
Second, since $\zb_{+}$ is a $\delta$-solution to \eqref{eq:zero_k_set}, there exists $\eb$ such that  $\eb \in t\nabla{F}(\zb) + t\nabla^2{F}(\zb)(\zb_{+} - \zb)  + \Ac(\zb_{+})$ with $\Vert \eb\Vert_{\zb}^{\ast} \leq t\delta$
by Definition \ref{de:approx_sol}.
Combining these expressions and using the monotonicity of $\Ac$ in Definition \ref{de:monotone_opers}, we can show that $\iprods{t[\nabla{F}(\zb) + \nabla^2{F}(\zb)(\zb_{+} - \zb) - \nabla{F}(\zb) - \nabla^2{F}(\zb)(\bar{\zb}_{+} - \zb)] -\eb, \bar{\zb}_{+} - \zb_{+} } \geq 0$. This inequality leads to
\vspace{-0.75ex}
\begin{equation}\label{eq:lm31_est}
t\Vert\zb_{+} - \bar{\zb}_{+}\Vert_{\zb}^2 \leq \iprods{\eb, \zb_{+} - \bar{\zb}_{+}} \leq \Vert\eb\Vert_{\zb}^{\ast}\Vert\zb_{+} - \bar{\zb}_{+}\Vert_{\zb},
\vspace{-0.75ex}
\end{equation}
which implies $\Vert\zb_{+} - \bar{\zb}_{+}\Vert_{\zb} \leq t^{-1}\Vert\eb\Vert_{\zb}^{\ast}$.
Hence, $\Vert\eb\Vert_{\zb}^{\ast}\leq t\delta$ implies $\Vert\zb_{+} - \bar{\zb}_{+}\Vert_{\zb} \leq \delta$.

Next, since $\zb_{+}$ is a $\delta$-approximate solution to \eqref{eq:zero_k_set} at $t$ in the sense of Definition~\ref{de:approx_sol} up to the accuracy $\delta$, \Shu{there exists $\eb \in\R^p$ such that}
\vspace{-0.75ex}
\begin{equation*}
\eb \in t\left[\nabla{F}(\zb) + \nabla^2{F}(\zb)(\zb_{+} - \zb)\right] + \Ac(\zb_{+})~~\text{with}~~\norm{\eb}_{\zb}^{\ast} \leq t\delta.
\vspace{-0.75ex}
\end{equation*}
In addition, we have $\zb_{+} \in\intx{\Zc}$ due to Theorem~\ref{th:main_estimate1} below.
Hence, we have $\Ac_{\Zc}(\zb_{+}) = \Ac(\zb_{+})$.
Using this relation and the above inclusion, we can show that
\vspace{-0.75ex}
\begin{equation}\label{eq:lm4_est100}
\begin{array}{ll}
\mathrm{dist}_{\zb}(\boldsymbol{0}, \Ac_{\Zc}(\zb_{+})) &\leq \Vert \eb - t\left[\nabla{F}(\zb) + \nabla^2{F}(\zb)(\zb_{+} - \zb)\right] \Vert_{\Shu{\zb}}^{\ast} \vspace{0.75ex}\\
&\leq \norm{\eb}_{\zb}^{\ast} + t\norm{\nabla{F}(\zb)}_{\zb}^{\ast} + t\Vert \nabla^2{F}(\zb)(\zb_{+} - \zb)\Vert_{\zb}^{\ast} \vspace{0.75ex}\\
&\leq t \left[\delta + \sqrt{\nu}  + \Vert \nabla^2{F}(\zb)(\bar{\zb}_{+} - \zb)\Vert_{\zb}^{\ast}  + \Vert \nabla^2{F}(\zb)(\bar{\zb}_{+} - \zb_{+})\Vert_{\zb}^{\ast}\right]\vspace{0.75ex}\\
&\leq t \left[\delta + \sqrt{\nu}  + \lambda_{t}(\zb) + \norm{\zb_{+} - \bar{\zb}_{+}}_{\zb}\right]\vspace{0.75ex}\\
&\leq t\left( \sqrt{\nu} + \lambda_{t}(\zb) + 2\delta\right).
\end{array}
\vspace{-0.75ex}
\end{equation}
Here, we have used $\norm{\nabla{F}(\zb)}_{\zb}^{\ast} \leq \sqrt{\nu}$, and $\norm{\zb_{+} - \bar{\zb}_{+}}_{\zb} \leq \delta$ by the first part of this lemma.
We note that if $\lambda_t(\zb) + \delta < 1$, then $\mathrm{dist}_{\zb_{+}}(\boldsymbol{0}, \Ac_{\Zc}(\zb_{+})) \leq (1-\lambda_t(\zb) -\delta)^{-1}\mathrm{dist}_{\zb}(\boldsymbol{0}, \Ac_{\Zc}(\zb_{+}))$.
Combining this inequality and the last estimate, we obtain \eqref{eq:sol_measure}.
Finally, if we choose $t \leq (1-\lambda_t(\zb)-\delta)\left( \sqrt{\nu} + \lambda_{t}(\zb) + 2\delta\right)^{-1}\varepsilon$, then $\mathrm{dist}_{\zb_{+}}(\boldsymbol{0}, \Ac_{\Zc}(\zb_{+})) \leq \varepsilon$.
Hence, $\zb_{+}$ is an $\varepsilon$-solution to \eqref{eq:mono_inclusion} in the sense of Definition~\ref{de:MVIP_approx_sol}.
\Eproof

\beforesubsec
\subsection{\bf The proof of Theorem \ref{th:main_estimate1}: Key estimate of generalized Newton-type schemes}\label{apdx:th:main_estimate1}
\aftersubsec
First, similar to \cite{Bauschke2011}, we can easily show the the following non-expansive property holds
\vspace{-0.25ex}
\begin{equation}\label{eq:Pc_oper_property}
 \Vert \Pc_{\hat{\zb}}(\ub; t) -  \Pc_{\hat{\zb}}(\vb; t)\Vert_{\hat{\zb}} \leq \Vert\ub - \vb\Vert_{\hat{\zb}},~~~\forall\ub,\vb\in\R^p.
\vspace{-0.25ex}
\end{equation}
We note that $\Vert \zb_{+} - \zb\Vert_{\zb} \leq \Vert \bar{\zb}_{+} - \zb\Vert_{\zb} + \Vert \zb_{+} - \bar{\zb}_{+}\Vert_{\zb} = \lambda_t(\zb) + \delta(\zb) < 1$ by our assumption.
This shows that $\zb_{+}\in\intx{\Zc}$ due to \cite[Theorem 4.1.5 (1)]{Nesterov2004}.

Next, we consider the generalized gradient mappings $G_{\zb}(\zb;t_{+})$ and $G_{\zb_{+}}(\zb_{+}; t_{+})$ at $\zb$ and $\zb_{+}$, respectively defined by \eqref{eq:gradient_mapping} as follows:
\vspace{-0.5ex}
\begin{equation}\label{eq:lm32_proof1}
\begin{array}{ll}
G_{\zb}(\zb; t_{+}) &:= \nabla^2{F}(\zb)\left( \zb - \Pc_{\zb}\left(\zb - \nabla^2{F}(\zb)^{-1}\nabla{F}(\zb);  t_{+}\right)\right), \vspace{0.75ex}\\
G_{\zb_{+}}(\zb_{+}; t_{+}) &:= \nabla^2{F}(\zb_{+})\left(\zb_{+} - \Pc_{\zb_{+}}\left(\zb_{+} -  \nabla^2{F}(\zb_{+})^{-1}\nabla{F}(\zb_{+});  t_{+}\right)\right).
\end{array}
\vspace{-0.5ex}
\end{equation}
Let $r_{\zb}(\bar{\zb}_{+}) := \nabla{F}(\zb) + \nabla^2{F}(\zb)(\bar{\zb}_{+} - \zb)$. Then, by using $\bar{\zb}_{+} :=   \Pc_{\zb}\left(\zb - \nabla^2{F}(\zb)^{-1}\nabla{F}(\zb);  t_{+}\right)$ from \eqref{eq:full_step_Newton}, we can show that
\vspace{-0.5ex}
\begin{equation}\label{eq:lm32_proof2}
-r_{\zb}(\bar{\zb}_{+}) := -\left[\nabla{F}(\zb) + \nabla^2{F}(\zb)(\bar{\zb}_{+} - \zb)\right] \in t_{+}^{-1}\Ac(\bar{\zb}_{+}).
\vspace{-0.5ex}
\end{equation}
Clearly, we can rewrite \eqref{eq:lm32_proof2} as $\bar{\zb}_{+} -\nabla^2{F}(\zb_{+})^{-1}r_{\zb}(\bar{\zb}_{+}) \in \bar{\zb}_{+} + t_{+}^{-1}\nabla^2{F}(\zb_{+})^{-1}\Ac(\bar{\zb}_{+})$.
Then, using the definition \eqref{eq:Pc_oper} of $\Pc_{\zb_{+}}(\cdot) := \left(\Id + t_{+}^{-1}\nabla^2{F}(\zb_{+})^{-1}\Ac\right)^{-1}(\cdot)$, we can derive
\begin{equation}\label{eq:lm32_proof2a}
\zb_{+} = \Pc_{\zb_{+}}\left(\bar{\zb}_{+} -\nabla^2{F}(\zb_{+})^{-1}r_{\zb}(\bar{\zb}_{+}); t_{+} \right) + (\zb_{+} - \bar{\zb}_{+}).
\end{equation}
Now, we can estimate $\lambda_{t_{+}}(\zb_{+})$ defined by \eqref{eq:nt_decrement} using \eqref{eq:lm32_proof1}, \eqref{eq:lm32_proof2a}, \eqref{eq:Pc_oper_property} and  \eqref{eq:lm32_proof2} as follows:
\vspace{-0.5ex}
\begin{align}\label{eq:lm32_proof3}
\lambda_{t_{+}}(\zb_{+}\!) &\!:=\! \Vert G_{\zb_{+}}(\zb_{+}; t_{+})\Vert^{\ast}_{\zb_{+}} \overset{\tiny\eqref{eq:lm32_proof1}}{=} \Vert \zb_{+} - \Pc_{\zb_{+}}\left(\zb_{+} - \nabla^2{F}(\zb_{+})^{-1}\nabla{F}(\zb_{+}); t_{+}\right)\Vert_{\zb_{+}} \nonumber\\
&\overset{\tiny\eqref{eq:lm32_proof2a}}{{\!\!\!}={\!}} \Big\Vert  \Pc_{\zb_{+}}{\!\!}\left(\bar{\zb}_{+} \!-\! \nabla^2{F}(\zb_{+}\!)^{-1}{\!} r_{\zb}(\bar{\zb}_{+}); t_{+} \right) -  \Pc_{\zb_{+}}{\!\!}\left(\zb_{+} \!-\!  \nabla^2{F}(\zb_{+})^{-1}{\!}\nabla{F}(\zb_{+});  t_{+}\right)  \!+\! (\zb_{+} \!-\! \bar{\zb}_{+}) \Big\Vert_{\zb_{+}} \nonumber\\
&\leq \Big\Vert  \Pc_{\zb_{+}}\left(\bar{\zb}_{+} -\nabla^2{F}(\zb_{+})^{-1}r_{\zb}(\bar{\zb}_{+}); t_{+} \right) -  \Pc_{\zb_{+}}\left(\zb_{+} -  \nabla^2{F}(\zb_{+})^{-1}\nabla{F}(\zb_{+});  t_{+}\right) \Big\Vert_{\zb_{+}}\nonumber\\
&{~~~~~~} + \Vert \zb_{+} - \bar{\zb}_{+}\Vert_{\zb_{+}} \nonumber\\
& \overset{\tiny\eqref{eq:Pc_oper_property}}{\leq} \Big\Vert \nabla^2{F}(\zb_{+})^{-1}\left[\nabla{F}(\zb_{+}) - r_{\zb}(\bar{\zb}_{+})\right] + (\bar{\zb}_{+} - \zb_{+}) \Big\Vert_{\Shu{\zb_{+}}}  +  \Vert \zb_{+} - \bar{\zb}_{+}\Vert_{\zb_{+}} \nonumber\\
&\overset{\tiny\eqref{eq:lm32_proof2}}{=} \Big\Vert \nabla^2{F}(\zb_{+})^{-1}\big[\nabla{F}(\zb_{+}) - \nabla{F}(\zb) - \nabla^2{F}(\zb)(\zb_{+} - \zb) + (\nabla^2{F}(\zb_{+})\nonumber\\
&{~~~~~~} - \nabla^2{F}(\zb))(\bar{\zb}_{+} - \zb_{+})\big]\Big\Vert_{\zb_{+}} + \Vert \zb_{+} - \bar{\zb}_{+}\Vert_{\zb_{+}} \nonumber\\
&\leq  \Vert \nabla{F}(\zb_{+}) - \nabla{F}(\zb) - \nabla^2{F}(\zb)(\zb_{+} - \zb)  \Vert_{\zb_{+}}^{\ast} \nonumber\\
& + \Vert (\nabla^2{F}(\zb_{+}) - \nabla^2{F}(\zb))(\bar{\zb}_{+} - \zb_{+})\Vert_{\zb_{+}}^{\ast} +\Vert \zb_{+} - \bar{\zb}_{+}\Vert_{\zb_{+}} \nonumber\\
& \leq \tfrac{1}{1 - \Vert \zb_{+} - \zb \Vert_{\zb}}\Big[ \Vert \nabla{F}(\zb_{+}) - \nabla{F}(\zb) - \nabla^2{F}(\zb)(\zb_{+} - \zb)  \Vert_{\zb}^{\ast} \nonumber\\
& +  \Vert (\nabla^2{F}(\zb_{+}) - \nabla^2{F}(\zb))(\bar{\zb}_{+} - \zb_{+})\Vert_{\zb}^{\ast} \Big]  + \tfrac{\Vert \zb_{+} - \bar{\zb}_{+}\Vert_{\zb} }{1 - \Vert \zb_{+} - \zb \Vert_{\zb}}.
\end{align}
Here, in the last equality of \eqref{eq:lm32_proof3}, we have used the fact that $\Vert\wb\Vert_{\zb_{+}}^2 = \iprods{\nabla^2{F}(\zb_{+})\wb, \wb} \leq (1-\Vert\zb_{+} - \zb\Vert_{\zb})^{-2}\iprods{\nabla^2{F}(\zb)\wb, \wb} = (1-\Vert\zb_{+} - \zb\Vert_{\zb})^{-2}\Vert\wb\Vert_{\zb}^2$ for any $\wb$ and $\zb, \zb_{+}$ such that $\Vert\zb_{+} - \zb\Vert_{\zb} < 1$, and the analogous fact for the dual norms.
Both facts can be derived from \cite[Theorem 4.1.6]{Nesterov2004}.
The condition $\Vert\zb_{+} - \zb\Vert_{\zb} < 1$ is guaranteed since $\Vert\zb_{+} - \zb\Vert_{\zb} \leq \Vert \zb - \bar{\zb}_{+}\Vert_{\zb} + \Vert\zb_{+} - \bar{\zb}_{+}\Vert_{\zb} = \lambda_{t_{+}}(\zb) + \delta(\zb) < 1$ by our assumption.

Similar to the proof of \cite[Theorem 4.1.14]{Nesterov2004}, we can show that
\begin{equation}\label{eq:lm32_proof3a}
\Vert \nabla{F}(\zb_{+}) - \nabla{F}(\zb) - \nabla^2{F}(\zb)(\zb_{+} - \zb)\Vert_{\zb}^{\ast} \leq \frac{\Vert \zb_{+} - \zb\Vert_{\zb}^2}{1 - \Vert \zb_{+} - \zb\Vert_{\zb}}.
\end{equation}
Next, we need to estimate $B := \Vert (\nabla^2{F}(\zb_{+}) - \nabla^2{F}(\zb))(\bar{\zb}_{+} - \zb_{+})\Vert_{\zb}^{\ast}$.
We define
\begin{equation*}
\Sigma := \nabla^2{F}(\zb)^{-1/2}\left(\nabla^2{F}(\zb_{+}) - \nabla^2{F}(\zb)\right)\nabla^2{F}(\zb)^{-1/2}.
\end{equation*}
By \cite[Theorem 4.1.6]{Nesterov2004}, we can show that
\begin{equation*}
\Vert \Sigma\Vert \leq \max\set{1 - (1 - \Vert\zb_{+} -\zb\Vert_{\zb})^2, (1 - \Vert\zb_{+} - \zb\Vert_{\zb})^{-2} - 1} = \frac{2\Vert\zb_{+} - \zb\Vert_{\zb} - \Vert\zb_{+} - \zb\Vert_{\zb}^2}{(1 - \Vert\zb_{+} - \zb\Vert_{\zb})^2}.
\end{equation*}
Using this inequality we can estimate $B$ as
\begin{align*}
B^2 &=  (\bar{\zb}_{+} - \zb_{+})^{\top}\nabla^2{F}(\zb)^{1/2}\Sigma^2\nabla^2{F}(\zb)^{1/2}(\bar{\zb}_{+} - \zb_{+}) \leq \Vert \Sigma \Vert^2\Vert\bar{\zb}_{+} - \zb_{+}\Vert_{\zb}^2\nonumber\\
&\leq \left(\frac{2\Vert\zb_{+} - \zb\Vert_{\zb} - \Vert\zb_{+} - \zb\Vert_{\zb}^2}{(1 - \Vert\zb_{+} - \zb\Vert_{\zb})^2}\right)^2\Vert\bar{\zb}_{+} - \zb_{+}\Vert_{\zb}^2,
\end{align*}
which implies
\begin{equation}\label{eq:lm32_proof3b}
B \leq \left(\frac{2\Vert\zb_{+} - \zb\Vert_{\zb} - \Vert\zb_{+} - \zb\Vert_{\zb}^2}{(1 - \Vert\zb_{+} - \zb\Vert_{\zb})^2}\right)\Vert\bar{\zb}_{+} - \zb_{+}\Vert_{\zb}.
\end{equation}
Substituting \eqref{eq:lm32_proof3a} and \eqref{eq:lm32_proof3b} into \eqref{eq:lm32_proof3} we get
\begin{align}\label{eq:lm32_proof4}
\lambda_{t_{+}}(\zb_{+}) &\leq \frac{\Vert\zb_{+} - \zb\Vert_{\zb}^2}{\left(1 - \Vert\zb_{+} - \zb\Vert_{\zb}\right)^2} +  \frac{\left[2\Vert\zb_{+} - \zb\Vert_{\zb} - \Vert\zb_{+} - \zb\Vert_{\zb}^2\right]\Vert\bar{\zb}_{+} - \zb_{+}\Vert_{\zb}}{(1 - \Vert\zb_{+} - \zb\Vert_{\zb})^3} + \frac{\Vert \zb_{+} - \bar{\zb}_{+}\Vert_{\zb} }{1 - \Vert \zb_{+} - \zb \Vert_{\zb}} \nonumber\\
&= \frac{\Vert\zb_{+} - \zb\Vert_{\zb}^2}{\left(1 - \Vert\zb_{+} - \zb\Vert_{\zb}\right)^2}  + \frac{\Vert \zb_{+} - \bar{\zb}_{+}\Vert_{\zb} }{\left(1 - \Vert \zb_{+} - \zb \Vert_{\zb}\right)^3}.
\end{align}
Finally, we note that $\lambda_{t_{+}}(\zb) := \Vert G_{\zb}(\zb; t_{+})\Vert_{\zb}^{\ast} = \Vert \zb - \Pc_{\zb}\left(\zb - \nabla^2{F}(\zb)^{-1}\nabla{F}(\zb); t_{+} \right)\Vert_{\zb} = \Vert\zb - \bar{\zb}_{+} \Vert_{\zb}$ due to \eqref{eq:full_step_Newton}.
Using the triangle inequality we have $\Vert \zb_{+}-\zb\Vert_{\zb} \leq \Vert \zb - \bar{\zb}_{+}\Vert_{\zb} + \Vert\zb_{+} - \bar{\zb}_{+}\Vert_{\zb} = \lambda_{t_{+}}(\zb) + \delta(\zb) < 1$.
Since the right-hand side of \eqref{eq:lm32_proof4} is monotonically increasing with respect to $\Vert \zb_{+}-\zb\Vert_{\zb}$, using the last inequality into \eqref{eq:lm32_proof4}, we obtain \eqref{eq:key_estimate1}.
\Eproof

\beforesubsec
\subsection{\bf The proof of Theorem \ref{th:generalized_newton_convergence}: Local quadratic convergence of \ref{eq:iFsGNM}}\label{apdx:th:generalized_newton_convergence}
\aftersubsec
We first prove (a). Given a fixed parameter $t > 0$ sufficiently small, our objective is to find $\beta\in (0, 1)$ such that if $\lambda_{t}(\zb^k) \leq \beta$, then $\lambda_{t}(\zb^{k\!+\!1}) \leq \beta$.
Indeed, using the key estimate \eqref{eq:key_estimate1} with $t$ instead of $t_{+}$, we can see that to guarantee $\lambda_{t}(\zb^{k\!+\!1}) \leq \beta$, we require
\begin{equation*}
\left(\frac{\lambda_{t}(\zb^k) + \delta(\zb^k)}{1 - \lambda_{t}(\zb^k) - \delta(\zb^k)}\right)^2 +  \frac{\delta(\zb^k)}{\left(1-\lambda_{t}(\zb^k) - \delta(\zb^k)\right)^3} \leq \beta.
\end{equation*}
Since the left-hand side of this inequality is monotonically increasing when $\lambda_{t}(\zb^k)$ and $\delta(\zb^k)$ are increasing, we can overestimate it by
\begin{equation*}
\left(\frac{\beta+\delta}{1 - \beta - \delta}\right)^2 + \frac{\delta}{(1-\beta - \delta)^3} \leq \beta.
\end{equation*}
Using the identity $\frac{\beta+\delta}{1-\beta-\delta} = \frac{\beta}{1-\beta} + \frac{\delta}{(1-\beta)(1-\beta-\delta)}$, we can write the last inequality as
\begin{equation}\label{eq:th43_est2}
\Big[\frac{2\beta}{(1-\beta)^2(1-\beta-\delta)} +  \frac{\delta}{(1-\beta)^2(1-\beta-\delta)^2} + \frac{1}{(1-\beta - \delta)^3}\Big]\delta \leq \beta - \left(\frac{\beta}{1-\beta}\right)^2.
\end{equation}
Clearly, the left-hand side of \eqref{eq:th43_est2} is positive if $0 < \delta < 1-\beta$.
Hence, we need to choose $\beta \in (0, 0.5(3-\sqrt{5}))$ such that the right-hand side of \eqref{eq:th43_est2} is also positive.
Now, we choose $\delta \geq 0$ such that $\delta \leq \beta(1-\beta) < 1-\beta$.
Then, \eqref{eq:th43_est2} can be one more time overestimated by
\begin{equation*}
\Big(\frac{2\beta^3 - 5\beta^2 + 3\beta + 1}{(1-\beta)^4}\Big)\delta \leq \beta(1 - 3\beta + \beta^2),
\end{equation*}
which implies
\begin{equation*}
0 \leq \delta \leq \frac{\beta(1 - 3\beta + \beta^2)(1-\beta)^4}{2\beta^3 - 5\beta^2 + 3\beta + 1} < \beta(1-\beta),~~\forall\beta\in \left(0, 0.5(3-\sqrt{5})\right).
\end{equation*}
This inequality suggests that we can choose $\delta := \frac{\beta(1 - 3\beta + \beta^2)(1-\beta)^4}{2\beta^3 - 5\beta^2 + 3\beta + 1} > 0$. In this case, we also have $\delta(\zb) + \lambda_t(\zb) \leq \delta + \beta < 1$, which guarantees the condition of Theorem \ref{th:main_estimate1}.
Hence, we can conclude that $\lambda_t(\zb^k) \leq \beta$ implies $\lambda_t(\zb^{k+1}) \leq \beta$. Hence, $\set{\zb^k}$ belongs to $\Qc_{t}(\beta)$.

(b)~Next, to guarantee a quadratic convergence, we can choose $\delta_k$ such that $\delta(\zb^k) \leq \delta_k \leq \bar{\delta}_k := \frac{\lambda_{t}(\zb^k)^2}{1-\lambda_{t}(\zb^k)}$.
Substituting the upper bound $\bar{\delta}_k$ of $\delta(\zb^k)$ into \eqref{eq:key_estimate1} we obtain
\begin{equation*}
\lambda_{t}(\zb^{k\!+\!1})  \leq \left(\frac{2-4\lambda_t(\zb^k) + \lambda_t(\zb^k)}{(1-2\lambda_t(\zb^k))^3}\right)\lambda_t(\zb^k)^2.
\end{equation*}
Let us consider the function $s(r) := \frac{(2 - 4r + r^2)r^2}{(1 - 2r)^3}$ on $[0, 1]$.
We can easily check that $s(r) < 1$ for all $r \in [0, 1]$.
Hence, $\lambda_{t}(\zb^{k\!+\!1}) < 1$ as long as $\lambda_{t}(\zb^k) < 1$.
This proves the estimate \eqref{eq:iFsGNM_est}.

Now, let us choose some $\beta \in (0, 1)$ such that $\lambda_t(\zb^k) \leq \beta$. Then \eqref{eq:iFsGNM_est} leads to
\begin{equation*}
\lambda_t(\zb^{k+1}) \leq \left(\frac{2-4\beta+\beta^2}{(1-2\beta)^3}\right)\lambda_t(\zb^k)^2 = c\lambda_t(\zb)^2,
\end{equation*}
where $c :=  \frac{2-4\beta+\beta^2}{(1-2\beta)^3} > 0$.
We need to choose $\beta \in (0, 1)$ such that $c\lambda_t(\zb^k) < 1$. Since $\lambda_t(\zb^k) \leq \beta$, we choose $c\beta < 1$, which is equivalent to $9\beta^3 - 16\beta^2 + 8\beta - 1 < 0$.
If $\beta \in (0, 0.18858]$ then $9\beta^3 - 16\beta^2 + 8\beta - 1 < 0$.
Therefore, the radius of the quadratic convergence region of $\set{\lambda_t(\zb^k)}$ is $r := 0.18858$.

(c)~Finally, for any $\beta \in (0, 0.18858]$, we can write $c\lambda_t(\zb^{k+1}) \leq (c\lambda_t(\zb^k))^2$.
By induction, $c\lambda_t(\zb^k) \leq (c\lambda_t(\zb^0))^{2^k} \leq c^{2^k}\beta^{2^k} < 1$.
We obtain $\lambda_t(\zb^k) \leq   c^{2^{k-1}}\beta^{2^k}$.
Let us choose $\delta_k := \frac{\lambda_t(\zb^k)^2}{1-\lambda_t(\zb^k)}$.
For $\epsilon \in (0, \beta)$, assume that $c^{2^{k-1}}\beta^{2^k}\leq\epsilon$.
From Lemma~Lemma~\ref{le:sol_measure}, we can choose $t := (1-\epsilon)(\sqrt{\nu} + \epsilon + 2\epsilon^2/(1-\epsilon))^{-1}\varepsilon$ .
Then, $\zb^k$ is an $\varepsilon$-solution of \eqref{eq:mono_inclusion}.
It remains to use the fact that $c^{2^{k-1}}\beta^{2^k}\leq\epsilon$ to upper bound the number of iterations $k := \mathcal{O}\left(\ln\left(\ln(1/\epsilon)\right)\right)$.
\Eproof

\beforesubsec
\subsection{\bf The proof of Theorem \ref{th:generalized_ds_newton_convergence}: Local quadratic convergence of \ref{eq:iDsGNM}}\label{apdx:th:generalized_ds_newton_convergence}
\aftersubsec
(a)~Given a fixed parameter $t > 0$ sufficiently small, it follows from \ref{eq:iDsGNM} and \eqref{eq:lm32_proof2a} that
\begin{equation*}
\begin{array}{ll}
\bar{\zb}^{k+2} &= \Pc_{\zb^{k+1}}\left(\zb^{k+1} - \nabla^2{F}(\zb^{k+1})^{-1}\nabla{F}(\zb^{k+1}); t\right),\vspace{1ex}\\
\zb^{k+1} &= \Pc_{\zb^{k+1}}\left(\Shu{\bar{\zb}^{k+1}} - \nabla^2{F}(\zb^{k+1})^{-1}r_{\zb^k}(\bar{\zb}^{k+1}); t\right) + (\zb^{k+1} - \bar{\zb}^{k+1}).
\end{array}
\end{equation*}
Hence, using these notations and  the same proof as  \eqref{eq:lm32_proof4} with $t$ instead of $t_{+}$, \Shu{and supposing $\Vert\zb^{k+1} - \zb^k\Vert_{\zb^k}<1$,} we can derive
\begin{align}\label{eq:th44_proof1}
\Vert \bar{\zb}_{k+2} - \zb^{k+1}\Vert_{\zb_{k+1}} \leq \left(\frac{\Vert\zb^{k+1} - \zb^k\Vert_{\zb^k}}{1-\Vert\zb^{k+1} - \zb^k\Vert_{\zb^k}}\right)^2 + \frac{\Vert\zb^{k+1} - \bar{\zb}^{k+1}\Vert_{\zb^k}}{(1-\Vert\zb^{k+1} - \zb^k\Vert_{\zb^k})^3}.
\end{align}
Now, let us define  $\tilde{\lambda}_t(\zb^k) := \Vert \tilde{\zb}^{k\!+\!1} - \zb^k\Vert_{\zb^k}$ and $\alpha_k := (1 + \tilde{\lambda}_t(\zb^k))^{-1}$ as in \ref{eq:iDsGNM}.
From the update \ref{eq:iDsGNM}, $\zb^{k\!+\!1} := (1-\alpha_k)\zb^k + \alpha_k\tilde{\zb}^{k\!+\!1}$, we have
\begin{equation*}
\begin{array}{lll}
&\Vert \zb^{k\!+\!1} - \zb^k\Vert_{\zb^k} &= \alpha_k\Vert \tilde{\zb}^{k\!+\!1} - \zb^k\Vert_{\zb^k} = \alpha_k\tilde{\lambda}_t(\zb^k),~~~\text{and}\vspace{1ex}\\
&\Vert \zb^{k+1} - \bar{\zb}^{k+1}\Vert_{\zb^k} &\leq \Vert \zb^{k+1} - \tilde{\zb}^{k+1}\Vert_{\zb^k} + \Vert \tilde{\zb}^{k+1} - \bar{\zb}^{k+1}\Vert_{\zb^k} = (1-\alpha_k)\Vert\tilde{\zb}^{k+1} - \zb^k\Vert_{\zb^k} + \delta(\zb^k) \vspace{0.75ex}\\
& & = (1-\alpha_k)\tilde{\lambda}_t(\zb^k) + \delta(\zb^k).
\end{array}
\end{equation*}
Substituting these expressions into \eqref{eq:th44_proof1} we get
\begin{align*}
\Vert \bar{\zb}_{k+2} - \zb^{k+1}\Vert_{\zb_{k+1}} &\leq \left( \frac{\alpha_k\tilde{\lambda}_{t}(\zb^k)}{1 - \alpha_k\tilde{\lambda}_{t}(\zb^k)}\right)^2  + \frac{\delta(\zb^k) + (1-\alpha_k)\tilde{\lambda}_t(\zb^k)}{\left(1 - \alpha_k\tilde{\lambda}_{t}(\zb^k)\right)^3}.
\end{align*}
Substituting $\alpha_k := (1 + \tilde{\lambda}_{t}(\zb^k))^{-1}$ into the last inequality and simplifying the result, we get
\begin{equation*}
\Vert \bar{\zb}_{k+2} - \zb^{k+1}\Vert_{\zb_{k+1}} \leq \left(2 + 2\tilde{\lambda}_{t}(\zb^k) + \tilde{\lambda}_{t}(\zb^k)^2\right) \tilde{\lambda}_{t}(\zb^k)^2 + \left(1 + \tilde{\lambda}_{t}(\zb^k)\right)^3\delta(\zb^k).
\end{equation*}
Next, by the triangle inequality, it follows from \eqref{eq:lm32_proof1} and the definition of $\lambda_t(\zb)$ and $\tilde{\lambda}_t(\zb)$ that  $\tilde{\lambda}_t(\zb^{k\!+\!1}) = \Vert\tilde{\zb}^{k+2} - \zb^{k+1}\Vert_{\zb^{k+1}} \leq \Vert \bar{\zb}^{k+2} - \zb^{k+1}\Vert_{\zb^{k+1}} + \Vert \tilde{\zb}^{k+2} - \bar{\zb}^{k+2}\Vert_{\zb^{k+1}} = \Vert \bar{\zb}^{k+2} - \zb^{k+1}\Vert_{\zb^{k+1}}  + \delta(\zb^{k\!+\!1})$.
Combining this estimate and the above inequality we get
\begin{equation*}
\tilde{\lambda}_{t}(\zb^{k\!+\!1}) \leq \left(2 + 2\tilde{\lambda}_{t}(\zb^k) + \tilde{\lambda}_{t}(\zb^k)^2\right) \tilde{\lambda}_{t}(\zb^k)^2 + \left(1 + \tilde{\lambda}_{t}(\zb^k)\right)^3\delta(\zb^k) + \delta(\zb^{k\!+\!1}).
\end{equation*}
If we choose $\delta(\zb^k) \leq \delta_k \leq \frac{\tilde{\lambda}_{t}(\zb^k)^2}{1+ \tilde{\lambda}_{t}(\zb^k)}$, then, by induction, $\delta(\zb^{k\!+\!1}) \leq \delta_{k\!+\!1} \leq\frac{\tilde{\lambda}_{t}(\zb^{k\!+\!1})^2}{1+\tilde{\lambda}_{t}(\zb^{k\!+\!1})}$.
Substituting these bounds into the last inequality and simplifying the result, we obtain
\begin{equation*}
\tilde{\lambda}_{t}(\zb^{k\!+\!1}) \leq \left(\frac{2\tilde{\lambda}_t(\zb^k)^2 + 4\tilde{\lambda}_t(\zb^k) + 3}{1 - \tilde{\lambda}_t(\zb^k)^2\left(2\tilde{\lambda}_t(\zb^k)^2 + 4\tilde{\lambda}_t(\zb^k) + 3\right)}\right)\tilde{\lambda}_{t}(\zb^k)^2,
\end{equation*}
which is indeed \eqref{eq:key_estimate_2b}.

From \eqref{eq:key_estimate_2b}, after a few elementary calculations, we can see that $\tilde{\lambda}_{t}(\zb^{k\!+\!1}) \leq \tilde{\lambda}_t(\zb^k)$ if $\tilde{\lambda}_t(\zb^k)(1 + \tilde{\lambda}_t(\zb^k))( 2\tilde{\lambda}_t(\zb^k)^2 + 4\tilde{\lambda}_t(\zb^k) + 3) \leq 1$.
We note that the function $s(\tau) := \tau(1+\tau)(2\tau^2 + 4\tau+3)$ is increasing on $[0, 0.5(3-\sqrt{5}))$.
By numerically computing $\tilde{\lambda}_t(\zb^k)$ we can observe that if $\tilde{\lambda}_t(\zb^k) \in [0, 0.21027]$, then $\tilde{\lambda}_{t}(\zb^{k\!+\!1}) \leq \tilde{\lambda}_t(\zb^k)$.
Hence, if $\tilde{\lambda}_t(\zb^k) \leq \beta$ then $\tilde{\lambda}_t(\zb^{k+1}) \leq \beta$. We can say that $\set{\zb^k}\subset \Omega_t(\beta)$.

We now prove (b). Indeed, if we take any $\beta \in (0, 0.21027]$, we can show from \eqref{eq:key_estimate_2b} that
\begin{equation*}
\tilde{\lambda}_{t}(\zb^{k\!+\!1}) \leq \left(\frac{2\beta^2 + 4\beta + 3}{1 - \beta^2\left(2\beta^2 + 4\beta + 3\right)}\right)\tilde{\lambda}_{t}(\zb^k)^2,
\end{equation*}
where $\bar{c} :=  \left(\frac{2\beta^2 + 4\beta + 3}{1 - \beta^2\left(2\beta^2 + 4\beta + 3\right)}\right) \in (0, +\infty)$.
To guarantee $\bar{c}\beta < 1$, we need to choose $\beta > 0$ such that $2\beta^4 + 6\beta^3 + 7\beta^2 + 3\beta - 1 < 0$.
This condition leads to $\beta \in (0, 0.21027]$.
Hence, for any $0 < \beta \leq 0.21027$, if $\zb^0\in\Qc_{t}(\beta)$, then $\tilde{\lambda}_{t}(\zb^{k\!+\!1}) \leq \bar{c}\tilde{\lambda}_{t}(\zb^k)^2 < 1$ and, therefore, $\big\{\tilde{\lambda}_{t}(\zb^k)\big\}$ quadratically converges to zero.

(c)~To prove the last conclusion in (c), from \eqref{eq:lm4_est100}, we can show that
\begin{align*}
\mathrm{dist}_{\zb^k}(\boldsymbol{0}, \Ac_{\Zc}(\zb^{k+1})) \leq t\delta_k +  t\norm{\nabla{F}(\zb^k)}_{\zb^k}^{\ast} + t\norm{\zb^{k+1} - \zb^k}_{\zb^k} \leq t(\delta_k + \sqrt{\nu} + \alpha_k\tilde{\lambda}_t(\zb^k)).
\end{align*}
Since $\tilde{\lambda}_t(\zb^k) \leq \bar{c}^{2^k-1}\lambda_t(\zb^0)^{2^k} \leq \bar{c}^{2^k-1}\beta^{2^k}$, $\delta_k \leq \frac{\tilde{\lambda}_t(\zb^k)^2}{1 + \tilde{\lambda}_t(\zb^k)}$, and $\alpha_k = \frac{\tilde{\lambda}_t(\zb^k)}{1 + \tilde{\lambda}_t(\zb^k)}$, we obtain the last conclusion as a consequence of Lemma~\ref{le:sol_measure} with the same proof as in Theorem~\ref{th:generalized_newton_convergence}.
\Eproof

\beforesubsec
\subsection{\bf The proof of Lemma \ref{le:update_penalty_param}: The update rule for the penalty parameter}\label{apdx:le:update_penalty_param}
\aftersubsec
Let us define $\bar{\ub}^k := \Pc_{\zb^k}\left(\zb^k - \nabla^2{F}(\zb^k)^{-1}\nabla{F}(\zb^k); t_k\right)$.
Then, $\lambda_{t_k}(\zb^k)$ defined by \eqref{eq:nt_decrement} becomes $\lambda_{t_k}(\zb^k) := \Vert G_{\zb^k}(\zb^k; t_k)\Vert_{\zb^k}^{\ast} = \Vert \zb^k - \Pc_{\zb^k}\left(\zb^k - \nabla^2{F}(\zb^k)^{-1}\nabla{F}(\zb^k); t_k\right)\Vert_{\zb^k} = \Vert \zb^k - \bar{\ub}^k \Vert_{\zb^k}$.
We note that $\bar{\ub}^k = \Pc_{\zb^k}\left(\zb^k - \nabla^2{F}(\zb^k)^{-1}\nabla{F}(\zb^k); t_k\right)$ leads to
\begin{equation*}
- t_k \left(\nabla{F}(\zb^k) + \nabla^2{F}(\zb^k)(\bar{\ub}^k - \zb^k)\right) \in \Ac(\bar{\ub}^k).
\end{equation*}
Combining this inclusion and \eqref{eq:lm32_proof2} and using the monotonicity of $\Ac$, we can derive
\begin{equation*}
\iprods{t_{k\!+\!1}\left[\nabla{F}(\zb^k) + \nabla^2{F}(\zb^k)(\bar{\zb}^{k\!+\!1} - \zb^k)\right] - t_k\left[\nabla{F}(\zb^k) + \nabla^2{F}(\zb^k)(\bar{\ub}^k - \zb^k)\right], \bar{\zb}^{k\!+\!1} - \bar{\ub}^k} \leq 0.
\end{equation*}
By rearranging this expression using $t_{k\!+\!1} := (1-\sigma_{\beta})t_k$ from \ref{eq:inexact_pf_scheme}, we finally obtain
\begin{align*}
\Vert \bar{\zb}^{k\!+\!1} - \bar{\ub}^k\Vert_{\zb^k}^2 &\leq \frac{\sigma_{\beta}}{1 - \sigma_{\beta}}\iprods{\nabla{F}(\zb^k) + \nabla^2{F}(\zb^k)(\bar{\ub}^k - \zb^k), \bar{\zb}^{k\!+\!1} - \bar{\ub}^k} \nonumber\\
& \leq  \frac{\sigma_{\beta}}{1 - \sigma_{\beta}}\Vert \nabla{F}(\zb^k) + \nabla^2{F}(\zb^k)(\bar{\ub}^k - \zb^k)\Vert_{\zb^k}^{\ast}\Vert\bar{\zb}^{k\!+\!1} - \bar{\ub}^k\Vert_{\zb^k}.
\end{align*}
where the last inequality follows from the elementary Cauchy-Schwarz inequality.
This inequality eventually leads to
\begin{equation*}
\begin{array}{ll}
\Vert \bar{\zb}^{k\!+\!1} - \bar{\ub}^k\Vert_{\zb^k} &\leq \frac{\sigma_{\beta}}{1 - \sigma_{\beta}}\Vert \nabla{F}(\zb^k) + \nabla^2{F}(\zb^k)(\bar{\ub}^k - \zb^k)\Vert_{\zb^k}^{\ast} \vspace{1ex}\\
&\leq  \frac{\sigma_{\beta}}{1 - \sigma_{\beta}}\left[ \Vert \nabla{F}(\zb^k)\Vert_{\zb^k}^{\ast} + \Vert \nabla^2{F}(\zb^k)(\bar{\ub}^k - \zb^k)\Vert_{\zb^k}^{\ast} \right] \vspace{1ex}\\
&\leq  \frac{\sigma_{\beta}}{1 - \sigma_{\beta}}\left[ \Vert \nabla{F}(\zb^k)\Vert_{\zb^k}^{\ast} + \Vert  \bar{\ub}^k - \zb^k \Vert_{\zb^k} \right].
\end{array}
\end{equation*}
Now, by the triangle inequality, we have $\Vert\bar{\zb}^{k\!+\!1} - \zb^k\Vert_{\zb^k} \leq \Vert\bar{\zb}^{k\!+\!1} - \bar{\ub}^k\Vert_{\zb^k} + \Vert\bar{\ub}^k - \zb^k\Vert_{\zb^k}$. This inequality is equivalent to $\lambda_{t_{k\!+\!1}}(\zb^k)  \leq \Vert\bar{\zb}^{k\!+\!1} - \bar{\ub}^k\Vert_{\zb^k} + \lambda_{t_k}(\zb^k)$ due to the definitions  $\lambda_{t_{k\!+\!1}}(\zb^k) = \Vert \bar{\zb}^{k\!+\!1} - \zb^k\Vert_{\zb^k}$ and  $\lambda_{t_k}(\zb^k) = \Vert \bar{\ub}^k - \zb^k\Vert_{\zb^k}$.
Using the last estimate \Shu{in} the above inequality we get
\begin{equation*}
\lambda_{t_{k\!+\!1}}(\zb^k) \leq \lambda_{t_k}(\zb^k) + \frac{\sigma_{\beta}}{1 - \sigma_{\beta}}\left[ \Vert \nabla{F}(\zb^k)\Vert_{\zb^k}^{\ast} + \lambda_{t_k}(\zb^k) \right],
\end{equation*}
which is  \eqref{eq:key_estimate2}.
The second inequality of \eqref{eq:key_estimate2} follows from the fact that \Shu{$\Vert \nabla{F}(\zb^k)\Vert_{\zb^k}^{\ast}\leq \sqrt{\nu}$}.

Let  us denote by $\gamma_k := \left(\frac{\sigma_{\beta}}{1-\sigma_{\beta}}\right)\left(\sqrt{\nu} + \lambda_{t_k}(\zb^k)\right)$.
For a given $\beta\in (0, 1)$, we now assume that $\lambda_{t_k}(\zb^k) \leq \beta$.
Then, by using \eqref{eq:key_estimate2} \Shu{in} \eqref{eq:key_estimate1} and the monotonic increase of its right-hand side with respect to $\lambda_{t_{k+1}}(\zb^k)$, we can derive
\begin{align*}
\lambda_{t_{k\!+\!1}}(\zb^{k\!+\!1}) &\leq \left(\frac{\lambda_{t_k}(\zb^k) + \vert\gamma_k\vert + \delta_k}{1 - \lambda_{t_k}(\zb^k) - \vert\gamma_k\vert - \delta_k}\right)^2 + \frac{\delta_k}{\left(1 - \lambda_{t_k}(\zb^k) - \vert\gamma_k\vert-\Shu{ \delta_k}\right)^3} \vspace{1ex}\nonumber\\
& \leq \left(\frac{\beta + \vert\gamma_k\vert + \delta_k}{1 - \beta - \vert\gamma_k\vert - \delta_k}\right)^2 + \frac{\delta_k}{(1 - \beta - \vert\gamma_k\vert - \delta_k)^3}\Shu{,}
\end{align*}
\Shu{as long as $\beta + \vert\gamma_k\vert + \delta_k <1$.}
Let us denote $\theta_k := \beta + \vert\gamma_k\vert$.
By using the identity $\frac{\beta + \vert\gamma_k\vert  + \delta_k}{1 - \beta - \vert\gamma_k\vert -\delta_k} = \frac{\beta + \vert\gamma_k\vert }{1 - \beta - \vert\gamma_k\vert } + \frac{\delta_k}{(1-\theta_k)(1-\theta_k-\delta_k)}$, we can  rewrite the last inequality as
\begin{align*}
\lambda_{t_{k\!+\!1}}(\zb^{k\!+\!1}) \leq \left(\!\frac{\theta_k}{1 \!-\! \theta_k }\!\right)^2 + \left[\frac{2\theta_k }{(1-\theta_k)^2(1\!-\!\theta_k \!-\!\delta_k)} + \frac{\delta_k}{(1-\theta_k)^2(1\!-\!\theta_k \!-\!\delta_k)^2} + \frac{1}{(1\!-\!\theta_k \!-\!\delta_k)^3}\right]\delta_k.
\end{align*}
If we choose $\delta_k$ such that $0\leq\delta_k \leq \theta_k(1-\theta_k)<1-\theta_k$, then the above inequality implies
\begin{align}\label{eq:lm32_est20b}
\lambda_{t_{k\!+\!1}}(\zb^{k\!+\!1}) \leq  \left(\frac{\theta_k}{1 \!-\! \theta_k }\right)^2  +  \left[\frac{2\theta_k(1\!-\!\theta_k)^2 + \theta_k(1 \!-\! \theta_k) \!+\! 1}{(1-\theta_k)^6}\right]\delta_k :=  \left(\frac{\theta_k}{1 - \theta_k }\right)^2 + M_k\delta_k.
\end{align}
\Shu{Take} any $c\in (0, 1)$, .e.g., $c := 0.95$, and \Shu{choose} $\delta_k$ such that $0 \leq \delta_k \leq \frac{(1-c^2)}{c^2M_k}\left(\frac{\theta_k}{1-\theta_k}\right)^2$.
Hence, in order to guarantee $\lambda_{t_{k\!+\!1}}(\zb^{k\!+\!1}) \leq \beta$, by using \eqref{eq:lm32_est20b}, we can impose the condition $\left(\frac{\theta_k}{1 - \theta_k }\right)^2 + M_k\delta_k \leq\frac{1}{c^2}\left(\frac{\theta_k}{1-\theta_k}\right)^2 \leq \beta$, which is equivalent to $\frac{\theta_k}{1 - \theta_k} \leq c\sqrt{\beta}$.
This condition leads to $\theta_k \Shu{\geq} \frac{c\sqrt{\beta}}{1+c\sqrt{\beta}} $, and therefore, $\vert\gamma_k\vert \leq \frac{c\sqrt{\beta}}{1+c\sqrt{\beta}} - \beta$.
Since $\vert\gamma_k\vert > 0$, we need to choose $\beta$ such that $0 < \beta < 0.5(1 + 2c^2 - \sqrt{1 +4c^2})$.

Next, by the choice of $\delta_k$, we require $0 \leq \delta_k \leq \min\set{ \frac{(1-c^2)}{c^2M_k}\left(\frac{\theta_k}{1-\theta_k}\right)^2, \theta_k(1-\theta_k)}$. Using the fact that $M_k = \frac{2\theta_k(1-\theta_k)^2 + \theta_k(1-\theta_k) + 1}{(1-\theta_k)^6}$ from \eqref{eq:lm32_est20b} and $0 \leq \theta_k \leq \frac{c\sqrt{\beta}}{1+c\sqrt{\beta}}$, we can show that the condition on $\delta_k$ holds if we choose
\begin{align*}
\delta_k \leq \bar{\delta} := \frac{(1-c^2)\beta}{(1+c\sqrt{\beta})^3\left[3c\sqrt{\beta} + c^2\beta + (1+c\sqrt{\beta})^3\right]}.
\end{align*}
On the other hand, we have  $\vert\gamma_k\vert = \left\vert   \left(\frac{\sigma_{\beta}}{1-\sigma_{\beta}}\right)\left(\sqrt{\nu} + \lambda_{t_k}(\zb^k)\right)\right\vert \leq \left(\frac{\sigma_{\beta}}{1-\sigma_{\beta}}\right)\left(\sqrt{\nu} + \beta\right)$.
In order to guarantee that $\vert\gamma_k\vert \leq \frac{c\sqrt{\beta}}{1+c\sqrt{\beta}} - \beta$, we use the above estimate to impose a condition $\left(\frac{\sigma_{\beta}}{1-\sigma_{\beta}}\right) \leq \frac{1}{\sqrt{\nu} + \beta}\left(\frac{c\sqrt{\beta}}{1 + c\sqrt{\beta}} - \beta\right)$, which leads to
\begin{equation*}
\sigma_{\beta} \leq \bar{\sigma}_{\beta} := \frac{c\sqrt{\beta} - \beta(1 + c\sqrt{\beta})}{(1+c\sqrt{\beta})\sqrt{\nu} + c\sqrt{\beta}}.
\end{equation*}
This estimate is exactly the right-hand side of \eqref{eq:choice_of_sigma}.
Finally, using \eqref{eq:key_estimate2} and the definition of $\gamma_k$, we can easily show that $\lambda_{t_{k\!+\!1}}(\zb^k) \leq \lambda_{t_k}(\zb^k) + \abs{\gamma_k} \leq \beta + \abs{\gamma_k} \equiv \theta_k \leq \frac{c\sqrt{\beta}}{1+c\sqrt{\beta}}$.
\Eproof

\beforesubsec
\subsection{\bf The proof of Theorem \ref{th:ppf_convergence}: The worst-case iteration-complexity of \ref{eq:inexact_pf_scheme}}\label{apdx:th:ppf_convergence}
\aftersubsec
By Lemma~\ref{le:sol_measure} and $\lambda_{t_{k+1}}(\zb^k) \leq \frac{c\sqrt{\beta}}{1 + c\sqrt{\beta}}$, we can see that $\zb^k$ is an $\varepsilon$-solution of \eqref{eq:mono_inclusion} if $t_k := M_0^{-1}\varepsilon$, where $M_0 := \left(1 - \frac{c\sqrt{\beta}}{1 + c\sqrt{\beta}}\right)^{-1}\left(\sqrt{\nu} + \frac{c\sqrt{\beta}}{1 + c\sqrt{\beta}} + 2\bar{\delta}_t(\beta)\right)  = \mathcal{O}(\sqrt{\nu})$.

On the other hand, by induction, it follows from the update rule $t_{k\!+\!1} = (1-\sigma_{\beta})t_k$ of \ref{eq:inexact_pf_scheme}  that $t_k = (1-\sigma_{\beta})^kt_0$.
Hence, $\zb^k$ is an $\varepsilon$-solution of \eqref{eq:mono_inclusion} if  we have $t_k = (1-\sigma_{\beta})^k t_0 \leq \frac{\varepsilon}{M_0}$.
This condition leads to $k\ln(1-\sigma_{\beta}) \geq \ln\left(\frac{\varepsilon}{M_0t_0}\right)$, which implies  $k  \leq \frac{\ln(\varepsilon/(M_0t_0))}{\ln(1-\sigma_{\beta})}$.
Using an elementary inequality $\ln(1-\sigma_{\beta}) \leq -\sigma_{\beta}$, we can upper bound $k$ as
\begin{equation*}
k \geq \frac{1}{\bar{\sigma}_{\beta}}\ln\left(\frac{M_0 t_0}{\varepsilon}\right) = \frac{\left((1+c\sqrt{\beta})\sqrt{\nu} + c\sqrt{\beta}\right)}{c\sqrt{\beta} - \beta(1 + c\sqrt{\beta})}\ln\left(\frac{M_0t_0}{\varepsilon}\right).
\end{equation*}
Consequently, the worst-case iteration-complexity of \ref{eq:inexact_pf_scheme} is $\mathcal{O}\left(\sqrt{\nu}\ln\left(\frac{\sqrt{\nu} t_0}{\varepsilon}\right)\right)$.
\Eproof

\beforesubsec
\subsection{\bf The proof of Theorem \ref{th:complexity_of_phase_1}: Finding an initial point for \ref{eq:inexact_pf_scheme}}\label{apdx:th:complexity_of_phase_1}
\aftersubsec
From \eqref{eq:inexact_pf_scheme_phase1}, if we define $\nabla{\hat{F}}(\hat{\zb}^j) := \nabla{F}(\hat{\zb}^k) - t_0^{-1}\tau_{k+1}\zeta_0$, we still have $\nabla^2{\hat{F}}(\hat{\zb}^j) = \nabla^2{F}(\hat{\zb}^j)$.
Hence, the estimate \eqref{eq:key_estimate1} still holds for $\hat{\lambda}_{\tau}(\hat{\zb}^j)$.

Next, if we define $\bar{\vb}^j :=  \Pc_{\hat{\zb}^j}\left(\hat{\zb}^j - \nabla^2{F}(\hat{\zb}^j)^{-1}\left(\nabla{F}(\hat{\zb}^j)- \tau_jt_0^{-1}\hat{\zeta}^0\right); t_0\right)$, then, by the definition of $\Pc_{\hat{\zb}^j}$, we have
\begin{equation}\label{eq:thm5_est1}
-t_0\left[\nabla^2{F}(\hat{\zb}^j)(\bar{\vb}^j - \hat{\zb}^j) + \nabla{F}(\hat{\zb}^j) - \tau_jt_0^{-1}\hat{\zeta}_0\right] \in \Ac(\bar{\vb}^j).
\end{equation}
Similarly, since $\bar{\hat{\zb}}^{j+1} := \Pc_{\hat{\zb}^j}\left(\hat{\zb}^j - \nabla^2{F}(\hat{\zb}^j)^{-1}\left(\nabla{F}(\hat{\zb}^j)- \tau_{j+1}t_0^{-1}\hat{\zeta}^0\right); t_0\right)$, we have
\begin{equation}\label{eq:thm5_est2}
-t_0\left[\nabla^2{F}(\hat{\zb}^j)(\bar{\hat{\zb}}^{j+1} - \hat{\zb}^j) + \nabla{F}(\hat{\zb}^j) - \tau_{j+1}t_0^{-1}\hat{\zeta}_0\right] \in \Ac(\bar{\hat{\zb}}^{j+1}).
\end{equation}
Using \eqref{eq:thm5_est1}, \eqref{eq:thm5_est2} and the monotonicity of $\Ac$, we have
\begin{equation*}
t_0\iprods{\nabla^2{F}(\hat{\zb}^j)(\bar{\hat{\zb}}^{j+1}  - \bar{\vb}^j), \bar{\hat{\zb}}^{j+1}  - \bar{\vb}^j} \leq (\tau_j - \tau_{j+1})\iprods{\hat{\zeta}_0, \bar{\vb}^j - \bar{\hat{\zb}}^{j+1}}.
\end{equation*}
Using $\tau_{j+1} := \tau_j - \Delta_j$ and the Cauchy-Schwarz inequality, the last inequality leads to
\begin{equation}\label{eq:thm5_est4}
t_0\norm{\bar{\hat{\zb}}^{j+1}  - \bar{\vb}^j}_{\hat{\zb}^j} \leq \Delta_j\Vert\hat{\zeta}_0\Vert_{\hat{\zb}^j}^{\ast}.
\end{equation}
Now, similar to the proof of Lemma~\ref{le:update_penalty_param}, using \eqref{eq:thm5_est4}, we can derive
\begin{equation}\label{eq:thm5_est6}
\hat{\lambda}_{\tau_{j+1}}(\hat{\zb}^{j}) \leq \hat{\lambda}_{\tau_{j}}(\hat{\zb}^j) + \frac{\Delta_j}{t_0}\Vert\hat{\zeta}_0\Vert_{\hat{\zb}^j}^{\ast}.
\end{equation}
By the same argument as the proof of \eqref{eq:choice_of_sigma}, we can show that with $\hat{\gamma}_k := \frac{\Delta_j}{t_0}\Vert\hat{\zeta}_0\Vert_{\hat{\zb}^j}^{\ast}$, we have $\abs{\hat{\gamma}_k} \leq \frac{c\sqrt{\eta}}{1+c\sqrt{\eta}} - \eta$.
This shows that $\Delta_j \leq \frac{t_0}{\Vert\hat{\zeta}_0\Vert_{\hat{\zb}^j}^{\ast}}\left(\frac{c\sqrt{\eta}}{1+c\sqrt{\eta}} - \eta\right)$, which is the first estimate of \eqref{eq:choice_of_sigma0}.
The second estimate of \eqref{eq:choice_of_sigma0} can be derived as in Lemma~\ref{le:update_penalty_param} using $\eta$ instead of $\beta$.

We prove \eqref{eq:diff_nt_decrement}. From \eqref{eq:nt_decrement} and \eqref{eq:nt_decrement0}, using the triangle inequality, we can upper bound
\begin{align*}
\lambda_{t_0}(\zb^0) &:= \big\Vert \zb^0 - \Pc_{\zb^0}\big(\zb^0 - \nabla^2F(\zb^0)^{-1}\nabla{F}(\zb^0); t_0\big) \big\Vert_{\zb^0} \nonumber\\
&\overset{\tiny{\zb^0 := \hat{\zb}^j}}{=} \big\Vert \hat{\zb}^j - \Pc_{\hat{\zb}^j}\big(\hat{\zb}^j - \nabla^2F(\hat{\zb}^j)^{-1}\nabla{F}(\hat{\zb}^j); t_0\big) \big\Vert_{\hat{\zb}^j} \nonumber\\
&\leq  \Big\Vert \hat{\zb}^j \!-\! \Pc_{\hat{\zb}^j}\big(\hat{\zb}^j \!-\! \nabla^2F(\hat{\zb}^j)^{-1}\big(\nabla{F}(\hat{\zb}^j) \!-\! \tau_jt_0^{-1}\hat{\zeta}^0\big); t_0\big)\Big\Vert_{\hat{\zb}^j} \nonumber\\
&+ \Big\Vert \Pc_{\hat{\zb}^j}\big(\hat{\zb}^j - \nabla^2F(\hat{\zb}^j)^{-1}\nabla{F}(\hat{\zb}^j); t_0\big) -  \Pc_{\hat{\zb}^j}\big(\hat{\zb}^j \!-\! \nabla^2F(\hat{\zb}^j)^{-1}\big(\nabla{F}(\hat{\zb}^j)  \!-\! \tau_jt_0^{-1}\hat{\zeta}^0\big); t_0\big) \Big\Vert_{\hat{\zb}^j} \nonumber\\
& \overset{\tiny\eqref{eq:nt_decrement0},\eqref{eq:Pc_oper_property}} \leq \hat{\lambda}_{\tau_j}(\hat{\zb}^j) + \big\Vert t_0^{-1}\tau_j\nabla^2{F}(\hat{\zb}^j)^{-1}\hat{\zeta}^j\big\Vert_{\hat{\zb}^j} \nonumber\\
&= \hat{\lambda}_{\tau_j}(\hat{\zb}^j) + \tau_jt_0^{-1}\Vert\hat{\zeta}^0\Vert_{\hat{\zb}^j}^{\ast},
\end{align*}
which proves the first inequality of \eqref{eq:diff_nt_decrement}.

By \cite[Corollary 4.2.1]{Nesterov2004}, we have $\Vert\hat{\zeta}^0\Vert_{\hat{\zb}^j}^{\ast} \leq \kappa\Vert\hat{\zeta}^0\Vert_{\bar{\zb}_F^{\star}}^{\ast}$, where $\bar{\xb}_F^{\star}$ and $\kappa$ are given by \eqref{eq:analytical_center} and below \eqref{eq:analytical_center}, respectively.
Hence, $\bar{\Delta}_{\eta} := \frac{\mu_{\eta}}{\kappa\Vert\hat{\zeta}^0\Vert_{\bar{\zb}_F^{\star}}^{\ast}} \leq \bar{\Delta}_{j}$.
The second estimate of \eqref{eq:diff_nt_decrement} follows from $\tau_j := \tau - \sum_{l=0}^{j-1}\Delta_j \leq 1 - j\bar{\Delta}_{\eta}$ due to the update rule \eqref{eq:inexact_pf_scheme_phase1} with $\Delta_j := \bar{\Delta}_{j} \geq\bar{\Delta}_{\eta}$.
In order to guarantee $\lambda_{t_0}(\zb^0) \leq \beta$, it follows from \eqref{eq:diff_nt_decrement} and the update rule of $\tau_j$ that
\begin{equation*}
j \geq \frac{1}{\bar{\Delta}_{\eta}}\left(1 - \frac{(\beta - \eta)t_0}{\kappa\Vert\hat{\zeta}^0\Vert_{\bar{\zb}_F^{\star}}^{\ast}}\right).
\end{equation*}
Finally, substituting $\bar{\Delta}_{\eta} =  \frac{t_0}{\kappa\Vert\hat{\zeta}_0\Vert_{\bar{\zb}^{\star}_F}^{\ast}}\left(\frac{c\sqrt{\eta}}{1+c\sqrt{\eta}} - \eta\right)$ into this estimate and after simplifying the result, we obtain the remaining conclusion of Theorem   \ref{th:complexity_of_phase_1}.
\Eproof

\beforesubsec
\subsection{\bf The proof of  Theorem \ref{th:primal_recovery}: Primal recovery for \eqref{eq:constr_cvx2} in Algorithm~\ref{alg:A1c}}\label{apdx:th:primal_recovery}
\aftersubsec
By the definition of $\varphi$, we have $\varphi(\yb) := f^{\ast}(\cb - L^{\ast}\yb) = f^{\ast}(t^{-1}(\cb - L^{\ast}\yb)) - \nu\ln(t)$ due to the self-concordant logarithmic homogeneity of $f$.
Using the property of the Legendre transformation $f^{\ast}$ of $f$, we can express this function as
\begin{equation*}
\varphi(\yb) = t^{-1}\max_{\xb\in\intx{\Kc}}\set{\iprods{\cb - L^{\ast}\yb, \xb} - tf(\xb) }  - \nu\ln(t).
\end{equation*}
We show the point $\xb^k$ given by \eqref{eq:primal_sol} solves the above maximization problem.
We can write down the optimality condition of the above maximization problem as
\begin{equation*}
\cb - L^{\ast}\yb^{k\!+\!1} - t_{k\!+\!1}\nabla{f}(\xb^{k\!+\!1}) = 0,
\end{equation*}
which leads to $\nabla{f}(\xb^{k\!+\!1}) = t_{k\!+\!1}^{-1}(\cb - L^{*}\yb^{k\!+\!1})$.
On the other hand, by the well-known property of $f$ \cite{Nesterov2004}, we have $\xb^{k\!+\!1} = \nabla{f^{*}}(\nabla{f}(\xb^{k\!+\!1})) = \nabla{f^{*}}\left( t_{k\!+\!1}^{-1}(\cb - L^{*}\yb^{k\!+\!1})\right) \in \intx{\Kc}$.

Now, we prove \eqref{eq:recovery}.
We note that $\cb - L^{*}\yb^{k\!+\!1} - t_{k\!+\!1}\nabla{f}(\xb^{k\!+\!1}) = 0$ and $\Vert\nabla{f}(\xb)\Vert_{\xb}^{\ast}\leq\sqrt{\nu}$, which leads to
\begin{equation*}
\Vert L^{*}\yb^{k\!+\!1} -\cb \Vert_{\xb^{k\!+\!1}}^{*} = t_{k\!+\!1}\Vert \nabla{f}(\xb^{k\!+\!1})\Vert_{\xb^{k\!+\!1}}^{*} \leq t_{k\!+\!1}\Shu{\sqrt{\nu}}.
\end{equation*}
Since $t_{k\!+\!1} \leq \varepsilon$, this estimate leads to the first inequality of \eqref{eq:recovery}.

From \eqref{eq:approx_sol}, there exists $\eb^k\in\R^p$ such that $\eb^k \in \nabla{\varphi}(\yb^k) + \nabla^2{\varphi}(\yb^k)(\yb^{k\!+\!1} - \yb^k) + t_{k\!+\!1}^{-1}\partial{\psi}(\yb^{k\!+\!1})$ and $\Vert \eb^k \Vert_{\yb^{k}}^{*} \leq \delta_k$.
This condition leads to
\begin{equation*}
\eb^k + \nabla{\varphi}(\yb^{k\!+\!1}) - \nabla{\varphi}(\yb^k) - \nabla^2{\varphi}(\yb^k)(\yb^{k\!+\!1} - \yb^k) \in \nabla{\varphi}(\yb^{k\!+\!1}) +  t_{k\!+\!1}^{-1}\partial{\psi}(\yb^{k\!+\!1}).
\end{equation*}
This expression leads to
\begin{align}\label{eq:pro23_proof1}
\mathrm{dist}_{\yb^{k\!+\!1}}\Big(0, &\nabla{\varphi}(\yb^{k\!+\!1}) +  t_{k\!+\!1}^{-1}\partial{\psi}(\yb^{k\!+\!1})\Big)  \leq \Vert \eb^k + \nabla{\varphi}(\yb^{k\!+\!1}) - \nabla{\varphi}(\yb^k) - \nabla^2{\varphi}(\yb^k)(\yb^{k\!+\!1} - \yb^k)\Vert_{\yb^{k\!+\!1}}^{*} \nonumber\\
&\leq \Vert \eb^k\Vert_{\yb^{k\!+\!1}}^{*} + \Vert \nabla{\varphi}(\yb^{k\!+\!1}) - \nabla{\varphi}(\yb^k) - \nabla^2{\varphi}(\yb^k)(\yb^{k\!+\!1} - \yb^k)\Vert_{\yb^{k\!+\!1}}^{*}.
\end{align}
To estimate the right-hand side of this inequality, we define $M_k := \Vert \nabla{\varphi}(\yb^{k\!+\!1}) - \nabla{\varphi}(\yb^k) - \nabla^2{\varphi}(\yb^k)(\yb^{k\!+\!1} - \yb^k)\Vert_{\yb^{k\!+\!1}}^{*}$.
With the same proof as \cite[Theorem 4.1.14]{Nesterov2004}, we can show that
\begin{equation}\label{eq:pro23_proof2a}
M_k \leq \left(1 - \Vert \yb^{k\!+\!1} - \yb^k\Vert_{\yb^k}\right)^{-2}\Vert \yb^{k\!+\!1} - \yb^k\Vert_{\yb^k}^2 \leq \frac{\left(\delta(\yb^k) + \lambda_{t_{k\!+\!1}}(\yb^k)\right)^2}{\left(1- \lambda_{t_{k\!+\!1}}(\yb^k) - \delta(\yb^k)\right)^2}.
\end{equation}
Here, we use $\Vert \yb^{k\!+\!1} - \yb^k\Vert_{\yb^k} \leq \Vert \yb^{k\!+\!1} - \bar{\yb}^{k\!+\!1}\Vert_{\yb^k} + \Vert \bar{\yb}^{k\!+\!1} - \yb^k\Vert_{\yb^k} = \delta(\yb^k) + \lambda_{t_{k\!+\!1}}(\yb^k)$ by the \Shu{definitions} of $\lambda_{t_{+}}(\yb)$ in \eqref{eq:nt_decrement} and of $\delta(\yb)$ \Shu{above} \eqref{eq:key_estimate1}.
Substituting \eqref{eq:pro23_proof2a} into \eqref{eq:pro23_proof1} to get
\begin{equation}\label{eq:pro23_proof2}
\mathrm{dist}_{\yb^{k\!+\!1}}\left(0, \nabla{\varphi}(\yb^{k\!+\!1}) +  t_{k\!+\!1}^{-1}\partial{\psi}(\yb^{k\!+\!1})\right) \leq \Vert \eb^k\Vert_{\yb^{k\!+\!1}}^{*} + \frac{\left(\delta(\yb^k) + \lambda_{t_{k\!+\!1}}(\yb^k)\right)^2}{\left(1- \lambda_{t_{k\!+\!1}}(\yb^k) - \delta(\yb^k)\right)^2}.
\end{equation}
Next, it remains to estimate $\Vert \eb^k\Vert_{\yb^{k\!+\!1}}^{*}$. Indeed, we have
\begin{equation*}
\begin{array}{ll}
\Vert \eb^k\Vert_{\yb^{k\!+\!1}}^{*} &\leq\! \big(1 - \Vert \yb^{k\!+\!1} - \yb^{k}_{t_k}\Vert_{\yb^k}\big)^{-1}\Vert \eb^k\Vert_{\yb^k} \leq \left(1- \lambda_{t_{k\!+\!1}}(\yb^k) - \delta(\yb^k)\right)^{-1}\Vert \eb^k\Vert_{\yb^k} \vspace{0.75ex}\\
&\leq \frac{\delta_k}{1-\Shu{\lambda_{t_{k\!+\!1}}(\yb^k)}-\delta_k}.
\end{array}
\end{equation*}
Using this estimate into \eqref{eq:pro23_proof2} and $\lambda_{t_{k\!+\!1}}(\yb^k) \leq c\sqrt{\beta}(1+c\sqrt{\beta})^{-1}$ from Lemma \ref{le:update_penalty_param}, we obtain
\begin{equation*}
\mathrm{dist}_{\yb^{k\!+\!1}}\left(0, \nabla{\varphi}(\yb^{k\!+\!1}) +  t_{k\!+\!1}^{-1}\partial{\psi}(\yb^{k\!+\!1})\right) \leq \frac{\delta_k(1+c\sqrt{\beta})}{(1-\delta_k(1+c\sqrt{\beta}))} + \frac{(\delta_k(1+c\sqrt{\beta})+c\sqrt{\beta})^2}{(1 -\delta_k(1+c\sqrt{\beta}))^2}.
\end{equation*}
Substituting an upper bound $\delta_t := \frac{(1-c^2)\beta}{(1+c\sqrt{\beta})^3\left[3c\sqrt{\beta} + c^2\beta + (1+c\sqrt{\beta})^3\right]}.$ of $\delta_k$ from Lemma \ref{le:update_penalty_param} into the last estimate and simplifying the result, we get
\begin{equation}\label{eq:pro23_proof3}
{\!\!\!\!}\mathrm{dist}_{\yb^{k\!+\!1}}\left(0, \nabla{\varphi}(\yb^{k\!+\!1}) +  t_{k\!+\!1}^{-1}\partial{\psi}(\yb^{k\!+\!1})\right) \leq
\theta(c,\beta),{\!\!\!}
\end{equation}\label{eq:pro23_proof3b}
where $\theta(c,\beta)$ is defined as
\begin{equation}
\begin{array}{ll}
\theta(c,\beta) &:= \frac{(1-c^2)\beta}{(1+c\sqrt{\beta})^2\left[3c\sqrt{\beta} + c^2\beta + (1+c\sqrt{\beta})^3\right]-(1-c^2)\beta}
\vspace{1ex}\\
&+ \left(\frac{(1-c^2)\beta + c\sqrt{\beta}(1+c\sqrt{\beta})^2\left[3c\sqrt{\beta} + c^2\beta + (1+c\sqrt{\beta})^3\right]}{(1+c\sqrt{\beta})^2\left[3c\sqrt{\beta} + c^2\beta + (1+c\sqrt{\beta})^3\right]-(1-c^2)\beta}\right)^2.
\end{array}
\end{equation}
Using the fact that $c \in (0, 1)$ and $0 \leq \beta < 0.5(1 + 2c^2 - \sqrt{1 + 4c^2})$, we have $\theta(c,\beta) \leq 1$.
Since $\nabla{\varphi}(\cdot) = \Shu{-L\nabla{f^{*}}( \cb-L^{*}(\cdot) ) = -t_{k\!+\!1}^{-1}L\nabla{f^{*}}(t_{k\!+\!1}^{-1}(\cb-L^{*}(\cdot)))}$ due to \eqref{eq:smooth_dual_term}\Shu{,
using} \eqref{eq:primal_sol} we can show that $\nabla{\varphi}(\yb^{k\!+\!1}) = t_{k\!+\!1}^{-1}L\xb^{k\!+\!1}$.
Plugging this expression into \eqref{eq:pro23_proof3} and noting that $\partial{\psi}(\cdot) = \partial{g}^{*}(\cdot) + \bb$, we obtain
\begin{equation*}
\mathrm{dist}_{\yb^{k\!+\!1}}\left(L\xb^{k\!+\!1} - \bb, \partial{g^{*}}(\yb^{k\!+\!1})\right) = \mathrm{dist}_{\yb^{k\!+\!1}}\left(0, \bb - L\xb^{k\!+\!1} + \partial{g^{*}}(\yb^{k\!+\!1})\right)  \leq t_{k\!+\!1}\theta(c,\beta).
\end{equation*}
Let $\sb^{k+1} = \pi_{\partial{g^{\ast}}(\yb^{k+1})}(\Shu{L\xb^{k\!+\!1} - \bb})$ the projection of \Shu{$L\xb^{k\!+\!1} - \bb$} onto $\partial{g^{\ast}}(\yb^{k+1})$. Then,  $\sb^{k+1} \in \partial{g^{\ast}}(\yb^{k+1})$, and hence, $\yb^{k+1} \in\partial{g}(\sb^{k+1})$, which shows the second term of \eqref{eq:recovery}.
Using this relation in the last inequality and the definition of $\sb^{k+1}$, we obtain $\Vert L\xb^{k\!+\!1} - \bb - \sb^{k+1}\Vert_{\yb^{k+1}}^{\ast} \leq t_{k\!+\!1}\theta(c,\beta)$, which is the third term of \eqref{eq:recovery}.
Finally, since $\theta(c,\beta) \leq 1$, we have $\max\set{\sqrt{\nu}, \theta(c,\beta)} = \sqrt{\nu}$. Using \eqref{eq:recovery}, we can conclude that $(\xb^k, \sb^k)$ is an $\varepsilon$-solution of  \eqref{eq:constr_cvx} if $\sqrt{\nu}t_k \leq \varepsilon$.
\Eproof

\vspace{-3ex}
\bibliographystyle{plain}


\end{document}